\documentclass[a4paper,11pt]{amsart}

\usepackage{amsfonts}
\usepackage{amssymb}
\usepackage{graphics}

\newcommand{\C}{\mathbb{C}}
\newcommand{\N}{\mathbb{N}}
\newcommand{\Z}{\mathbb{Z}}
\newcommand{\R}{\mathbb{R}}

\newcommand{\calO}{{\mathcal{O}}}

\newcommand{\Diff}{{{\rm Diff}\, ({\mathbb C^n}, 0)}}
\newcommand{\ooY}{{\overline{Y}}}

\newcommand{\diff}{{{\rm Diff}\, ({\mathbb C}, 0)}}
\newcommand{\diffn}{{{\rm Diff}\, ({\mathbb C^n}, 0)}}

\newcommand{\ghatXone}{{\widehat{\mathfrak{X}}}}
\newcommand{\ghatX}{{\widehat{\mathfrak{X}}_2}}
\newcommand{\formdiffn}{{\widehat{\rm Diff}_{1} (\mathbb{C}^2, 0)}}
\newcommand{\formalC}{{\mathbb{C} [[x,y]]}}
\newcommand{\fieldC}{{\mathbb{C} ((x,y))}}
\newcommand{\diffCtwo}{{{\rm Diff}_1 ({\mathbb C}^2, 0)}}
\newcommand{\formdiffgeneral}{{\widehat{\rm Diff} (\mathbb{C}^2, 0)}}

\newcommand{\Diffgentwo}{{{\rm Diff}\, ({\mathbb C}^2, 0)}}

\def\picill#1by#2(#3)#4
{\vbox to #2
{\hrule width #1 height 0pt depth 0pt
\vfill\special{illustration #3 scaled #4}}}

\newtheorem{theorem}{Theorem}[section]
\newtheorem{prop}[theorem]{Proposition}
\newtheorem{corol}[theorem]{Corollary}
\newtheorem{lemma}[theorem]{Lemma}
\newtheorem{obs}[theorem]{Remark}
\newtheorem{example}[theorem]{Example}

\theoremstyle{definition}
\newtheorem{defi}[theorem]{Definition}
\newtheorem{ex}{Example}

\theoremstyle{remark}

\begin{document}

\title[Topological dynamics in $\Diffgentwo$]{Discrete orbits, recurrence and solvable subgroups of $\Diffgentwo$}

\author{Julio C. Rebelo \, \, \, \& \, \, \, Helena Reis}
\address{}
\thanks{}

\begin{abstract}
We discuss the local dynamics of a subgroup of $\Diffgentwo$ possessing locally discrete orbits as well as
the structure of the recurrent set for more general groups. It is proved,
in particular, that a subgroup of $\Diffgentwo$ possessing locally discrete orbits must be virtually solvable.
These results are of considerable interest in problems concerning integrable systems.
\end{abstract}

\maketitle

\section{Introduction}

This paper is devoted to establishing some general theorems about the dynamics of (virtually) {\it non solvable}\, subgroups of
$\Diffgentwo$. Whereas motivations for these results arise from a few different sources, problems concerned with integrable
systems and with Morales-Ramis-Sim\'o {\it differentiable Galois theory}\, are very
directly related to our main results, \cite{MoralesRamis}, \cite{M-Ramis-S}. In this introduction, we shall first state
our main results and then proceed to a general
discussion about their motivations and applications in perspective with some previous results.

Throughout this paper, a group will be said to be {\it virtually solvable}\, if it contains a normal, solvable subgroup of finite index.
Now, consider finitely many local diffeomorphisms $f_1, \ldots ,f_k$
inducing elements of $\Diffgentwo$. Denote by $G_U$ the pseudogroup generated by $f_1, \ldots ,f_k$ on some chosen neighborhood
$U$ of $(0,0) \in \C^2$; see Section~2.1 for details. At the level of germs, the subgroup of $\Diffgentwo$ generated by
$f_1, \ldots ,f_k$ is going to be denoted by $G$. When no misunderstood is possible, we shall allow ourselves
to identify $G_U$ and $G$. With this identification,
$G$ is said to have locally discrete orbits (resp. finite orbits), if there is a sufficiently small neighborhood
$U$ of $(0,0)$ where $G_U$ has locally
discrete orbits (resp. finite orbits). The reader is referred to Section~2.1 for accurate definitions. With this terminology,
our first main result reads:

\vspace{0.2cm}

\noindent {\bf Theorem~A}. {\sl Suppose that $G$ is a finitely generated subgroup of $\Diffgentwo$ with
locally discrete orbits. Then $G$ is virtually solvable.}

\vspace{0.2cm}

\vspace{0.2cm}

\noindent {\bf Remark}. Although we always work with finitely generated groups, the reader will note that
Theorem~A also holds for groups that are infinitely generated. A simple argument to derive
this slightly stronger statement from the proof of Theorem~A is provided at the end of Section~3,
see Theorem~\ref{infinitelygeneratedgroups}.

\vspace{0.2cm}

The notion of {\it recurrent points}\, allows us to accurately state Theorem~A. Given $U$ and $G_U$ as above, a point
$p \in U$ is said to be {\it recurrent}\, if there exists a sequence $\{ g_n\}$ of elements in $G_U$
such that $g_n (p) \rightarrow p$ with $g_n (p) \neq p$ for every~$n$. In this definition,
it is implicitly assumed that $p$ belongs to the domain of definition
of $g_n$ when $g_n$ is viewed as an element of the pseudogroup $G_U$. A recurrent point does not have locally finite orbit and,
conversely, a point whose orbit is not locally finite must be recurrent. Thus, Theorem~A can be rephrased by saying that
there are always recurrent points for a non-virtually solvable group $G \subset \Diffgentwo$. The size of the set formed by these
recurrent points may, however, be relatively small as it may coincide with a Cantor set (this is very much similar to the case
of a Kleinian group having a Cantor set as its limit set, cf. Section~\ref{section.examples}). To obtain a general result about
the size of recurrent points, we are led to consider the
normal subgroup $\diffCtwo$ of $\Diffgentwo$ consisting of those local diffeomorphisms tangent to the identity.
When $G$ happens to be a (pseudo-) subgroup of $\diffCtwo$, the following stronger result holds:

\vspace{0.2cm}

\noindent {\bf Theorem~B}. {\sl Consider a non-solvable group $G \subset \diffCtwo$ and denote by $\Omega \,(G)$ the set
of points that fail to be recurrent for~$G$. Then there is a neighborhood
$U$ of $(0,0) \in \C^2$ such that $\Omega (G) \cap U$ is contained in a countable union of proper analytic subsets of~$U$
(in particular $\Omega (G) \cap U$ has null Lebesgue measure).}

\vspace{0.1cm}

\noindent {\bf Remark}. In the above statement the reader will note that the group of germs at $(0,0) \in \C^2$
naturally associated to $G$ is only assumed to be
non solvable as opposed to non virtually solvable. Also it is easy to prove that for a group $G$ generated
by a random choice of~$n \geq 2$ elements in $\diffCtwo$, the resulting set $\Omega \, (G)$ is reduced to the origin
of $\C^2$, cf. Remark~\ref{ongenericgroups}.

On a different note, we know of no example of non-solvable group $G \subset \diffCtwo$ for which $\Omega (G)$
is not contained in a {\it proper analytic set}. It would be nice to know whether this stronger statement always holds.

\vspace{0.2cm}

Concerning the above theorems, it may be observed that suitable versions of them are likely to
hold in arbitrary dimensions although we have not tried to work out any of these generalizations. Indeed, we decided
to restrict our attention to the $2$-dimensional case
partly because this setting is already full of new phenomena
and partly because the corresponding proofs are already fairly involved. Yet, a careful reading of our
arguments indicates that more typical arguments of complex dimension two were used only at a few points which, in turn,
suggests the existence of suitable arbitrary dimensional versions of the mentioned results.

We can now go back to the beginning of this introduction and discuss the motivations for the above statements.
The most important motivations can be ascribed to several types of {\it Galois theories}\,
and to integrability problems, see below. However,
we may begin by observing that very little of general is known about the dynamics of {\it large}\, (e.g. non-solvable)
subgroups of $\diffn$ when $n\geq 2$. In this sense,
the above results stand among the first ones in this direction. The situation contrasts
with the case of the local dynamics associated with subgroups of $\diff$ and
a brief review of the main results in this case may be a good starting point for us.
Whereas the local dynamics of subgroups of $\diff$
still holds some subtle open problems, the topic can be regarded
as well understood since a large body of knowledge on these dynamics can be found in the literature,
see \cite{shcherbakov}, \cite{nakai}, \cite{Loraybook}, \cite{russians}, \cite{Yoccozpetitsdiviseurs}.
The picture changes drastically when $n\geq 2$ as many new phenomena
emerge to provide a far more involved landscape. Indeed, when $n\geq 2$, there is a significant body of theory developed
in the case of the dynamics associated with a {\it parabolic germ}, cf. \cite{ecalle}, \cite{hakim}, \cite{abate},
\cite{raissy}. For non-solvable groups, the results of \cite{lorayandI} provide satisfactory
answers for {\it non-discrete groups containing a hyperbolic contraction}. These conditions, however, are not always
satisfied in the cases of interest.

Along the lines of the above paragraph, a first motivation for this work can broadly be described as the beginning
of a systematic study of the dynamics associated to ``large'' subgroups of $\Diffgentwo$, where by
``large'' we typically mean {\it non solvable}\, (and in some cases non virtually solvable). Naturally,
when considering these groups, we might be tempted to parallel the theory of Shcherbakov-Nakai
vector fields applicable to non solvable
subgroups of $\diff$. Although their theory remains an important guiding principle for our investigations,
the very existence of {\it free discrete}\, subgroups of $\Diffgentwo$ is enough to ensure that Shcherbakov-Nakai vector fields
cannot be associated with subgroups of $\Diffgentwo$ without additional assumptions; see Section~\ref{section.examples}
for details and definitions. In this direction, whereas our recurrence
statements constitute a less powerful tool than vector fields {\it approximating the dynamics of the group}, they
have the advantage of holding for arbitrary non virtually solvable subgroups of $\Diffgentwo$ and, in fact, they
constitute the first general result concerning the dynamics of these groups. Moreover, as far as general
non-virtually-solvable subgroups of $\Diffgentwo$ are concerned, Theorem~A is probably not far from sharp.
Also it is worth mentioning that in a number of standard applications of
Shcherbakov-Nakai theory, only the recurrent character of the dynamics is needed so that the Theorem~A
suffices to derive important conclusions. As an outstanding example of these situations, we quote the work of Camacho
and Scardua on the ``Analytic limit set problem'', see \cite{camachoKyoto}, \cite{camachoasterisque}: the remarkable
conclusion that the holonomy group of the limit set in question must be solvable requires only the fact that the dynamics
of a non solvable subgroup of $\diff$ has recurrent points. Thus, Theorem~A is strong enough to yield the analogous conclusion
for suitable higher dimensional versions of the problem in question.

The second and more important motivation for the previously stated results, however, comes from a few fundamental questions
concerning the {\it integrable character}\, of certain systems (vector fields). Most of this goes along the connection between
integrable systems and Galois differential theories in the spirit of \cite{MoralesRamis}, \cite{M-Ramis-S}. Yet, our first motivation
stemming from {\it integrable systems}\,
can be traced back to a classical theorem due to Mattei and Moussu \cite{M-M} asserting that,
in dimension~$2$, the existence of holomorphic first integrals for (local) holomorphic foliations
can be read off the topological dynamics associated to the singular point. It was recently shown in \cite{thesis} that,
strictly speaking, this remarkable phenomenon no longer holds in higher dimensions and some additional curious examples were provided in
\cite{HelenaAndI}. These examples made it clear that a fundamental question in this problem is to decide which, if any,
kind of ``integrable character'' can be associated with a finitely generated subgroup $G$ of $\Diffgentwo$ possessing
finite orbits (or more generally locally finite orbits so as to allow for meromorphic as well as other types
of first integrals). Indeed,
the cornerstone of Mattei-Moussu's argument \cite{M-M} is the fact that a subgroup $G$ of $\diff$ all of whose orbits
are finite must be finite itself: a result no longer valid in dimension~$2$; see \cite{HelenaAndI}. Naturally, finite
groups always admit non-constant first integrals for their actions which leads to the existence of first integrals for
the initial foliation.

From the point of view of differentiable Galois theories, or from the point of view of Morales-Ramis-Sim\'o theory,
solvable groups are associated with integrable systems where {\it integrability}\,
should be understood in a type of {\it quadrature sense}\,
slightly more general than the standard context of Arnold-Liouville theorem. In this sense,
Theorem~A provides a fully satisfactory answer to the preceding question, namely the integrable character of
a subgroup of $\Diffgentwo$ possessing locally finite orbits lies in the fact that this group must be virtually solvable.

We are finally able to explain other aspects of Morales-Ramis-Sim\'o theory \cite{M-Ramis-S} that have provided us with
extra motivation for the present work. Inasmuch Galois differentiable theories are highly developed in the linear case, and
they allow us to decide whether or not a given equation is solvable by quadratures, a far
more general non-integrability criterion applicable to genuinely non-linear situations
is summarized by Morales-Ramis-Sim\'o theorem \cite{M-Ramis-S}. This theorem
asserts that the Galois group associated with the $k^{\rm th}$-variational equation arising from a periodic solution
must be virtually solvable (actually virtually abelian) provided that the system is integrable in the
sense of Arnold-Liouville. This context is somehow very close to our Theorems~A and~B and this issue deserves further comments.

The first main difference between the two sets of results lies in the groups considered: both Morales-Ramis and Morales-Ramis-Sim\'o
theories focus on {\it Galois groups}\, which may be larger than the more commonly used
{\it holonomy groups}, primarily concerned by
the results in this work. In this sense, the theories in \cite{MoralesRamis}, \cite{M-Ramis-S} are more complete since
they have a {\it better chance}\, at detecting non-integrable behavior. On the other hand, the advantage
of our direct analysis of the holonomy group is the possibility of providing further information on the actual dynamics of
several non-integrable systems. As a matter of fact, when the mentioned group is not (virtually) solvable, then our results allow
us to derive non-trivial conclusions concerning the dynamics of the (necessarily non-integrable) system in question.

Nonetheless, modulo a reasonable theory for the associated {\it Stokes phenomena}, which is often available for irregular
singular points, our statements can be applied to the dynamics of Galois groups as well. Indeed, the main difference
between the Galois group and the holonomy group lies in the fact that
the former also includes the so-called {\it Stokes diffeomorphisms} (Stokes matrices in the linear case).
In codimension~$2$, these ``Stokes diffeomorphisms'' are realized as local diffeomorphisms defined around
$(0,0) \in \C^2$ and they may or may not fix the origin. However, in the case they all fix the origin
(or rather if we decide to consider the subgroup formed by elements fixing the origin)
our results can directly be applied to investigate the dynamics of the resulting {\it Galois
group}, hence providing a nice complement to their theory.
In the more general case, both theories can probably be merged together into a similar dynamical study of
pseudogroups defined about $(0,0) \in \C^2$ that will not be discussed here. In any event, it should be pointed
out that, according to the point of view developed by Ramis and his co-authors,
this ``enlarged dynamics'', i.e. the dynamics associated with the Galois group, should be studied along with
the dynamics of the usual holonomy group. This remark opens the way to further applications of our results.

Another more specific, and likely deep, question raised by our results concerns
the classification of solvable non-abelian subgroups of $\Diffgentwo$ possessing locally finite orbits. The reader is
reminded that, for $n=1$, the corresponding result is due to Birkhoff, though it was independently
re-discovered by Loray in \cite{Loray}. Since this beautiful result possesses a number of applications, we believe
that its generalization to dimension~$2$ is a problem worth further investigation.

Let us close this introduction with an outline of the structure of the paper. The basic idea underlining most of
the present work is rather simple and comes from Ghys recurrence theorem proved in \cite{ghysBSBM} in a different context.
More precisely Ghys proves that a group of real analytic diffeomorphisms of a compact manifold
generated by diffeomorphisms close to the identity has recurrent dynamics provided that the group
is not {\it pseudo-solvable}. However, exploiting his idea to prove Theorem~A involves two main issues, the first one
being related to the assumption on {\it closeness}\, to the identity made in Ghys' theorem \cite{ghysBSBM}.
The other fundamental difficulty is related to the notion of {\it pseudo-solvable group}\, introduced in the same paper \cite{ghysBSBM}.
From an algebraic point of view, the main issue lies in the definition
of {\it pseudo-solvable group}\, which is related
to the fact that certain {\it sets of commutators}\, should not degenerate into the identity. This is actually a tricky point:
the geometric meaning of {\it pseudo-solvability}\, is not clear especially because the notion may, in principle, depend on
the generating set. As a consequence, the need to work with this type of groups
limits the dynamical applications of Ghys's ideas. To overcome this difficulty, we are led
to determine the borderline between pseudo-solvable groups and solvable groups when the former is realized as a group
of diffeomorphisms (with a given degree or regularity). In other words, we search for results ensuring that
a pseudo-solvable group of diffeomorphisms is, indeed, solvable.
This problem is already singled out in \cite{ghysBSBM} where
the author shows that every pseudo-solvable subgroup of $\diff$, or of the group of real analytic diffeomorphisms of the circle,
is solvable.

From the analytic side, our approach conceptually
hinges from the dichotomy involving discrete and non-discrete groups; see Section~\ref{section.examples}
for a detailed self-contained discussion. As already mentioned, among finitely generated subgroups of $\Diffgentwo$
there are groups that are {\it discrete}\, in a natural sense as well as groups that are {\it non-discrete}\, in the same sense.
Roughly speaking, a group is said to be non-discrete if it contains a non-trivial sequence of elements defined on some
fixed neighborhood $U$ of $(0,0) \in \C^2$ and converging uniformly to the identity on this neighborhood.
Ultimately, the importance of showing that a pseudo-solvable group is actually solvable lies in this dichotomy: the
corresponding result yields a powerful criterion to detect non-discrete groups. In fact, every sequence of
``iterated commutators'' starting from two elements sufficiently close to the identity will converge to the identity;
see Section~\ref{provingtheorems} (this explains the assumption on ``closeness to the identity'' made in Ghys's recurrence
theorem mentioned above). From this point, our general argument will allow us to connect {\it discrete subgroup of $\Diffgentwo$}\,
with Kleinian groups in an accurate sense. Then, by relying on basic facts from Kleinian group theory combined
to equally basic results on stable manifold theory of hyperbolic fixed point, we shall manage to establish Theorem~A
in the case of discrete groups. The complementary case of non-discrete groups can then be handled by resorting
to the argument on {\it convergence of iterated commutators close to the identity}\, as in \cite{ghysBSBM}.

The proof that a pseudo-solvable subgroup of $\Diffgentwo$ must be solvable (Theorem~\ref{commuting9})
is, however, the main technical difficulty in the paper. This is not really a surprise since the algebraic complexity
of subgroups of $\Diffgentwo$ is known to be much greater than the corresponding one-dimensional case
of subgroups of $\diff$ (see for example \cite{ecalle}).
The fact that Ghys result on pseudo-solvable subgroups of $\diff$ can be established in an easier way is related
to the fact that ``commutation relations'' in $\diff$ are very restrictive. Evidence for this issue arises, for example,
from the fact that the structure of solvable subgroups of $\diff$ has been well known for over two decades
(see \cite{Loraybook}, \cite{cerveaumoussu}, \cite{russians}) while there is relatively little literature in the higher
dimensional cases, apart from \cite{ecalle} and the recent papers \cite{ribon} and \cite{ribon2}.
Simple basic phenomena such as the possible existence of non-constant first integrals and the presence of higher
rank abelian groups, which have no one-dimensional analogue,
add significantly to the algebraic complication of the general picture.

A comment is needed in order to relate our discussion with the recent work by Martelo and Ribon in \cite{ribon}.
These authors have provided a systematic treatment of solvable subgroups of $\Diff$
at {\it formal level} in the sense that their results apply to the group of formal diffeomorphisms and not only to convergent ones.
In view of the purpose of this present work, we can make our discussion shorter by restricting ourselves
to the case of subgroups of $\diffCtwo$, where $\diffCtwo$ denotes the subgroup
of $\Diffgentwo$ consisting of diffeomorphisms tangent to the identity.
In this context, the results in \cite{ribon} can be summarized as follows. To a group $G \subset \diffCtwo$ they associate
a Lie algebra of {\it formal vector fields}\, whose exponentiation contains the initial group $G$. Moreover, if $G$
is solvable then so is its Lie algebra. This correspondence between a solvable Lie group $G \subset \diffCtwo$ and
its solvable Lie algebra has a number of natural properties and, essentially, reduces the problem of formally classifying
solvable subgroups of $\diffCtwo$ to the problem of classifying solvable Lie algebras of formal vector fields in
two variables. The latter problem is also settled in the mentioned paper through an inductive procedure
(Theorem~6 of \cite{ribon}).

The results of \cite{ribon} will find applications in our study of pseudo-solvable groups. However, further elaborations will also be
needed. The need for additional elaborations has its roots in the fact that
our purpose is to prove that a pseudo-solvable subgroup $G \subset \diffCtwo$
is, indeed, solvable. Therefore the initial group $G$ is not known to be solvable. Still, we may try to
consider its Lie algebra in the sense of \cite{ribon}. The main problem here is to determine that this Lie algebra
should be ``pseudo-solvable'', with an appropriate definition of {\it pseudo-solvable Lie algebras}. The very
definition of pseudo-solvable groups and its, a priori, dependence on the generating set is the main obstacle to
exploit this type of idea. However, if this connection between groups and Lie algebras can be made accurate, then
the the desired statement will be reduced to prove its version for Lie algebras. These questions will be further
detailed in the course of the work.

This paper is organized as follows. Section~2 contains background material and
is divided in three paragraphs. The first one contains several pertinent definitions including the statement
of Theorem~\ref{commuting9} claiming, in particular, that a ``pseudo-solvable subgroup'' of $\Diffgentwo$ is solvable.
This theorem is the main algebraic result which cannot be avoided in the proof of Theorems~A and~B as well as in the description
of ``discrete'' and of ``non-discrete'' subgroups of $\Diffgentwo$. In addition to basic preparatory material, Section~2.3
contains a review of \cite{ribon} construction of the Lie algebra associated to a given subgroup of $\diffCtwo$ or,
more generally of $\formdiffn$. The other result of \cite{ribon}, see also
\cite{ribon2}, that will be used in this paper is the formal classification
of solvable subgroups of $\formdiffn$ and the corresponding list will de made explicit
in Section~5.3: it amounts to a particular case of Theorem~6 in \cite{ribon} and further detail
can also be found in the recent preprint \cite{ribon2}.

Going back to Theorem~\ref{commuting9}, its proof is the object of Sections~5, 6, and~7.
In Section~5.1 and~5.2, more advanced results concerning abelian subgroups of $\formdiffn$ as well as their normalizers
in $\formdiffn$ will be detailed. Finally, Sections~6 and~7 are entirely devoted
to proving Theorem~\ref{commuting9} by building in the previously developed material. We also note that
Campbell-Hausdorff type formulas will play a prominent role in much of the proofs given in Sections~5 and~6.

The reader willing to take for grant the statement of Theorem~\ref{commuting9} will find the proofs of Theorems~A
and~B in Section~3. Additional details and examples illustrating these theorems are supplied in the short Section~4.

\vspace{0.1cm}

\noindent {\bf Acknowledgments}. We are very indebted to the referee for several comments and suggestions that helped us
to simplify our discussion of pseudo-solvable subgroups of $\Diffgentwo$.

The first author wishes to thank M. Garakani for discussions concerning algebraic properties
of subgroups of $\Diff$. Part of this work was conducted during a visit of the authors to
IMPA and we would like to thank the CNPq-Brazil for partial financial support.
The second author was partially supported by FCT through CMUP. Finally, both authors were also supported
by project EXPL/MAT-CAL/1575/2013.

\section{Basic notions}

Throughout this work $\Diffgentwo$ stands for the group of germs of holomorphic diffeomorphisms fixing $(0,0) \in \C^2$
and $\diffCtwo$ denotes its normal subgroup consisting of diffeomorphisms tangent to the identity.
The group of formal diffeomorphisms
of $(\C^2,0)$ is denoted by $\formdiffgeneral$ whereas $\formdiffn$ is the formal counterpart of
$\diffCtwo$ i.e., it is constituted by formal diffeomorphisms tangent to the identity.
Similarly, by $G$ we shall always denote a finitely generated subgroup of one of the groups $\Diffgentwo$,
$\diffCtwo$ or $\formdiffn$.

To make accurate our discussion, it is convenient to begin with a few standard definitions.
First, let $\formalC$ denote
the space of formal series in the variables $x,y$. Similarly $\fieldC$ will stand for the field of fractions
(or field of quotients) of $\formalC$. An element $F \in \formdiffgeneral$ consists of a pair of formal series
$(F_1(x,y) , F_2 (x,y))$, $F_1 (x,y), \, F_2(x,y) \in \formalC$, satisfying the following condition: setting
$F_1 (x,y) = a_1 x + a_2y + {\rm h.o.t.}$ and $F_2 (x,y) = b_1 x + b_2y + {\rm h.o.t.}$, the $2 \times 2$
matrix whose entries are the coefficients $a_1, a_2, b_1, b_2$ is invertible. The formal diffeomorphism
$F$ is said to belong to $\formdiffn$ when this matrix happens to coincide
with the identity.

\subsection{Pseudogroups and additional terminology}

Assume that $G$ is actually a subgroup of $\Diffgentwo$ generated by the elements $h_1, \ldots, h_k$.
A natural way to make sense of the local dynamics of $G$ consists
of choosing representatives for $h_1, \ldots, h_k$ as local diffeomorphisms fixing $(0,0) \in \C^2$.
These representatives are still denoted
by $h_1, \ldots, h_k$ and, once this choice is made, $G$ itself can be identified to the {\it pseudogroup}\,
generated by these local diffeomorphisms on a (sufficiently small) neighborhood of the origin. It is however convenient to
recall the definition of {\it pseudogroup}. For this, consider a small neighborhood $V$ of the origin where the
local diffeomorphisms $h_1, \ldots, h_k$, along with their inverses $h_1^{-1}, \ldots, h_k^{-1}$, are all well defined diffeomorphisms
onto their images.
The pseudogroup generated by $h_1, \ldots, h_k$ (or rather by $h_1, \ldots , h_k, h_1^{-1}, \ldots, h_k^{-1}$
if there is any risk of confusion) on $V$ is defined as follows. Every element of this pseudogroup has the form
$F = F_s \circ \ldots \circ F_1$ where each $F_i$, $i \in \{1, \ldots, s\}$, belongs to the set
$\{h_i^{\pm 1}, i=1, \ldots, k\}$. The element $F$ should be regarded as an one-to-one holomorphic map
defined on a subset of $V$. Indeed, the domain of definition of $F = F_s \circ \ldots \circ F_1$, as an
element of the pseudogroup, consists of those points $x \in V$ such that for every $1 \leq l < s$ the point $F_l \circ
\ldots \circ F_1(x)$ belongs to $V$. Since the origin is fixed by the diffeomorphisms $h_1, \ldots, h_k$, it follows that
every element $F$ in this pseudogroup possesses a non-empty open domain of definition. This
domain of definition may however be disconnected. Whenever no misunderstanding is possible, the pseudogroup defined above will also
be denoted by $G$ and we are allowed to shift back and forward from $G$ viewed as pseudogroup or as group of germs.

Let us continue with some definitions that will be useful throughout the text. Suppose we are given local holomorphic
diffeomorphisms $h_1, \ldots, h_k, h_1^{-1}, \ldots, h_k^{-1}$ fixing the origin of $\C^n$. Let $V$ be a
neighborhood of the origin where all these maps yield diffeomorphisms from $V$ onto the corresponding image. From now on, $G$ will be
identified to the corresponding pseudogroup on $V$.
Given an element $h \in G$, the domain of definition of $h$ (as element of $G$) will be denoted by ${\rm Dom}_V (h)$.

\begin{defi}
The $V_G$-orbit $\calO_V^G (p)$ of a point $p \in V$ is the set of points in $V$ obtained from $p$ by
taking its image through every element of $G$ whose domain of definition (as element of $G$) contains $p$.
In other words,
\[
\calO_V^G (p) = \{q \in V \; \, ; \; \,  q = h(p), \; h \in G \; \; {\rm and} \; \; p \in {\rm Dom}_V (h) \} \, .
\]
Fixed $h \in G$, the $V_h$-orbit of $p$ can be defined as the $V_{<h>}$-orbit of $p$, where $<h>$ denotes the subgroup
of $\diffn$ generated by $h$.
\end{defi}

\begin{defi}
\label{finiteorbitsforpoints}
Given a pseudogroup $G$ and a point $p$, the orbit $\calO_V^G (p)$ of $p$ under $G$ is said to be {\it finite} if
the set $\calO_V^G (p)$ is finite. This orbit $\calO_V^G (p)$ is called {\it locally discrete (or locally finite)},
if there is a neighborhood $W \subset \C^n$ of $p$ such that $W \cap \calO_V^G (p) = \{ q \}$. Finally,
if the orbit of $p$ is not locally discrete then it is said to be {\it recurrent}.
\end{defi}

For $G$ and $V$ as above,
we can now define the notions of {\it pseudogroups with finite orbits}\, and of
{\it pseudogroups with locally discrete orbits}\, (or, equivalently, locally
finite orbits).

\begin{defi}\label{def_finiteorbits}
A pseudogroup $G \subseteq \Diffgentwo$ is said to have finite orbits if there exists a sufficiently small
open neighborhood $V$ of $0 \in \C^n$ all of whose points have finite orbits.
Analogously, $h \in G$ is said to have finite orbits if the pseudogroup
$\langle h \rangle$ generated by $h$ has finite orbits.

Similarly, a pseudogroup is said to have locally discrete orbits (or locally finite orbits)
if there is $V$ small as above such that every point in $V$ has locally discrete orbit.
\end{defi}

Let us now remind the reader
the definition of solvable group. Let $G$ be a given group and denote by $D^1 G$
its {\it first derived group}, namely the subgroup generated by all elements of the form $[g_1, g_2] =
g_1 \circ g_2 \circ g_1^{-1} \circ g_2^{-1}$ where $g_1, g_2 \in G$. The second derived group $D^2 G$
of $G$ is defined as the first derived group of $D^1 G$, i.e., $D^2 G =D^1 (D^1 G)$. More
generally, we set $D^j G = D^1 (D^{j-1} G)$. The group
$G$ is said to be {\it solvable}\, if the groups $D^j G$ become reduced to $\{ {\rm id} \}$ for sufficiently
large $j \in \N$. The smallest $r \in \N^{\ast}$ for which $D^r G = \{ {\rm id} \}$ is called the {\it derived
length}\, of $G$. Equivalently, the group $G$ is also said to be step-$r$ solvable. Thus, an abelian group
is step-$1$ solvable. Step-$2$ solvable groups are also called {\it metabelian groups}.

Since this will be needed later, we may also provide the definition of a solvable Lie algebra. Let then
$\mathfrak{g}$ denote a Lie algebra. The {\it first derived algebra} $D^1 \mathfrak{g}$ of $\mathfrak{g}$
is defined as the Lie algebra generated by the elements $[X,Y]$ where $X, Y \in \mathfrak{g}$. The $j^{\rm th}$-derived
algebra $D^j \mathfrak{g}$ is inductively defined by setting $D^j \mathfrak{g} = D^1 (D^{j-1} \mathfrak{g})$.
Naturally $\mathfrak{g}$ is said to be solvable if $D^j \mathfrak{g}$ is reduced to {\it zero} for sufficiently
large $j \in \N$. Again the derived length of a Lie algebra is defined as the smallest positive $r \in \N^{\ast}$
for which $D^k \mathfrak{g} = \{ 0 \}$.

Closely related to solvable groups is the notion of nilpotent groups. In this case, we consider
$C^1G =D^1G$ (also sometimes called the {\it first central subgroup}). The {\it central series}\, $C^jG$ of $G$
is inductively defined by letting $C^j G$ to be the group generated by all elements of the form $[a,b]$
where $a \in G$ and $b \in C^{j-1} G$. A group $G$ is said to be step-$r$ nilpotent if $r$ is the smallest
integer for which $C^r G = \{ {\rm id} \}$.
We leave to the reader to adapt this definition to Lie algebras.

It is now convenient to recall the definition of pseudo-solvable groups as formulated in \cite{ghysBSBM}.

\begin{defi}
\label{definitionpseudosolvablegroup}
Let $G$ be a group and consider a given finite generating set $S$ for $G$.
To the generating set $S$, a sequence of sets $S(j) \subseteq G$ is associated as follows: $S(0) =S$ and $S(j+1)$
is the set whose elements are the commutators written under the form $[F_1^{\pm 1} ,F_2^{\pm 1}]$ where $F_1 \in S(j)$ and
$F_2 \in S(j) \cup S(j-1)$ ($F_2 \in S(0)$ if $j=0$).
The group $G$ is said to be {\it pseudo-solvable}\, if it admits a (finite) generating set $S$ as above such that
the sequence $S(j)$ becomes reduced to the identity for $j$ large enough.
\end{defi}

As already mentioned, a large part of this paper is devoted to the following

\begin{theorem}
\label{commuting9}
A pseudo-solvable subgroup $G$ of $\formdiffn$ is necessarily solvable.
\end{theorem}

The discussion revolving around the proof of Theorem~\ref{commuting9} is of algebraic nature and the
corresponding results are of interest in their own right. In dimension~$1$, the analogous result was established in \cite{ghysBSBM}
and the argument employed there suggests a natural strategy to handle other situations. However, once we
try to implement this strategy for, say, subgroups of $\formdiffn$, several new difficulties are quickly encountered.
Among these difficulties, we quote
the existence of non-constant first integrals and the existence of rank~$2$ abelian groups. Clearly
neither of these phenomena has an one-dimensional analogue and this partially accounts for the much simpler nature
of the problem in dimension one.
Another point concerning the above mentioned Ghys's strategy is that it naturally
requires some previous knowledge of the structure of solvable subgroups from $\diffCtwo$ or from $\formdiffn$.
In dimension one, the structure of these groups is highly developed (see for example \cite{nakai}, \cite{russians},
\cite{cerveaumoussu} and \cite{Loraybook}) and the corresponding information comes in hand when implementing Ghys's strategy.

Only recently similar material on the classification of solvable subgroups of
$\formdiffn$ (or more generally of $\widehat{\rm Diff}_1 (\C^n ,0)$) has become available through the work of
Martelo and Ribon \cite{ribon}. As already mentioned, these authors associate a Lie algebra to a subgroup of,
say, $\formdiffn$ so that this Lie algebra is solvable if and only if the initial group is so. Moreover, the exponential
of this Lie algebra contains the initial group. By means of this construction, the classification of solvable
subgroups of $\formdiffn$ becomes reduced to the classification of solvable Lie algebras of formal vector fields
with zero linear part. The structure of these solvable Lie algebras is also clarified in their work. However, in the
context of proving Theorem~\ref{commuting9}, further elaboration of their results is needed since
we are dealing with pseudo-solvable groups which a priori are not solvable. Naturally, we may still consider
the Lie algebra associated to a pseudo-solvable subgroup of $\formdiffn$ and try to prove this Lie algebra
is ``pseudo-solvable'' with an appropriate definition of pseudo-solvability for Lie algebras. The difficulty to carry out
this idea lies in the fact that it is hard to compute the {\it infinitesimal generator}\, of the commutator of
two elements $F_1, F_2 \in \formdiffn$ in terms of the infinitesimal generators of $F_1, F_2$; see Section~2.2. Indeed,
this type of computation is governed by the Campbell-Hausdorff formula whose complexity is very high for arbitrary
elements $F_1, F_2$. However, by suitably blending our knowledge of solvable subgroups of $\formdiffn$ with
Campbell-Hausdorff type formulas, these difficulties will eventually be overcome and Theorem~\ref{commuting9} established.
Though this will not be made explicit in the course of our discussion, our method actually shows that the Lie algebra
associated to a pseudo-solvable group has special properties that might be used to define pseudo-solvable Lie algebras.
Yet, since ultimately all these groups will turn out to be solvable, we consider that working out these notions more
explicitly is not really necessary.

Going back to Theorem~\ref{commuting9}, the above mentioned strategy to prove that a given pseudo-solvable group
is actually solvable is naturally suggested by the very definition of pseudo-solvable group.
This is as follows. Consider a pseudo-solvable group
$G$ along with a finite generating set $S =S(0)$ leading to a sequence of sets $S(j)$ that degenerates into
$\{ {\rm id} \}$ for large enough $j \in \N$. Denote by
$G(j)$ (resp. $G (j,j-1)$) the subgroup generated by
$S(j)$ (resp. $S(j) \cup S(j-1)$). Let $k$ be the largest integer for which $S(k)$ is not reduced to the identity.
It then follows that $G(k)$ is abelian. Similarly the group $G (k,k-1)$
is solvable. In particular, the {\it smallest}\, integer $m$ for which
$G (m,m-1)$ is solvable can be considered. Furthermore, if $m=1$ then the initial group
$G$ is solvable and hence there is nothing else to be proved. Suppose then that $m \geq 2$ and note that
every element $F$ in $S (m-2)$ satisfies the condition
\begin{equation}
F^{\pm 1} \circ G (m-1) \circ F^{\mp 1} \subset G (m,m-1) \, . \label{Conjugatingsolvablegroups}
\end{equation}
To derive a contradiction with the fact that $m \geq 2$ (so that $G$ is not solvable), we only need to show
that $G (m-1,m-2)$ must be solvable as well. In other words, we need to show that the group generated by
$$
G (m,m-1) \cup S(m-2)
$$
is still {\it solvable}. To establish this statement, we are however allowed to
exploit the assumption that the {\it elements $F$ of $S(m-2)$ satisfy the condition expressed in
Equation~(\ref{Conjugatingsolvablegroups})}\,
where $G(m-1)$ and $G (m,m-1)$ are both solvable groups with $G (m) \subset G (m,m-1)$. Furthermore neither $G(m-1)$
nor $G(m)$ is reduced to $\{ {\rm id}\}$. Indeed, we can be slightly more precise by saying that for every
$F \in S(m-2)$ and $g \in S(m-1)$, we have
\begin{equation}
F \circ g \circ F^{-1} = g \circ \overline{g} \, , \label{Conjugatingsolvablegroups-2}
\end{equation}
for some $\overline{g} \in S(m)$. In any event, we are then led to investigate the structure of
{\it the solutions ``$F$'' of the functional relation expressed by~(\ref{Conjugatingsolvablegroups})}.
Besides, and inasmuch we shall apply Theorem~\ref{commuting9} only to subgroups of $\diffCtwo$, the issue about
convergence of power series will play no role in the course of the discussion. This explains why Theorem~\ref{commuting9}
is stated for formal subgroups of $\formdiffn$.

In the present case both $G(m-1)$ and $G (m,m-1)$ are subgroups of $\formdiffn$ and this explains
why the implementation of the above mentioned strategy requires detailed
information on solvable subgroups of $\formdiffn$. At this point, we shall have occasion to take advantage of the
results established in \cite{ribon}.

\subsection{Some formal computations}

Here some basic statements concerning formal diffeomorphisms
in $\formdiffn$ and formal vector fields will quickly be reviewed as a preparation for more elaborate arguments.
To begin with, consider again the set $\ghatXone$ of formal vector fields at $(\C^2,0)$ so that
every element (formal vector field) in $\ghatXone$ has the form
$a(x,y) \partial /\partial x + b(x,y) \partial /\partial y$ where $a, \, b \in \formalC$.
The space of formal vector fields whose first jet at the origin
vanishes is going to be denoted by $\ghatX$.
Formal vector fields as above act as derivations on $\formalC$ by means of the formula
$X_{\ast} f = df. X \in \formalC$, where $f \in \formalC$ and $X \in \ghatXone$. This action can naturally
be iterated so that $(X)^k_{\ast} f$ is inductively defined by $X_{\ast} [ (X)^{k-1}_{\ast} f]$ for $k \in \N$.
By way of definition, we also have $(X)^0_{\ast} f =f \in \formalC$.

Next, let $t \in \C$ and $X \in \ghatXone$ be fixed. The {\it exponential of $X$ at time-$t$}, $\exp (tX)$, can be defined as
the operator from $\formalC$ to itself given by
\begin{equation}
\exp (tX) (h) = \sum_{j=0}^{\infty} \frac{t^j}{j!} (X)^j_{\ast} h \, . \label{betterasformula}
\end{equation}
Naturally $\exp (0.X)$ is the identity operator and $\exp (t_1 X) \circ \exp (t_2X) = \exp ((t_1+t_2)X)$.

Recall that the order of a function (or vector field) at the origin is
the degree of its first non-zero homogeneous component. Suppose then
that $X \in \ghatX$ so that $X =a(x,y) \partial /\partial x + b(x,y) \partial /\partial y$ where the orders of both $a, \, b$
at $(0,0) \in \C^2$ are at least~$2$. It then follows that the order of $X_{\ast} h$ is strictly greater
than the order of $h$ itself. In particular, for $h=x$, we conclude that
\begin{equation}
\exp (tX) (x) = x + t.a(x,y) + \cdots \; \; \, {\rm and} \; \; \,  \exp (tX) (y) = y + t.b(x,y) + \cdots \label{previousformula}
\end{equation}
where the dots stand for terms whose degrees in $x,y$ are strictly greater than the order of $a$ (resp. $b$) at the origin.
Therefore, for every $X \in \ghatX$ and every $t \in \C$, the pair of formal
series $(\exp (tX)(x), \exp (tX) (y))$ can be viewed as an element of $\formdiffn$, namely the group of formal diffeomorphisms
of $(\C^2,0)$ that are tangent to the identity at the origin. In turn, we call the {\it exponential of $X$}\,
the subgroup of $\formdiffn$ consisting of all formal diffeomorphisms $(\exp (tX)(x), \exp (tX) (y))$, $t \in \C$.
If the vector field $X$ happens to be holomorphic,
as opposed to merely formal, then $(\exp (tX)(x), \exp (tX) (y))$
is an actual diffeomorphism tangent to the identity and coinciding with the diffeomorphism induced by the local
flow of $X$ at time~$t$. Next, by letting ${\rm Exp} \, (X) = (\exp (X)(x), \exp (X) (y))$ and, more generally,
${\rm Exp} \, (tX) = (\exp (tX)(x), \exp (tX) (y))$, the following well-known lemma holds:
\begin{lemma}
\label{correspondence}
The map ${\rm Exp}$ settles a bijection between $\ghatX$ and $\formdiffn$.
\end{lemma}

\begin{proof}
In the sequel $p_n (x,y), \, q_n (x,y) , \, a_n (x,y), \, b_n (x,y)$ denote homogeneous polynomials of degree~$n$ in the variables
$x,y$. Let $F \in \formdiffn$ be given by $F(x,y) = (x + \sum_{n=2}^{\infty} p_n (x,y) , y + \sum_{n=2}^{\infty} q_n (x,y))$. Similarly
consider a vector field $X \in \ghatX$ given as
$$
X = \sum_{n=2}^{\infty} \left[ a_n (x,y) \frac{\partial}{\partial x} +  b_n (x,y) \frac{\partial}{\partial y} \right] \, .
$$
The equation ${\rm Exp}\, (X) =F$ amounts to $p_{m+1} = a_{m+1} + R_{m+1} (x,y)$ and
$q_{n+1} = b_{n+1} + S_{m+1} (x,y)$ where $R_{m+1} (x,y)$ (resp. $S_{m+1} (x,y)$) stands for the homogeneous component
of degree~$m+1$ of the vector field
$$
\sum_{j=2}^m \frac{1}{j!} (Z_m)^j_{\ast} (x)
$$
(resp. of $\sum_{j=2}^m (Z_m)^j_{\ast} (y) /  j!$), where
$Z_m = \sum_{n=2}^{m} [ a_n (x,y) \partial /\partial x +  b_n (x,y) \partial /\partial y ]$.
These equations show that, given $F \in \formdiffn$, there is one unique $X \in \ghatX$ such that
${\rm Exp} \, (X) =F$. The lemma is proved.
\end{proof}

For $F \in \formdiffn$, recall that the formal vector field $X$ satisfying ${\rm Exp} \, (X) =F$ is
called the {\it infinitesimal generator of $F$}. The notation $X = \log \, (F)$ may also be used to state
that $X$ is the infinitesimal generator of $F \in \formdiffn$.
Note that the series of $X$ need not converge even when $F$ is an actual holomorphic diffeomorphism.

Recall that the order ${\rm ord}\, (f)$ at $(0,0)$ of an element $f \in \formalC$ is nothing but the degree of the first non-zero
homogeneous component of the formal series of $f$.
Next, if $F \neq {\rm id} $ is a formal diffeomorphism tangent to the identity, the order of the (formal) function
$F-{\rm id}$ is called the contact order with the identity of $F$. Here the order of the formal function $F-{\rm id}$
is defined by considering the minimum of the order of its components. Now, we have:

\begin{lemma}
\label{commuting1}
Consider two elements $F_1,\, F_2$ in $\formdiffn$ together with their respective infinitesimal generators $X_1, \, X_2$.
Then the following holds:
\begin{enumerate}
  \item $F_1, \, F_2$ commute if and only if so do $X_1, \, X_2$.
  \item If $F_1, \, F_2$ do not commute, then the contact order with the identity of $[F_1, F_2] =
  F_1 \circ F_2 \circ F_1^{-1} \circ F_2^{-1}$ is strictly greater than the corresponding orders of $F_1$ and of $F_2$.
\end{enumerate}

\end{lemma}

\begin{proof}
The statement is very well known and can also be seen as
a particular case of the results in \cite{ribon} comparing nilpotence lengths
for a nilpotent group and for its nilpotent Lie algebra. We shall give an elementary argument that will help us to
state a consequence of Campbell-Hausdorff formula that will find further applications later on.

Consider the first claim in the above statement. It suffices to show that $[X_1, X_2] =0$ provided that $F_1$ and $F_2$ commute
since the converse is clear.
For this, denote by $Z_+$ (resp. $Z_-$) the infinitesimal generator of $F_1 \circ F_2$ (resp. $F_1^{-1} \circ F_2^{-1}$). The
diffeomorphisms $F_1, F_2$ commute if and only if
$F_1 \circ F_2 \circ F_1^{-1} \circ F_2^{-1} = {\rm Exp} \, (Z_+) {\rm Exp} \, (Z_-) = {\rm id}$.
Denoting by $Z$ the infinitesimal generator of $F_1 \circ F_2 \circ F_1^{-1} \circ F_2^{-1}$, we have
\begin{eqnarray*}
Z & = & \log \, ( {\rm Exp} \, (Z_+) {\rm Exp} \, (Z_-) ) = \\
 & = & Z_+ + Z_- + \frac{1}{2} [Z_+, Z_-] + \frac{1}{12} [Z_+ ,[Z_+,Z_-]] - \frac{1}{12} [Z_- ,[Z_+,Z_-]] + {\rm h.o.t.}
\end{eqnarray*}
as it follows from Campbell-Hausdorff formula, see \cite{who??}. In turn,
$$
Z_+  =  \log \, ( F_1 \circ F_2 ) = \log \, ( {\rm Exp} \, (X_1) {\rm Exp} \, (X_2)) =
 X_1 + X_2 + \frac{1}{2} [X_1, X_2] + \cdots \, .
$$
Analogously
$$
Z_- = -X_1 -X_2 + \frac{1}{2} [X_1, X_2] + \cdots \, .
$$
Therefore
\begin{eqnarray}
Z & = & X_1 + X_2 + \frac{1}{2} [X_1, X_2] + \cdots  + ( -X_1 -X_2 + \frac{1}{2} [X_1, X_2] + \cdots ) +  \nonumber \\
 & & + \frac{1}{2} \left[X_1 + X_2 + \frac{1}{2} [X_1, X_2] + \cdots ,
-X_1 -X_2 + \frac{1}{2} [X_1, X_2] + \cdots \right] + \cdots \nonumber \\
 & = & [X_1,X_2] + \frac{1}{2} [X_1,[X_1,X_2]] + \frac{1}{2} [X_2,[X_1,X_2]] +
\cdots \, . \label{finalcommutators}
\end{eqnarray}
Assuming that $[X_1, X_2]$ does not vanish identically, we can write
$[X_1, X_2]  = \sum_{j \geq k}^{\infty} Y_j$,
where $Y_j$ is a degree~$j$ homogeneous vector field and where $Y_k$ is not
identically zero. The orders
of the higher iterated commutators appearing in Equation~(\ref{finalcommutators}) are strictly greater than $k$, since
the orders of $X_1,X_2$ at the origin are at least~$2$. In other words, we have $Z = Y_k + {\rm h.o.t.}$. Since
$F_1 \circ F_2 \circ F_1^{-1} \circ F_2^{-1} = {\rm Exp} \, (Z)$, it follows that $F_1, F_2$ do not commute either and this
establishes the first assertion.

Concerning the second assertion suppose that $F_1, F_2$ do not commute. Then the Lie bracket $[X_1, X_2]$ does not
vanish identically. Therefore Formula~(\ref{finalcommutators}) shows that the order of contact with the identity
of $[F_1, F_2]$ coincides with the order of $[X_1 ,X_2]$ at the origin. The latter order is strictly greater than
the maximum between the orders of $X_1$ and $X_2$ since the first jets of both $X_1 , \, X_2$ vanish at the
origin. In fact, if $r \geq 2$ (resp. $s \geq 2$) stands for the order of $X_1$ (resp. $X_2$) at the origin,
then the order of $[X_1 ,X_2]$ is, at least, equal to $r+s-1$. The statement follows at once.
\end{proof}

The order at $(0,0)$ of a formal vector field
$X \in \ghatXone$ is the minimum between the orders of its components and this is well defined since this minimum
does not depend on the choice of the formal coordinates. Next consider
an element $h \in \fieldC$, the quotient field of $\formalC$,
and set $h =f/g$ with $f,g \in \formalC$. The order of $h$ at $(0,0)$ can be defined
as the unique integer $n \in \Z$ for which the limit
$$
\lim_{\lambda \rightarrow 0} \frac{h(\lambda x ,\lambda y)}{\lambda^n}
$$
is a non-identically zero quotient of two homogeneous polynomials. Alternatively, this value of $n$ is simply
the difference ${\rm ord}\, (f) - {\rm ord}\, (g)$. The extension of this definition to formal vector fields
with coefficients in $\fieldC$ is immediate: the order at $(0,0)$ of the vector field in question
is the minimum between the orders of its components. Clearly this notion of order is again well defined since it does
not depend on the choice of the formal coordinates.

In what follows, a formal vector field $X$ with coefficients in $\formalC$ will often be referred to as a (formal) vector field
belonging to $\ghatXone$ (or occasionally to $\ghatX$). Unless otherwise mentioned, whenever we talk about formal vector fields
without specifying that they belong to either $\ghatXone, \, \ghatX$ they are allowed
to have coefficients in $\fieldC$.

Two formal vector fields $X, Y \in \ghatX$ are said to be {\it everywhere parallel}\, if $X$ is a multiple of $Y$
by an element in $\fieldC$. When $X, Y \in \ghatX$ {\it are not}\, everywhere parallel, then every formal vector field
$Z \in \ghatXone$ can be expressed as a linear combination of $X, Y$ with
coefficients in $\fieldC$. More precisely, for $X, \, Y$ and $Z$ as above let
$$
Z = f X + gY
$$
where $f,g \in \fieldC$. In fact, by setting $X = A \partial /\partial x + B \partial /\partial y$, $Y = C\partial /\partial x
+ D \partial /\partial y$ and $Z = P \partial /\partial x + Q \partial /\partial y$, we obtain:
\begin{equation}
f = \frac{PD - QC}{AD - BC} \; \; \, {\rm and} \; \; \, g = \frac{QA - PB}{AD - BC} \, . \label{veryelementary}
\end{equation}

To close this section, recall that the standard Hadamard lemma
expresses the pull-back of a vector field $X$ by a formal diffeomorphism $F$ in terms of the infinitesimal generator $Z$ of $F$;
see \cite{who??}.
More precisely, Hadamard lemma provides us with the formula
\begin{equation}
F^{\ast} X = X + [Z,X] + \frac{1}{2} [Z,[Z,X]] + \frac{1}{3!} [Z,[Z,[Z,X]]] +
\cdots \, . \label{hadamardlemmastatement1}
\end{equation}

\subsection{Subgroups of $\diffCtwo$ and their Lie algebras}

This section contains a summary of Martelo Ribon's construction \cite{ribon} of a Lie algebra associated to a group
of formal diffeomorphisms. For brevity, we restrict ourselves to subgroups of $\formdiffn$.

Let $\mathfrak{m}$ denote the maximal ideal of $\formalC$ and note that every formal diffeomorphism
$f \in \formdiffn$ acts on the vector space $\mathfrak{m}/\mathfrak{m}^k$ of $k$-jets of elements in $\formalC$.
More precisely $f$ defines an element $f_k \in {\rm GL}\, (\mathfrak{m}/\mathfrak{m}^k)$ whose action on
the vector space $\mathfrak{m}/\mathfrak{m}^k$ is given by $g+\mathfrak{m}^k \mapsto g \circ f +\mathfrak{m}^k$.
Next, let $D_k \subset {\rm GL}\, (\mathfrak{m}/\mathfrak{m}^k)$ be the subgroup consisting of those automorphisms
having the form $\{ f_k \} \in {\rm GL}\, (\mathfrak{m}/\mathfrak{m}^k)$ for some $f \in \formdiffn$. It is easy
to check that $D_k$ is an algebraic group. Furthermore there are natural (restriction) morphisms
$\pi_k : D_{k+1} \rightarrow D_k$ of algebraic groups for every $k \in \N^{\ast}$.

Suppose now that we are given a group $G \subset \diffCtwo$. Fixed $k \in \N^{\ast}$,
we can consider all automorphisms in ${\rm GL}\, (\mathfrak{m}/\mathfrak{m}^k)$ having the form $\{ f_k \}$
for some $f \in G$. The Zariski-closure $G_k$ of this group is the smallest algebraic subgroup of $D_k$
containing all the mentioned automorphisms. Clearly $G_k$ is itself an algebraic group and the natural
character of the preceding constructions ensures that $\pi_k$ sends $G_{k+1}$ to $G_k$. The following
lemma is very standard.

\begin{lemma}
\label{connectnessfromribon}
The groups $G_k$ are connected for every $k \in \N^{\ast}$.
\end{lemma}

\begin{proof}
Consider an element $f_k$ in $G_k$. The element $f_k$ is induced by a certain element $f \in \formdiffn$.
In turn, $f$ is the time-one map of a formal vector field $X$. However, for $k$ fixed, the
mentioned formal vector field induces an actual element $X_k$ in ${\rm End}\, (\mathfrak{m}/\mathfrak{m}^k)$
and we have $f_k = {\rm Exp}\, (X_k)$, where the exponential here is to be understood in the sense of a finite
dimensional algebraic group. Now consider an algebraic equation $\mathcal{R}$ whose solution
set contains $G_k$. Note that $\mathcal{R} ({\rm Exp}\, (tX_k))$ is a polynomial in the variable $t$ and this
polynomial must vanish at the integral powers of $f_k$. Since $f_k$ is not of torsion in $G_k$
(since $f$ is tangent to the identity), it follows that the polynomial $\mathcal{R} ({\rm Exp}\, (tX_k))$
vanishes for every $t \in \N$ and hence it must vanish identically. Thus we conclude that
${\rm Exp}\, (tX_k)$ is contained in $G_k$ for every $t \in \C$ so that $f_k$ can be connected to
the identity by a path contained in $G_k$. The lemma follows at once.
\end{proof}

Next we set
$$
\overline{G} = \{ f \in \formdiffn \; ; \; \, f_k \in G_k \;\, {\rm for \; \, every} \, \; k \in \N^{\ast} \} \, .
$$
The group $\overline{G}$ is closed for the Krull topology and it clearly contains $G$. Also, there follows from
Lemma~\ref{connectnessfromribon} that $\overline{G}$ is connected. Moreover, by construction, it is also
clear that $\overline{G}$ inherits the algebraic properties of $G$. In other words, $\overline{G}$ is solvable
(resp. nilpotent) if and only if $G$ is so. Furthermore, in these cases, the derived lengths (resp. nilpotent
length) of both $G, \, \overline{G}$ coincide. By slightly abusing notations, the group $\overline{G}$ defined
above will often be referred to as the {\it Zariski-closure}\, of $G$.

For every $k \in \N^{\ast}$, let $\mathfrak{g}_k$ denote the Lie algebra associated to the algebraic
group $G_k$. Consider the Lie algebra $\mathfrak{g} \subset \ghatX$ defined as follows:
\begin{equation}
\mathfrak{g} = \{ X \in \ghatX \; ; \; \, X_k \in \mathfrak{g}_k
\;\, {\rm for \; \, every} \, \; k \in \N^{\ast} \} \, . \label{definitionLiealgebraassociated}
\end{equation}
The Lie algebra $\mathfrak{g}$ is, by definition, the Lie algebra associated to the initial group $G \subset \diffCtwo$.

Here is a good point to further explain some comments made in Section~2.1 concerning the Lie algebra
$\mathfrak{g}$ and the structure of pseudo-solvable subgroups of $\formdiffn$. A solvable
algebraic group possesses a solvable Lie algebra. Furthermore, the Zariski-closure of a solvable
group is known to be solvable. Analogous conclusions hold true for nilpotent groups so that the Lie algebra
$\mathfrak{g}$ inherits the algebraic properties of the initial group $G$.

In the context of pseudo-solvable groups, however, the analogous statements cannot immediately be derived.
Regardless of finding a suitable notion of ``pseudo-solvable Lie algebra'', it is not totally clear that the Zariski-closure
of a pseudo-solvable group still is pseudo-solvable since the definition depends on the generating set and
has no a priori implication on commutators of elements that do not belong to the generating set in question. This is an
inconvenient characteristic of the definition of pseudo-solvable subgroups that, ultimately, can only be clarified
through a detailed study of the condition expressed by relation~(\ref{Conjugatingsolvablegroups}).

\begin{obs}
{\rm There is a simple alternative construction of
the Lie algebra $\mathfrak{g}$  associated with a group $G \subset \formdiffn$ which is as follows. Consider the collection
formed by the infinitesimal generators of all elements of $G$ which is clearly contained in $\ghatX$.
The Lie algebra associated with $G$ then coincides with the
Lie algebra generated by this collection of vector fields. Whereas this definition is much simpler to be formulated, it has
the inconvenient of missing the role played by the above mentioned algebraic groups in the whole picture.
In particular, if this alternative definition is adopted from the beginning, then it is not clear
that the Lie algebra associated with a, say, solvable group
must be solvable as well. Although amendments can be made for this deficiency by systematically using various Campbell-Hausdorff
type formulas in a way similar to the use made in this paper (see for example the proof of Lemma~\ref{LemmainvolvingHadamard}),
it is definitely useful to keep both constructions in mind, while being aware of their equivalence.}
\end{obs}

Given a Lie algebra $\mathfrak{g} \subset \ghatX$, the {\it exponential of $\mathfrak{g}$} is the image of
$\mathfrak{g}$ by the exponential map ${\rm Exp}$. In other words, it is the subgroup of $\formdiffn$
consisting of all formal diffeomorphisms $(\exp (tX)(x), \exp (tX) (y))$ where $X \in \mathfrak{g}$ and $t \in \C$.
The proposition below from \cite{ribon} summarizes the main properties of the Lie algebra
$\mathfrak{g}$ associated with a solvable subgroup $G \subset \formdiffn$.

\begin{prop}
\label{frommarteloribon}
{\rm (\cite{ribon})}
Let $G \subset \formdiffn$ be a finitely generated group and denote by $\mathfrak{g} \subset \ghatX$
its associated Lie algebra. Then the following holds:
\begin{enumerate}
  \item For every $X \in \mathfrak{g}$ the exponential ${\rm Exp}\, (tX)$ of $X$ at time-$t$ is contained in
  $\overline{G}$ for every $t \in \C$.

    \item The group $\overline{G}$ is spanned by ${\rm Exp}\, (\mathfrak{g})$. Furthermore ${\rm Exp} :
    \mathfrak{g} \rightarrow \overline{G}$ is a bijection.

  \item Assuming furthermore that $G$ is solvable (resp. nilpotent), then the Lie
  algebra $\mathfrak{g}$ is solvable (resp. nilpotent) as well. Besides the same derived lengths (resp. nilpotent lengths)
  of $\mathfrak{g}$ and of $G$ coincide.\qed
\end{enumerate}
\end{prop}

The statement of Proposition~\ref{frommarteloribon} can be complemented by saying that the derived length
of $G$ as in the statement is at most~$3$, cf. \cite{ribon}. Whenever possible, dealing with Lie algebras is
preferable to working with groups themselves since most calculations become simplified.

Thanks to Proposition~\ref{frommarteloribon}, the formal classification of solvable subgroups of $\formdiffn$
becomes reduced to the classification of solvable Lie algebras of formal vector fields with zero linear parts.
A structural description of these algebras appears in Theorem~6 of \cite{ribon} and it will be detailed later in Section~5.

In closing this section, we would like to draw the reader's attention to a few subtle issues
involving general solvable/nilpotent groups that need to be taken into account in the course of our discussion.
Let then $G$ be a finitely generated group and consider its central and derived series
$\{ C^k G \}$ and $\{ D^k G\}$. Consider also a finite generating set $S$ for $G$ along with the corresponding
sequence of sets $S(k)$ arising from Definition~\ref{definitionpseudosolvablegroup}.

Similarly to the definition of pseudo-solvable groups by means of a generating set, one may wonder about a
hypothetical notion of ``pseudo-nilpotent group'' obtained by defining a suitable sequence of sets $\overline{S} (k)$ by means
of $S$ and requiring this series to degenerate into $\{ {\rm id} \}$ for large~$k$. However, this idea is of no interest
since it is an elementary
algebraic fact, going back to Zassenhaus, that the resulting groups would still be nilpotent. In other words,
a group is known to be nilpotent once we can prove that ``its central series restricted to a finite generating set''
becomes reduced to the identity, see \cite{ghysBSBM} for further details.
A similar property, however, is not shared by solvable groups in general. This difference of behaviors opposing nilpotent
and solvable groups has its roots in the fact
that the quotient of the free group on two generators by its second derived group {\it is not finitely presented},
though it is clearly a step-$2$ solvable group. It is this very issue that makes the notion of pseudo-solvable
group non-trivial. In particular, at combinatorial level, there are pseudo-solvable groups that are not solvable.
On the other hand, it is unclear whether this type of
``pathological'' behavior can still be produced by groups of diffeomorphisms or, in our case,
by subgroups of $\formdiffn$.

Next let $G \subset \formdiffn$ be a finitely generated solvable group. Following Proposition~\ref{frommarteloribon},
denote by $\overline{G}$ its Zariski-closure naturally associated to the Lie
algebra $\mathfrak{g}$ of $G$. The first thing
to be noted is that $\mathfrak{g}$ may be infinite dimensional though it is finitely generated as Lie algebra
(something that does not happen if $G$ is nilpotent). Another subtle point contrasting with experience coming from
the usual theory of algebraic groups, is the fact that the subgroup of unipotent elements of a solvable group need not be nilpotent.
Both phenomena are well illustrated by the following example.

\begin{example}
\label{attention-1}
{\rm Consider the Lie algebra $\mathfrak{g}$ generated by the vector fields $x^2y^2 \partial /\partial y$ and by
$x^2y \partial /\partial y$. The subgroup of $\formdiffn$ obtained by exponentiating $\mathfrak{g}$ is unipotent since all its elements
are tangent to the identity. The dimension of
Lie algebra $\mathfrak{g}$ is infinite and this Lie algebra is not nilpotent.
To check both claims, first note that the commutator between $x^2y^2 \partial /\partial y$ and
$x^2y \partial /\partial y$ has the form $-x^4y^2 \partial /\partial y$. In turn
the commutator of $x^4y^2 \partial /\partial y$
with the vector field $x^2y \partial /\partial y$ gives rise to the vector field $-x^6 y^2 \partial /\partial y$ while
the commutator of $x^6 y^2 \partial /\partial y$ with $x^2y \partial /\partial y$ leads to
$-x^8 y^2 \partial /\partial y$. Continuing inductively, we see that $\mathfrak{g}$ is infinite dimensional
since all the corresponding vector fields have different orders at $(0,0)$. It also immediately follows that
$\mathfrak{g}$ is not nilpotent.

Now note that all the above mentioned vector fields $x^4y^2 \partial /\partial y$, $x^6 y^2 \partial /\partial y$ and so on
belong to $D^1 \mathfrak{g}$. Therefore this derived Lie algebra still is infinite dimensional. Finally, the reader
will easily check that $D^1 \mathfrak{g}$ is also an abelian Lie algebra so that $\mathfrak{g}$ is solvable.}
\end{example}

There is a few further points where the use of the Lie algebra associated with a solvable subgroup of $\formdiffn$
requires special attention. Related to the above mentioned issue concerning unipotent elements, there is the fact
that the first derived group (resp. first derived algebra) of a solvable group (resp. algebra) need not be nilpotent.
Whereas this contrasts again with the case of algebraic groups, the reason behind this phenomenon
can easily be explained as follows.
Given $G \subset \formdiffn$ denote by $\overline{G}$ its Zariski-closure so that
$G$ is solvable if and only if $\overline{G}$ is so and, in this case, both groups have the same derived length.
On the other hand, it is clear that the group $\overline{G}$
may be pictured as the projective limit of a sequence of finite dimensional algebraic groups $G_k$.
Also $D^1 \overline{G}$ is the projective limit of the algebraic
groups $D^1 G_k$ and the groups $D^1 G_k$ are nilpotent. However the
projective limit of a sequence of nilpotent groups need not be nilpotent unless the nilpotence length of
the groups $D^1 G_k$ is uniformly bounded (which is not always the case).

\section{Proof of Theorems~A and~B}
\label{provingtheorems}

Taking for grant Theorem~\ref{commuting9}, we are going to establish Theorems~A and~B in this section. First,
we will exploit Ghys's observation \cite{ghysBSBM} concerning convergence of commutators for diffeomorphisms ``close to the identity''
to establish the following proposition:

\begin{prop}
\label{almostthere}
Suppose that $G \subset \diffCtwo$ is a finitely generated group possessing locally discrete orbits. Then $G$ is solvable.
\end{prop}

\begin{proof}
Consider a finite set $S$ consisting of local diffeomorphisms of $(\C^2, 0)$ that are tangent to the identity.
Assume that the group $G$ generated by the set $S$ is not solvable (at level of groups of germs of diffeomorphisms).
Then consider the pseudogroup generated by $S$ on a sufficiently small neighborhood of the origin.
For the sake of notation, this neighborhood of $(0,0) \in \C^2$
will be left implicit in the course of the discussion.
The proof of the proposition amounts to showing that the resulting pseudogroup
$G$ is {\it non-discrete}\, in the sense that it contains a sequence of elements $g_i$ satisfying the following conditions
(cf. Section~4):
\begin{itemize}
\item $g_i \neq {\rm id}$ for every $i \in \N$. Furthermore, $g_i$ viewed as element of the pseudogroup $G$
is defined on a ball $B_{\epsilon}$ of uniform radius $\epsilon > 0$
around $(0,0) \in \C^2$.

\item The sequence of mappings $\{ g_i \}$ converges uniformly to the identity on $B_{\epsilon}$.
\end{itemize}
Assuming the existence of a sequence $g_i$ as indicated above, there follows that each of the sets
${\rm Fix}_i = \{ p \in B_{\epsilon} \, \; ; \; \, g_i (p) = p \}$ is a proper analytic subset of $B_{\epsilon}$.
For every $N \geq 1$, pose $A_N = \bigcap_{i=N}^{\infty} {\rm Fix}_i$ so that $A_N$ is also a proper analytic
set of $B_{\epsilon}$. Finally, let $F = \bigcup_{N=1}^{\infty} A_N$. The set $F$ has null Lebesgue measure so
that points in $B_{\epsilon} \setminus F$ can be considered. If $p \in B_{\epsilon} \setminus F$ then, by construction,
there is a subsequence of indices $\{ i(j) \}_{j \in \N}$ such that $g_{i(j)} (p) \neq p$ for every
$j$. Since $g_i$ converges to the identity on $B_{\epsilon}$, the sequence $\{ g_{i(j)} (p) \}_{j \in \N}$
converges to~$p$. This shows that the orbit of $p$ is not locally discrete and establishes the
proposition modulo verifying the existence of mentioned sequence $\{ g_i \}$.

The construction of the sequence $\{ g_i \}$ begins with an estimate concerning commutators of diffeomorphisms
that can be found in \cite{lorayandI}, page~159, which is itself similar to another estimate found in \cite{ghysBSBM}.
Let $F_1, F_2$ be local diffeomorphisms (fixing the origin and) defined on the ball $B_r$ of radius $r > 0$ around the
origin of $\C^2$. For small $\delta > 0$, to be fixed later, suppose that
\begin{equation}
\max \, \{ \; \sup_{z \in B_r} \Vert F_1^{\pm 1} (z) -z \Vert \; , \; \sup_{z \in B_r} \Vert F_2^{\pm 1} (z) -z \Vert \; \}
\leq \delta \, . \label{initializing}
\end{equation}
Then, given $\tau$ such that $4\delta + \tau < r$, the commutator $[F_1,F_2]$ is defined on the ball of radius $r -4\delta -\tau$
and, in addition, the following estimate holds:
\begin{equation}
\sup_{z \in B_{r -4\delta -\tau}} \Vert [F_1,F_2] (z) -z \Vert \leq \frac{2}{\tau}  \sup_{z \in B_r}
\Vert F_1 (z) -z \Vert \, . \, \sup_{z \in B_r} \Vert F_2 (z) -z \Vert \, . \label{estimateLorayandI}
\end{equation}

Let us apply the preceding estimate to $S(1)$. Up to conjugating the elements of $S$ by a homothety
having the form $(x,y) \mapsto (\lambda x, \lambda y)$, we can assume that all of them are defined on the unit ball.
Moreover, since these diffeomorphisms are tangent to the identity, the use of a conjugating homothety as above allows
to assume that the diffeomorphisms in question also satisfy Estimate~(\ref{initializing}) for $r=1$ and some arbitrarily small
$\delta >0$ to be fixed later. Setting $\tau = 4\delta$, it then follows that every element
$\overline{g}$ in $S(1)$ is defined on $B_{1- 8 \delta}$ and satisfies
$$
\sup_{z \in B_{1 - 8\delta}} \Vert \overline{g} (z) -z \Vert \leq \delta /2\, .
$$
Next, note that every element in $S(2)$ is the commutator of an element in $S (1)$ and an element in $S \cup S(1)$.
Thus, applying again Estimate~(\ref{estimateLorayandI}) to $r=1-8\delta$, $\delta$ and $\tau = 4 \delta$,
we conclude that every element $\overline{g}$ in $S(2)$ is defined on $B_{1 -16\delta}$.
Furthermore these elements $\overline{g}$ satisfy the estimate
$$
\sup_{z \in B_{1 -8\delta - 8\delta}} \Vert \overline{g} (z) -z \Vert \leq \delta /2^2 \, .
$$
Now, every element in $S(3)$ is the commutator of an element in $S(2)$ and an element in $S(1) \cup S(2)$.
Hence the distance to the identity of any of these elements is bounded by $\delta/2$. Thus, choosing
$\delta_1 =\delta/2$ and $\tau_1 = 4 \delta_1 =2\delta = \tau/2$, we obtain
$$
\sup_{z \in B_{1 - 8 \delta -(8+4)\delta}} \Vert \overline{g} (z) -z \Vert \leq \delta /2^3\, .
$$
For $S(4)$ we have to consider the commutator of an element in $S(3)$ with an element in $S(2) \cup S(3)$.
Now the distance to the identity of any of these diffeomorphisms is bounded by $\delta /2^2$ (on the ball
of radius $1 - 8 \delta -(8+4)\delta$). Hence this time we choose $\delta_2 =\delta_1/2$ and $\tau_2 = 4 \delta_2
=2 \delta_1 = \tau_1 /2$ so as to conclude that the elements in $S(4)$ satisfy
$$
\sup_{z \in B_{1 - 8 \delta -(8+4+2)\delta}} \Vert \overline{g} (z) -z \Vert \leq \delta /2^4\, .
$$
The proof continues inductively as follows: for $i \geq 3$, we divide the previous values of ``$\delta$''
and of ``$\tau$'' by~$2$. For every value of $i \in \N^{\ast}$ the radius chosen is then dictated by the choices
of ``$\delta$'' and ``$\tau$'' according to Formula~(\ref{initializing}). In particular, for every $i \geq 3$ and
$\overline{g}_{(i)}$ in $S(i)$, the local diffeomorphism $\overline{g}_{(i)}$ is defined on the ball of radius
$1 - 8 \delta - \delta \sum_{j=1}^{i-1} 2^{4-i}$. Hence, if $\delta < 1/48$, all the diffeomorphisms $\overline{g}_{(i)}$,
$i \in \N^{\ast}$, are defined on the ball of radius $1/2$.

Similarly, it is also clear that elements in $S (i)$ converge uniformly to the identity
on $B_{1/2}$. In fact, for $i \geq 3$ and $\overline{g}_{(i)} \in S(i)$, we have
$$
\sup_{z \in B(1/2)} \Vert \overline{g} (z) -z \Vert \leq \delta /2^i \, .
$$
Therefore, to obtain the desired sequence $g_i$, it suffices to select for every $i \in \N^{\ast}$
one diffeomorphism $g_i \in S(i)$ which is different from the identity.
In view of Theorem~\ref{commuting9}, the
sequence of sets $S (i)$ never degenerate into the identity alone so that the indicated choice of $g_i$
is always possible. The proof of the proposition is over.
\end{proof}

The above argument suffices to imply Theorem~B.

\begin{proof}[Proof of Theorem~B]
Let then $G \subset \diffCtwo$
be a given non-solvable group and consider again the sets
$S(i)$ constructed above. Without loss of generality we can suppose that the sequence $\{ g_j \}_{j \in \N}$ actually forms
an enumeration of the set $\bigcup_{i=1}^{\infty} [S(i) \setminus \{ {\rm id} \}]$, where ${\rm id}$ stands for the identity map.
In particular, it follows from the proof of Proposition~\ref{almostthere}
that all these local diffeomorphisms $g_j$ are defined and one-to-one on the ball $B(1/2)$ of radius~$1/2$
around the origin.

Now, consider the sets ${\rm Fix}_j$ given as
$$
{\rm Fix}_i = \{ p \in B (1/2) \, ; \; \, g_i (p) =p \} \, .
$$
Let $A_N = \bigcap_{j=N}^{\infty} {\rm Fix}_j$ so that $A_1 \subseteq A_2 \subseteq \cdots \subseteq A_N  \cdots
\subset B(1/2)$. For every fixed value of $N \in \N$, note that the set $A_N$
is a proper analytic subset of $B(1/2)$ since it is given as a countable intersection of proper analytic subsets ${\rm Fix}_j$.
Since the inclusion
$$
\Omega \, (G) \cap B(1/2) \subset \bigcup_{N=1}^{\infty} A_N
$$
clearly holds, the proof of Theorem~B results at once.
\end{proof}

\begin{obs}
\label{ongenericgroups}
{\rm In the introduction we have claimed that a generic $n$-tuple, $n \geq 2$, of
local diffeomorphisms in $\diffCtwo$ generates a subgroup $G \subset \diffCtwo$ whose set of non-recurrent points
$\Omega (G)$ is reduced to the origin. The purpose of this remark is to substantiate this claim by providing an accurate
statement along with a detailed indication of proof.

For this, let $n \geq 2$ be fixed and consider the product $(\diffCtwo)^n$ of $n$ copies of $\diffCtwo$.
Note that $\diffCtwo$, and hence $(\diffCtwo)^n$, can be equipped with the Takens topology discussed in \cite{compositio},
\cite{Jussieu} so that these sets become Baire spaces. Now, there is a
$G_{\delta}$-dense set $\mathcal{U} \subset (\diffCtwo)^n$ whose points are $n$-tuples $(F_1, \ldots , F_n)$
of diffeomorphism in $\diffCtwo$ satisfying the following conditions:
\begin{itemize}
  \item The subgroup $G$ generated by $F_1, \ldots , F_n$ is isomorphic to the free group in~$n$ letters.
  \item Every point $P$ different from the origin is such that its stabilizer in~$G$ is either trivial or
  infinite cyclic.
\end{itemize}
Whereas \cite{compositio}, \cite{Jussieu} deal with local diffeomorphisms of $(\C,0)$, as opposed to local
diffeomorphisms of $(\C^2, (0,0))$,
the above claim is actually much easier to be proved than the analogous statements in \cite{compositio}, \cite{Jussieu}.
In fact, to establish the above assertions {\it every type of perturbation}\, of a initial
$n$-tuple $(F_1, \ldots , F_n)$ can be considered while in \cite{compositio}, \cite{Jussieu} the construction
of perturbations was constrained by the condition that they needed to preserve the analytic conjugation classes
of the generators.

Finally, if $G = \langle F_1, \ldots , F_n \rangle$ is as above, then it is clear that the set $A_N$ is reduced to the origin
for every~$N \in \N$. Therefore the set $\Omega \, (G)$ of non-recurrent points must be reduced to the origin as well.}
\end{obs}

\vspace{0.1cm}

In what precedes, the condition of having a group $G$ constituted by diffeomorphisms tangent to the identity
was important to fix an initial set of local diffeomorphisms sufficiently close to the
identity on a fixed domain (the unit ball), cf. the proof of Proposition~\ref{almostthere}.
Convergence of iterated commutators no longer holds when we work with diffeomorphisms that are allowed to have arbitrary
linear parts. The proof of Theorem~A will thus require a more elaborated discussion.
We begin by pointing out another consequence of the proof of Proposition~\ref{almostthere} that will be useful
for the proof of Theorem~A. This begins as follows.

\begin{lemma}
\label{pseudosolvablegl2C}
A finitely generated pseudo-solvable subgroup of ${\rm GL}\, (2, \C)$ is necessarily solvable.
\end{lemma}

\begin{proof}
The analogous statement for subgroups of ${\rm GL}\, (2, \R)$ was proven in \cite{ghysBSBM}; the same argument
applies to ${\rm GL}\, (2, \C)$.
\end{proof}

Now we have:

\begin{lemma}
\label{convergingnottangentidentity}
There is a neighborhood $\mathcal{U}$ of the identity matrix in ${\rm GL}\, (2, \C)$ with the following property:
assume that $\Gamma \subset {\rm GL}\, (2, \C)$ is a non-solvable group generated by finitely many elements $\gamma_1, \ldots ,\gamma_s$
belonging to $\mathcal{U}$. Assume also that $G \subset \Diffgentwo$ is generated by local diffeomorphisms $f_1, \ldots ,
f_s$ with $D_{(0,0)} f_i = \gamma_i$ for every $i=1, \ldots ,s$. Then there is a neighborhood $U$ of $(0,0) \in \C^2$ and
a sequence of elements $\{ g_i\}$ in the pseudogroup $G$ generated by $f_1, \ldots ,
f_s$ satisfying the following conditions:
\begin{itemize}
  \item For every $i \in \N$, $g_i$ is defined on all of $U$ and $g_i \neq {\rm id}$.
  \item The sequence $g_i$ converges uniformly to the identity on $U$.
\end{itemize}
\end{lemma}

\begin{proof}
Without loss of generality, we can assume that $f_1, \ldots , f_s$ are defined on the unit ball of $\C^2$.
Moreover, by setting $S= \{ f_1, \ldots , f_s \}$, the corresponding sequence of sets $S(k)$ indicated in
Definition~\ref{definitionpseudosolvablegroup} never degenerates into $\{ {\rm id} \}$. In fact, owing to
Lemma~\ref{pseudosolvablegl2C}, for every $k$
there is an element in $S(k)$ whose derivative at the origin is different from the identity.

On the other hand, according to the proof of Proposition~\ref{almostthere}, there is $\delta >0$ such that
the following holds: given a finite set $S=S(0)$ consisting of local diffeomorphisms that are $\delta$-close
to the identity on the unit ball, then every sequence of elements
$\{ \overline{g}_k\}$, with $\overline{g}_k \in S(k)$ for every $k$, converges uniformly to the identity on the ball of
radius~$1/2$ (and in particular all these diffeomorphisms are defined on the ball in question).
Fixed this value of
$\delta$, the neighborhood $\mathcal{U}$ of the identity matrix in ${\rm GL}\, (2, \C)$ is
determined by letting
$$
\mathcal{U} = \{ \gamma \in  {\rm GL}\, (2, \C) \; ; \; \; \sup_{z \in B_1} \vert \gamma.z -z \vert <
\delta/2 \} \, .
$$
Hence, if $\{ \gamma_1, \ldots ,\gamma_s \} \subset \mathcal{U}$, then up to changing coordinates by means of a
suitable homothety, we obtain
$$
\sup_{z \in B_1} \vert f_i (z) -z \vert < \delta
$$
for every $i=1, \ldots, s$. The lemma follows at once.
\end{proof}

We can now start the approach to the proof of Theorem~A. Let $\rho$ be the homomorphism from $G$
to ${\rm GL}\, (2,\C)$ assigning to an element $g \in G$ its Jacobian matrix at the origin. Denoting by $\Gamma \subset
{\rm GL}\, (2,\C)$ the image of $\rho$, consider the short exact sequence
\begin{equation}
0 \longrightarrow G_0 = {\rm Ker}\, (\rho) \longrightarrow G \stackrel{\rho}{\longrightarrow}
\Gamma \longrightarrow 0 \, . \label{theshortexactsequence1}
\end{equation}
The kernel $G_0$ of $\rho$ consists of those elements in $G$ that are tangent to the identity. Since $G$, and hence $G_0$, has
locally discrete orbits, it follows from Proposition~\ref{almostthere} that $G_0$ is solvable.

At this point, we remind the reader that the terminology {\it virtually solvable}\, is used in this paper in a sense
slightly stronger than its most common use in the literature, namely in this paper
a group is said to be {\it virtually solvable}\, if it contains a normal, solvable subgroup with finite index.
In other words, the solvable group of finite index is also required to be normal.
In the sequel, the phrase {\it virtually solvable group}\, will always be used in this stronger sense.
With this terminology, we have:

\begin{lemma}
\label{enoughtocheckGamma}
To prove Theorem~A, it suffices to check that the group
$\Gamma \subset {\rm GL}\, (2,\C)$ is virtually solvable.
\end{lemma}

\begin{proof}
Suppose that $\Gamma$ is virtually solvable so that there is a normal, solvable subgroup $\Gamma_0 \subset \Gamma$ with
finite index. Denote by $\xi$ the natural (projection) homomorphism from $\Gamma$ onto the finite group
$\Gamma /\Gamma_0$ and consider the homomorphism $\xi \circ \rho : G \rightarrow \Gamma /\Gamma_0$. The kernel
${\rm Ker}\, (\xi \circ \rho)$ of $\xi \circ \rho$ is clearly a normal subgroup of $G$ having finite index since
$\Gamma /\Gamma_0$ is finite. Moreover this kernel is the extension of a solvable
group by another solvable group (namely $\Gamma_0$ and $G_0$) and hence it is itself a solvable group. Thus
$G$ is virtually solvable is Theorem~A is proved.
\end{proof}

The drawback of our notion of virtually solvable group lies in the fact that it prevents us from directly applying
Tits's theorem to prove Lemma~\ref{producinghyperbolicsaddles} below; cf. \cite{tits}, \cite{delaharpe}. Rather than
trying to refine Tits argument, we will directly produce a self-contained proof of the virtually solvable character
of $\Gamma$.

Before proceeding further, it is convenient to make a few general remarks aiming at showing that a given
group $\Gamma \subset {\rm GL}\, (2,\C)$ is virtually solvable. First, we
consider the surjective projection homomorphism $\sigma : {\rm GL}\, (2,\C) \rightarrow {\rm PSL}\, (2,\C)$
which realizes ${\rm GL}\, (2,\C)$ as a central extension of ${\rm PSL}\, (2,\C)$, i.e. the kernel of $\sigma$
is contained in the center of ${\rm GL}\, (2,\C)$. The restriction of $\sigma$ to ${\rm SL}\, (2,\C) \subset
{\rm GL}\, (2,\C)$ will still be denoted by $\sigma$ and it also realizes ${\rm SL}\, (2,\C)$ as a central
extension of ${\rm PSL}\, (2,\C)$.
Now, given a subgroup $\Gamma \subset {\rm GL}\, (2,\C)$, in order to show that $\Gamma$ is virtually solvable,
it suffices to check that $\sigma (\Gamma) \subset {\rm PSL}\, (2,\C)$ is virtually solvable. This is similar
to Lemma~\ref{enoughtocheckGamma}:
given a normal, solvable subgroup $H$ of $\sigma (\Gamma)$ with finite index, denote by $\pi$ the canonical projection
$\pi : \sigma (\Gamma) \rightarrow \sigma (\Gamma)/H$ and consider the homomorphism $\pi \circ \sigma$ restricted
to $\Gamma$. The kernel of $\pi \circ \sigma$ is clearly a normal subgroup of $\Gamma$ having finite index.
The claim then follows from observing that ${\rm Ker}\, (\pi \circ \sigma)$ must be a solvable group since
$H$ is solvable and $\sigma$ realizes ${\rm GL}\, (2,\C)$ as a central extension of ${\rm PSL}\, (2,\C)$.

Let us now go back to $\Gamma \subset {\rm GL}\, (2,\C)$ which is the image by $\rho$ of the group $G \subset \Diffgentwo$.
While $\Gamma$ is a subgroup of ${\rm GL}\, (2,\C)$, its standard action on $(\C^2, 0)$ has little to do with
the action of $G$. In fact, if $\gamma$ is an element of $\Gamma$, then $\gamma$ is simply the derivative at the
origin of an actual element $g \in G$ and it is $g$, rather than $\gamma$, that acts on $(\C^2, 0)$.
Thus, the effect of the non-linear terms in $g$ must be taken into account in the following discussion.
In this direction, we have the following:

\begin{lemma}
\label{producinghyperbolicsaddles}
Assume that the group $\Gamma$ is not virtually solvable. Then at least one of the following conditions holds:
\begin{enumerate}

\item there is a diffeomorphism $g \in G$ whose
derivative at $(0,0) \in \C^2$ is hyperbolic saddle (i.e., its eigenvalues $\lambda_1, \lambda_2$ satisfy
$0 < \vert \lambda_1 \vert < 1 < \vert \lambda_2 \vert$).

\item There is a sequence of elements $g_i$ in the pseudogroup $G$ satisfying the conclusions of
Lemma~\ref{convergingnottangentidentity}.
\end{enumerate}
\end{lemma}

\begin{proof}

Consider again the projection homomorphism $\sigma : {\rm GL}\, (2,\C) \rightarrow {\rm PSL}\, (2,\C)$
as well as its restriction to ${\rm SL}\, (2,\C) \subset {\rm GL}\, (2,\C)$ which realizes both
${\rm GL}\, (2,\C)$ and ${\rm SL}\, (2,\C)$ as central
extensions of ${\rm PSL}\, (2,\C)$. As noted above, $\sigma (\Gamma) \subset {\rm PSL}\, (2,\C)$ is not virtually
solvable since $\Gamma$ is by assumption not virtually solvable.

To describe our strategy for proving Lemma~\ref{producinghyperbolicsaddles}, we first consider the elements
of ${\rm PSL}\, (2,\C)$ classified into elliptic, parabolic and loxodromic ones; see \cite{apanasov}, \cite{ford}.
The reader will note that an element of ${\rm GL}\, (2,\C)$ having determinant equal to~$1$ and
projecting on a loxodromic element of ${\rm PSL}\, (2,\C)$
must be the differential of an element in $G$ exhibiting a hyperbolic saddle at the origin of $\C^2$. In
particular, Condition~(1) in the statement holds provided that we can find $\gamma \in D^1 \Gamma$ such that
$\sigma (\gamma)$ is a loxodromic element of ${\rm PSL}\, (2,\C)$.

A similar observation concerning Condition~(2) in Lemma~\ref{producinghyperbolicsaddles} is as follows.
Since $\Gamma$ is a countable subgroup of the Lie group ${\rm GL}\, (2,\C)$, the closure $\overline{\Gamma}$ of $\Gamma$
can be considered and $\overline{\Gamma} = \Gamma$ if and only if $\Gamma$ is discrete. Moreover, unless $\Gamma$ is discrete,
$\overline{\Gamma}$ is itself a real Lie group admitting a non-trivial real Lie algebra. If this real Lie algebra
is not solvable then $\overline{\Gamma}$ contains elements $\gamma_1, \ldots ,\gamma_s$ satisfying the assumptions
of Lemma~\ref{convergingnottangentidentity}. The same conclusion holds for the group $\Gamma$ since $\Gamma$ is dense
in $\overline{\Gamma}$ and the condition for a finite set to generate a non-solvable subgroup of ${\rm GL}\, (2,\C)$ is open.
Therefore, Lemma~\ref{convergingnottangentidentity} ensures that $G$ contains a sequence of elements satisfying
Condition~(2) in the statement provided that the real Lie algebra associated with $\overline{\Gamma}$ is not solvable.

Summarizing what precedes, our proof of Lemma~\ref{producinghyperbolicsaddles} is organized as follows. We assume aiming at
a contradiction that no element in $D^1 \Gamma$ projects on a loxodromic element of ${\rm PSL}\, (2,\C)$. Furthermore,
we also assume that the real Lie algebra associated with $\overline{\Gamma}$ is solvable where it is understood
that this Lie algebra is trivial (and hence solvable) if $\Gamma$ is discrete. From these two assumptions, we
shall conclude that $\sigma (\Gamma)$ must be virtually solvable hence deriving a contradiction with the assumption
that $\Gamma$ {\it is not}\, virtually solvable. This contradiction will then complete the proof of the lemma.

To implement the above mentioned strategy, it is natural to split the discussion into two cases according to whether
or not $\sigma (\Gamma) \subset {\rm PSL}\, (2,\C)$ is a discrete subgroup of ${\rm PSL}\, (2,\C)$.

\noindent {\bf Case A}: Suppose that $\sigma (\Gamma )$ is not discrete.

The closure of
the subgroup $\sigma (\Gamma )$ in ${\rm PSL}\, (2,\C)$ possesses a non-trivial real Lie algebra which will
be denoted by $\mathcal{L}_{PSL}$. We then have:

\noindent {\it Claim 1}. The algebra $\mathcal{L}_{PSL}$ is solvable.

\noindent {\it Proof of the Claim 1}. Assume aiming at a contradiction that $\mathcal{L}_{PSL}$ is not solvable.
Then the Lie algebra $D^1 \mathcal{L}_{PSL}$ associated with $D^1 (\overline{\sigma (\Gamma)})$ is non-trivial
and not solvable. This Lie algebra is, however, isomorphic to the Lie algebra $D^1 \mathcal{L}$ associated with
$D^1 (\overline{\Gamma}) \subset {\rm SL}\, (2,\C)$ since ${\rm SL}\, (2,\C)$ is a double covering of
${\rm PSL}\, (2,\C)$. Thus the Lie algebra associated with $D^1 (\overline{\Gamma})$, and hence the Lie algebra
associated with $\overline{\Gamma}$, is not solvable. The resulting contradiction establishes the claim.\qed

The Lie algebra $\mathcal{L}_{PSL}$ is therefore solvable. Clearly $\mathcal{L}_{PSL}$ is also
invariant by all elements in $\sigma (\Gamma)$.
In the sequel ${\rm PSL}\, (2,\C)$ will often be identified with the corresponding automorphism group of the Riemann sphere
$S^2$. With this identification, the exponential of the Lie algebra $\mathcal{L}_{PSL}$
is a solvable, connected subgroup ${\rm Exp}\, (\mathcal{L}_{PSL})$
of the automorphism group of the Riemann sphere.
This subgroup is not reduced to the identity since
$\mathcal{L}_{PSL}$ is not trivial which, in fact, ensures that ${\rm Exp}\, (\mathcal{L}_{PSL})$
must contain a real one-parameter
subgroup (i.e. a ``flow''). Note that, in principle, $\sigma (\Gamma)$ need not be connected so that we cannot yet derive a
contradiction. It turns out, however, that the solvable, connected subgroup ${\rm Exp}\, (\mathcal{L}_{PSL})$
must have a fixed point (as an elementary particular case of Borel's theorem, see \cite{who??}).

On the other hand, recall that a real one-parameter subgroup of ${\rm PSL}\, (2,\C)$
must have at least one and at most two fixed points in $S^2$. In other words, the set of fixed points of
${\rm Exp}\, (\mathcal{L}_{PSL})$ is non-empty and contains at most two points. Now note that the set formed
by these fixed points is necessarily invariant under $\sigma (\Gamma)$ since $\mathcal{L}_{PSL}$ is invariant under
$\sigma (\Gamma)$ (alternatively ${\rm Exp}\, (\mathcal{L}_{PSL})$ is normal in $\sigma (\Gamma)$ since it coincides
with the connected component of $\sigma (\Gamma)$ containing the identity). Summarizing what precedes,
up to passing to (necessarily normal) subgroup of $\sigma ( \Gamma)$
having index~$2$, we can assume the existence of $p \in S^2$ fixed by all elements in $\sigma ( \Gamma)$. The subgroup
of ${\rm PSL}\, (2,\C)$ fixing a given point in $S^2$ is however conjugate to the affine group of $C$. In particular all
these groups are solvable. Summarizing either $\sigma (\Gamma)$ embeds into a solvable group or it has a index~$2$ subgroup
that does. Since index~$2$ subgroups are always normal, we conclude that $\sigma (\Gamma)$
is necessarily virtually solvable (in fact it contains a normal, solvable subgroup of index~$2$).
The resulting contradiction establishes the lemma is this first case.

\noindent {\bf Case B}: Suppose that $\sigma (\Gamma)$ is discrete.

First of all, we can assume that $\sigma (\Gamma)$ is not a finite group. In fact, consider $\sigma$ restricted
to $\Gamma$ and its corresponding kernel which is an abelian group. If $\sigma (\Gamma)$ is finite, then $\Gamma$
becomes a finite extension of an abelian (normal) group and the desired contradiction arises immediately.
Therefore $\sigma (\Gamma)$ is assumed to be a finitely generated, infinite group in what follows.
Being finitely generated and infinite, a result of Schur \cite{schur} asserts that $\sigma (\Gamma)$ contains an
element of infinite order which will be denoted by $\sigma (\gamma)$ for some $\gamma \in \Gamma$. The element
$\sigma (\gamma)$ is either parabolic or loxodromic since $\sigma (\Gamma)$ is supposed to be discrete.

We consider also the subgroup $\sigma (D^1 \Gamma) = D^1 (\sigma (\Gamma))$ which is discrete since it is contained
in the discrete group $\sigma (\Gamma)$. In particular, it is a Kleinian group. Recalling that this Kleinian group
is assumed not to contain loxodromic elements, there follows that it must be an elementary Kleinian group; see \cite{apanasov}.
Moreover we have:

\noindent {\it Claim 2}. If the group $\sigma (D^1 \Gamma)$ is not solvable then
either this group is finite or contains an element of infinite order.

\noindent {\it Proof of Claim~2}. The only difficulty in applying again Schur lemma \cite{schur} to conclude the
statement is to ensure that $\sigma (D^1 \Gamma) \subset {\rm PSL}\, (2,\C)$ is finitely generated
(our assumption only ensures that $\Gamma$
is finitely generated). To overcome this difficulty suppose that $\sigma (D^1 \Gamma)$ is an infinite group
and consider an enumeration $\gamma_1, \gamma_2, \ldots$ of its elements. Consider then the groups
$(\sigma (D^1 \Gamma))_n$ generated by $\gamma_1, \ldots , \gamma_n$. All the groups $(\sigma (D^1 \Gamma))_n$ are finitely
generated so that Schur lemma applies to ensure the existence of an element of finite order unless all these groups
are finite. We assume then that this is the case.

Next we recall that finite subgroups of ${\rm PSL}\, (2,\C)$ were classified since Klein, and apart from cyclic
groups and dihedral groups, there are only finitely many of them (in correspondence with the platonic solids; see
for example \cite{finitegroups}). Thus, for $n$ large enough, every group $(\sigma (D^1 \Gamma))_n$ must be either cyclic or
dihedral. Therefore all these groups are abelian or metabelian, i.e. their derived length is at most~$2$.
This clearly implies that $\sigma (D^1 \Gamma)$
is solvable. The resulting contradiction proves the claim.\qed

Naturally we can assume $\sigma (D^1 \Gamma)$ to be non-solvable otherwise $\sigma (\Gamma)$ is solvable itself.
Assume also that $\sigma (\Gamma)$ is not finite and
consider an element in $\sigma (D^1 \Gamma)$ having infinite order. This element
must be parabolic since
elliptic and loxodromic elements are excluded (the existence of an elliptic element with infinite order would force the
group $\sigma (D^1 \Gamma)$ to be non-discrete). Elementary Kleinian groups containing parabolic elements are described
in \cite{ford} and these groups possess a fixed point in $S^2$. Therefore they are solvable as they can be realized
as a subgroup of the affine group of $\C$. Thus we conclude $\sigma (D^1 \Gamma)$ must be finite unless
$\sigma (D^1 \Gamma)$, and hence $\sigma (\Gamma)$, is solvable.
Since $\sigma (D^1 \Gamma)$ is finite, there follows that $\sigma (\Gamma)$ is amenable as a finite
(and hence amenable) extension of an abelian (and hence amenable) group, cf. \cite{eymard}.

Summarizing what precedes, the group $\sigma (\Gamma)$ is amenable. Moreover it was already seen that $\sigma (\Gamma)$
is a finitely generated, infinite group so that
Schur lemma ensures it must contain an element of infinite order
$\sigma (\gamma)$. In turn, $\sigma (\gamma)$
is either parabolic or loxodromic since $\sigma (\Gamma)$ is assumed to be discrete.
To complete the proof of the lemma, we now
proceed as follows. The action of $\sigma (\Gamma)$ on $S^2$ must preserve a probability measure $\mu$
since this group is amenable, see \cite{eymard}. In particular,
$\mu$ must be invariant by $\sigma (\gamma)$. Now we have:

\noindent $\bullet$ Suppose that $\sigma (\gamma)$ is parabolic.

The only probability measure preserved by a parabolic element (with infinite order) is the Dirac mass concentrated
at the unique fixed point for the element in question. In other words, there is a point $p \in S^2$ which
is fixed by the entire group $\sigma (\Gamma)$. Therefore $\sigma (\Gamma)$is solvable as it is conjugate to
a subgroup of the affine group of $\C$. The desired contradiction follows at once.

\noindent $\bullet$ Suppose that $\sigma (\gamma)$ is loxodromic.

A loxodromic element of ${\rm PSL}\, (2,\C)$ has exactly two fixed points $p_1$ and $p_2$ in $S^2$. Furthermore, the only
probability measures invariant under these elements are the convex combinations of Dirac masses concentrated at
$p_1$ and at $p_2$. Hence the set $\{ p_1, p_2 \}$ must be invariant by $\sigma (\Gamma)$. If one of these two
points is fixed by all of $\sigma (\Gamma)$, then we conclude as in the previous case that $\sigma (\Gamma)$, and hence
$\Gamma$, is solvable. A contradiction then results.

Finally, in the general case, $\sigma (\Gamma)$ contains a normal subgroup $[\sigma (\Gamma)]_{p_1}$ of index two
fixing $p_1$. Again $[\sigma (\Gamma)]_{p_1}$ must be solvable. Thus $\sigma (\Gamma)$ is virtually solvable.
This implies that $\Gamma$ is virtually solvable and provides the final contradiction ending the proof
of Lemma~\ref{producinghyperbolicsaddles}.
\end{proof}

\begin{proof}[Proof of Theorem~A]
Consider again the short exact sequence~(\ref{theshortexactsequence1}). To prove Theorem~A we will assume that
$\Gamma$ is not virtually solvable and derive from this the existence of recurrent points. Since $\Gamma$
is not virtually solvable, the alternative provided by Lemma~\ref{producinghyperbolicsaddles} holds. However,
if $G$ actually contains a sequence $\{ g_i \}$ of elements as in Condition~(2), then the existence of the mentioned
recurrent points follows at once from the argument employed in the proof of Theorem~B. Therefore, in order to prove
Theorem~A, we can assume without loss of generality the existence of an element $g \in G$ whose derivative $D_0 g$ at the origin
is a hyperbolic saddle as indicated in Condition~(1) of Lemma~\ref{producinghyperbolicsaddles}. In fact, we can assume
that $\sigma (\Gamma )$ is a non-elementary Kleinian group. Moreover we can also assume that
the Jacobian determinant of $D_0 g$ equals~$1$ since the preceding Lemma~\ref{producinghyperbolicsaddles}
actually ensures that the element in Condition~(1) can be assumed to belong to $D^1 \Gamma$. In this respect, however,
the only role played by the fact that the Jacobian determinant of $D_0 g$ equals~$1$ in the discussion below consists of
helping us to abridge notation,
as the reader will not fail to notice.

The eigenvalues of $D_0 g$ at the origin are then denoted by $\lambda$ and by $\lambda^{-1}$, with $\vert \lambda \vert > 1$.
It follows that $g$ has a hyperbolic fixed point at the origin with stable and unstable manifolds, $W^s_{g}, \, W^u_{g}$,
having complex dimension~$1$ and intersecting transversely at $(0,0) \in \C^2$. Fix then a {\it closed annulus}\, $A^s \subset W^s_{g}$
(resp. $A^u \subset W^u_{g}$) with radii $r_2 > r_1 > 0$ such that every point $p \in W^s_{g} $ (resp. $p \in W^u_{g} $)
possesses an orbit by $g$ non-trivially intersecting $A^s$ (resp. $A^u$).

Given a point $p$ in a fixed neighborhood $U$ of the origin, denote by
$\calO_G (p)$ the orbit of $p$ by the pseudogroup $G$. Similarly, let ${\rm Acc}_p (G)$ denote the set of {\it ends}\,
of $\calO_G (p)$. To define this set, we consider the closure $\overline{\calO_G (p)}$ of the orbit $\calO_G (p)$.
We then set ${\rm Acc}_p (G) = \overline{\overline{\calO_G (p)} \setminus \calO_G (p)}$, i.e. ${\rm Acc}_p (G)$ is the
closure of the difference $\overline{\calO_G (p)} \setminus \calO_G (p)$. In particular, if $p$ is a recurrent point
then $p \in {\rm Acc}_p (G)$. Furthermore ${\rm Acc}_p (G) = \emptyset$ provided that $\calO_G (p)$ is finite.
Clearly ${\rm Acc}_p (G)$ is closed and invariant by $G$ (viewed as pseudogroup). The following claim is
the key for the proof of Theorem~A.

\vspace{0.1cm}

\noindent {\it Claim}. For every point $p \in A^s$, the closed set $A^s \cap {\rm Acc}_p (G)$ is not empty.

\vspace{0.1cm}

Note that the claim does not immediately imply Theorem~A for it does not assert that $p$ itself belongs to
$A^s \cap {\rm Acc}_p (G)$. However, if this were the case, then clearly the orbit of $p$ would be recurrent and the
proof of Theorem~A would follow. However, by resorting to a standard
application of Zorn Lemma, the above claim can still be used to prove Theorem~A. Let us first provide the details of this argument
and then go back to the proof of the claim. To begin with, if $K \subseteq A^s$ is a non-empty closed set,
we shall say that $K$ is {\it relatively invariant}\, by the pseudogroup $G$ if, for every point $p \in K$
and every point $q \in A^s \cap {\rm Acc}_p (G)$, the point $q$ lies in $K$ as well. Next, let
$\mathfrak{C}$ denote the collection of non-empty closed sets in $A^s$ that are relatively invariant by the
pseudogroup $G$. The above claim ensures that the collection $\mathfrak{C}$ is not empty. In fact, $A^s \cap {\rm Acc}_p (G)$
in a non-empty set relatively invariant under $G$, and thus $A^s \cap {\rm Acc}_p (G)$ belongs to $\mathfrak{C}$
for every $p \in A^s$. Now, let the collection $\mathfrak{C}$ be endowed with the partial order defined by
inclusion. Finally, given a sequence $K_1 \supset K_2 \supset \ldots$ of sets in $\mathfrak{C}$, the
intersection $K_{\infty}=\bigcap_{i=1}^{\infty} K_i$ is non-empty since each $K_i$ is compact (closed and contained
in the compact set $A^s$). The set $K_{\infty}$ is clearly closed and relatively invariant by $G$ so that
it belongs to $\mathfrak{C}$. Moreover we have $K_{\infty} \subset K_i$ for every~$i$ i.e., in terms of the fixed
partial order $K_{\infty}$ is smaller than $K_i$ for every~$i$. According to Zorn Lemma, the collection $\mathfrak{C}$
contains minimal elements, so that we can consider a minimal element $K$. Choose then $q \in K$ and consider
the non-empty set $A^s \cap {\rm Acc}_q (G)$. If $q \not\in {\rm Acc}_q (G)$, then $A^s \cap {\rm Acc}_q (G)$
would be an element of $\mathfrak{C}$ strictly smaller than $K$. The resulting contradiction shows that
$q \in A^s \cap {\rm Acc}_q (G)$ and finishes the proof of Theorem~A.\qed

It only remains to prove the Claim.

\vspace{0.1cm}

\noindent {\it Proof of the Claim}. Recall that $A^s \subset W^s_{g}$
(resp. $A^u \subset W^u_{g}$) is an annulus such that every $p \in W^s_{g} $ (resp. $p \in W^s_{g} $)
possesses an orbit by $g$ non-trivially intersecting $A^s$ (resp. $A^u$).

Now consider another element $\overline{g} \in G$ whose Jacobian matrix at the origin defines a hyperbolic
saddle with determinant equal to~$1$.
Again stable and unstable manifolds for $\overline{g}$ will respectively be denoted by $W^s_{\overline{g}},
\, W^u_{\overline{g}}$. Since a (non-elementary) Kleinian group contains ``many'' loxodromic elements (including conjugates of $g$),
the element $\overline{g}$ can be chosen so that all the four invariant manifolds
$W^s_{g}, \, W^u_{g}, \, W^s_{\overline{g}}, \, W^u_{\overline{g}}$ intersect pairwise transversely at the origin.
The previously fixed annuli $A^s \subset W^s_{g}$ and $A^u \subset W^u_{g}$ will be denoted in the sequel
by $A^s_{g}$ and $A^u_{g}$. An annulus $A^s_{\overline{g}} \subset W^s_{\overline{g}}$
(resp. $A^s_{\overline{g}} \subset W^s_{\overline{g}}$) with analogous properties concerning $\overline{g}$ is
also fixed. To prove the claim it suffices to check that every point $p$ in $A^s_{g}$ is such that
$A^u_{\overline{g}} \cap {\rm Acc}_p (G) \neq \emptyset$. Indeed, let
$p^{\ast} \in A^u_{\overline{g}}$ be a point in $A^u_{\overline{g}} \cap {\rm Acc}_p (G)$. The analogue
argument changing the roles of $g, \, \overline{g}$ and replacing them by their inverses, will ensure that
$A^s_{g} \cap {\rm Acc}_{p^{\ast}} (G) \neq \emptyset$. Since $p^{\ast}$ lies in ${\rm Acc}_p (G)$ and
this set is invariant under the pseudogroup $G$, it will follow that $A^s_g \cap {\rm Acc}_p (G) \neq \emptyset$ as desired.

Finally to check that $A^u_{\overline{g}} \cap {\rm Acc}_p (G) \neq \emptyset$ for every point $p \in A^s_{g}$,
we proceed as follows. Consider local coordinates $(x,y)$ about the origin of $\C^2$ so that
$\{ x=0\} \subset W^u_{\overline{g}}$ and $\{ y=0\} \subset W^s_{\overline{g}}$.
Recall that $W^s_{g}$ is smooth and intersects the coordinate axes transversely at the origin.
Since this intersection is transverse, we can assume
that it is the only intersection point of $W^s_{g}$ with the coordinate axes. In particular, a point
$p \in A^s_{g}$ has coordinates $(u,v)$ with $u.v \neq 0$. By iterating $g$, we can find points
$p_n = (u_n , v_n) =g^n (p) \in \C^2$ such that $\vert u_n \vert \rightarrow 0$ and
$$
\frac{1}{C} \vert u_n \vert \leq \vert v_n \vert \leq C \vert u_n \vert \, ,
$$
for some uniform constant $C$ related to the ``angles'' between $W^s_{g}$ and the coordinate axes at the origin.
Now, for every $n$, consider the points of the form $\overline{g} (p_n) , \ldots , \overline{g}^{l(n)} (p_n)$
where $l(n)$ is the smallest positive integer for which the absolute value of the second component of
$\overline{g}^{l(n)} (p_n)$ is greater than $\sup_{z \in A^u_{\overline{g}}} \vert z \vert$. The integer $l(n)$ exists since
$\overline{g}$ has a hyperbolic fixed point at the origin and the action of $\overline{g}$ on $p_n$ is
such that the first coordinate becomes smaller and smaller while the second coordinate gets larger and larger.
Now it is clear that the closure of the set
$\bigcup_{n=1}^{\infty} \{ \overline{g} (p_n) , \ldots , \overline{g}^{l(n)} (p_n) \}$
intersects $A^u_{\overline{g}}$ non-trivially and this ends the proof of the Claim. The proof of Theorem~A
is completed as well.
\end{proof}

Let us close this section by showing how to extend Theorems~A and~B to encompass groups that are infinitely
generated.

\begin{theorem}
\label{infinitelygeneratedgroups}
Theorem~A and Theorem~B remain valid for infinitely generated groups.
\end{theorem}

\begin{proof}
We begin by justifying the case of Theorem~B. We consider a finitely generated subgroup $H$ of
$G$. If $H$ is not solvable, then $H$, and in particular $G$, has recurrent orbits away from a countable union
of proper analytic sets. Thus we can assume that $H$ is solvable. Owing to the classification
of solvable groups provided in \cite{ribon} (cf. also Section~5), there follows
that $D^3H = \{ {\rm id} \}$. In other words, the third derived group of every finitely generated subgroup
of $G$ is trivial. We then conclude that $D^3 G$ must be reduced to the identity as well and this yields
a contradiction proving our claim.

As to Theorem~A, we need to revisit the argument provided above. Given a finitely generated subgroup
$H$ of $\Diffgentwo$, let $\Gamma_H \subset {\rm GL}\, (2,\C)$ denote the image of $H$ by $\rho$. Assuming
that $H$ has locally discrete orbits, Theorem~A ensures that $\Gamma_H$ possesses a normal, solvable subgroup
$\Gamma_{H_0}$ having finite index. Moreover, we have seen that $H$ itself possesses a normal, solvable subgroup $H_0$
whose index equals the index of $\Gamma_{H_0}$ in $\Gamma_{H}$. A careful reading of the proof of
Lemma~\ref{producinghyperbolicsaddles} shows that the group $\Gamma_{H}$ possesses an index~$2$ (normal)
solvable subgroup unless $\Gamma_{H}$ is a finite group.
As already mentioned, bar abelian and metabelian groups, there are only a finite number
of finite subgroups of ${\rm PSL}\, (2,\C)$, see \cite{finitegroups}. A similar remark applies to subgroups
of ${\rm GL}\, (2,\C)$. Putting everything together, we conclude that all finitely
generated groups $H \subset \Diffgentwo$ having locally discrete orbits possess a normal, solvable subgroup
whose index is finite and, indeed, uniformly bounded by some constant $C$ whose exact value is not important for us.
On the other, again owing to the description of solvable groups provided in \cite{ribon}; cf. Section~5,
we know that every solvable subgroup of $\Diffgentwo$ has derived
length bounded by~$5$. Thus, we finally conclude that every finitely generated subgroup $H \subset G$
possesses a normal subgroup with index bounded by~$C$ which has derived length no greater than~$5$. There follows
that $G$ itself possesses a normal subgroup of index less than~$C$ whose derived length is no greater than~$5$.
In particular this subgroup is solvable and the statement results.
\end{proof}

\section{Discrete subgroups of $\Diffgentwo$, examples and complements}
\label{section.examples}

Consider a finitely generated subgroup $G \subset \diffn$. Up to
choosing representatives for elements of $G$ in some finite generating set, we let $G$ be identified with
a pseudogroup of local diffeomorphisms fixing the origin of $\C^n$. The following definition is very natural.

\begin{defi}
\label{definitiondiscretegroup}
The group $G \subset \diffn$ is said to be non-discrete if there is an open neighborhood $V \subseteq \C^n$
of the origin and a sequence of elements $\{ g_j \} \subset G$ satisfying the following conditions:
\begin{enumerate}
  \item For every $j \in \N$, the set $V$ is contained in the domain of definition of $g_j$ viewed as an element
  of the pseudogroup~$G$.
  \item For every $j \in \N$, the restriction of $g_j$ to $V$ does not coincide with the identity.
  \item The sequence $\{ g_j\}$ converges uniformly to the identity on compact parts of~$V$.
\end{enumerate}
\end{defi}

The above definition clearly makes sense in terms of germs since it does not depend on the set of representatives
chosen. The definition can be made more global at the expenses of considering pseudogroups acting on open sets
of $\C^n$, whether or not the origin is fixed. In this sense, the pseudogroup $G$ generated by a (finite)
collection of holomorphic diffeomorphisms defined around the origin will be called {\it globally non-discrete}\, if and only
if there is a non-empty open set~$V$ satisfying the conditions~(1), (2) and~(3) of Definition~\ref{definitiondiscretegroup}.

For pseudogroups $G$ as above, the definition below is also standard by now.

\begin{defi}
\label{definitionvectorfieldclosure}
An analytic vector field $X$ defined on a non-empty open set $U$ is said to be in the closure of~$G$,
if the following condition is satisfied, up to reducing~$U$: for every set $U' \subset U$ and
every $t_0 \in \R_+$ so that the local flow of~$X$ is defined on $U'$ for every $t \in [0, t_0]$, the resulting
local diffeomorphism $\Psi^{t_0}_X : U' \rightarrow \C^n$ induced by this local flow is the uniform limit on $U'$\,
of a sequence of elements $\{ g_j \}$ contained in $G$.
\end{defi}

In the case $n=1$, it is a simple fact that a non-solvable subgroup of $\diff$ is always non-discrete. Indeed, this
statement can be checked by specifying to the one-dimensional case the results in the previous section valid for $n=2$.
This phenomenon is in line
with the general character of Shcherbakov-Nakai theory in $\diff$ asserting the existence of non-identically zero
vector fields in the closure of these groups. In fact, a globally discrete pseudogroup
cannot admit non-identically zero vector fields in its closure since the local flow $\Psi^t_X$ converges to the
identity on~compact parts of $U$ as $t \rightarrow 0_+$.
Thus the first fundamental issue opposing subgroups of $\diff$ to subgroups of $\diffn$, $n\geq 2$, is the fact
that the latter contains discrete free subgroups on two generators.

\begin{ex}
(Schottky groups and discrete subgroups of ${\rm PSL}\, (2 ,\C)$). Consider a Schottky subgroup $\Gamma$
of ${\rm PSL}\, (2 ,\C)$. The group $\Gamma$ is free on~$2$ or more generators and $\Gamma$ is also discrete
as subgroup of ${\rm PSL}\, (2 ,\C)$ in {\it the classical sense}\, (i.e. as a set in ${\rm PSL}\, (2 ,\C)$).
Once a lift of ${\rm PSL}\, (2 ,\C)$ in ${\rm SL}\, (2 ,\C)$ is chosen,
$\Gamma$ can be identified with a subgroup of ${\rm SL}\, (2 ,\C)$. Since, in turn, ${\rm SL}\, (2 ,\C)$
can be viewed as linear diffeomorphisms of $\C^2$ fixing the origin, there follows that $\Gamma$ can also be identified with
a certain subgroup of $\Diffgentwo$. The purpose of this example is to prove the following statement which does not
depend on the chosen lift of ${\rm PSL}\, (2 ,\C)$ in ${\rm SL}\, (2 ,\C)$.

\noindent {\it Claim}. The group $\Gamma \subseteq \Diffgentwo$ is discrete in the sense of Definition~\ref{definitiondiscretegroup}.

\noindent {\it Proof of the Claim}. This is certainly a well-known result so that we shall content ourselves of sketching an
argument. Recall that ${\rm PSL}\, (2 ,\C)$ is identified with the automorphism group of the projective line
$\C P(1)$. In turn,
by considering $\C P(1)$ as the boundary of the unit ball $B^3$ of $\R^3$, the group ${\rm PSL}\, (2 ,\C)$
becomes identified with the group of orientation-preserving isometries of the hyperbolic ball. This allows us to
identify ${\rm PSL}\, (2 ,\C)$ with ${\rm SO}\, (3, \R) \times B^3$ by assigning to an element $\gamma \in {\rm PSL}\, (2 ,\C)$
the pair $(D_0 \gamma , \gamma (0))$ in ${\rm SO}\, (3, \R) \times B^3$. The topology of ${\rm PSL}\, (2 ,\C)$ also arises
from this identification. In particular, a subgroup $\Gamma \subset {\rm PSL}\, (2 ,\C)$ is discrete {\it
in the classical sense}\, if and only
if it contains only finitely many elements $\gamma_i$ such that $\Vert \gamma_i (0) \Vert < r$ for every $r \in (0,1)$.

On the other hand, consider the action of ${\rm PSL}\, (2 ,\C)$ on $\C P(1)$. Assume that $W \subset \C P(1)$ is a non-empty open
set and suppose that $\gamma_i \in {\rm PSL}\, (2 ,\C)$ is a sequence of elements converging uniformly to the identity
on $W$. Then by considering the extension of the action of $\gamma_i$ to $B^3$, there follows at once that
$$
\Vert \gamma_i (0) \Vert \rightarrow 0
$$
as $i \rightarrow \infty$. Therefore, if $\Gamma \subset {\rm PSL}\, (2 ,\C)$ is known to be discrete
(in the classical sense), then for every
non-empty open set $W \subset \C P(1)$ the group $\Gamma$ contains no sequence of elements (different from the identity) converging
uniformly to the identity on $W$. Finally, if our group $\Gamma$ were {\it not globally discrete}\,
then there would exist a non-empty open set $V \subset \C^2$ and a sequence
of elements $\{ \gamma_i \} \subset \Gamma \subset \Diffgentwo$ converging uniformly to the identity on $V$ (where
$(0,0) \not\in V$). Since
the action of $\Gamma$ on $\C^2$ is linear, it induces a projective action of $\Gamma$ on $\C P(1)$ coinciding with the
action induced by identifying ${\rm PSL}\, (2 ,\C)$ with the automorphism group of $\C P(1)$. Hence, by letting $W \subset \C P(1)$
be the image of $V \subset \C^2$ by the canonical projection $\C^2 \setminus \{ (0,0) \} \rightarrow \C P(1)$, it follows
that the sequence $\{ \gamma_i \}$ converges uniformly to the identity on $W$. As previously seen, this contradicts the fact that
$\Gamma$ is a discrete subgroup of ${\rm PSL}\, (2 ,\C)$.
\end{ex}

On the other hand, the results in Section~\ref{provingtheorems}, also show that every non-solvable subgroup of $\diffCtwo$
is non-discrete in the sense of Definition~\ref{definitiondiscretegroup} (i.e. for a chosen neighborhood of the origin).
At this level, there is no known obstruction to the existence of vector fields
in the closure of these non-solvable groups though no general affirmative result is so far available. Inasmuch as no ``counterexample''
is known, it seems a bit unlikely that non-trivial vector fields in the closure of the corresponding group will
exist without any (at least weak) additional assumption.

Going back to subgroups of $\diffn$, the notion of {\it global non-discrete}\, is less suited than the notion of non-discrete
set forth by Definition~\ref{definitiondiscretegroup} since the former depends on the representatives chosen. Actually,
even for a given finite set of local diffeomorphisms (fixing the origin),
it may happen that the pseudogroup they generate on an open set $U$ is non-discrete while it becomes discrete on a smaller
open set. Furthermore, from a technical point of view, the effects of non-linear terms away from the origin can easily become
out of control. Let us close this discussion with a remark showing that ``many'' discrete subgroups
of $\Diffgentwo$ can be produced by ``higher order perturbations'' of discrete subgroups of ${\rm GL}\, (2,\C)$.

\begin{ex}
(Non-linear perturbations of discrete subgroups of ${\rm GL}\, (2 ,\C)$). Given a subgroup $G \subset \Diffgentwo$
consider again the natural homomorphism $\rho : G \rightarrow {\rm GL}\, (2,\C)$ and the associated exact sequence
$$
0 \longrightarrow {\rm Ker}\, (\rho) \longrightarrow G \longrightarrow \rho (G) \subset {\rm GL}\, (2,\C)
\longrightarrow 0 \, ,
$$
where $\rho (g)$ is the derivative $D_0g$ at the origin. Then we have:

\noindent {\it Claim}. Suppose that $\rho (G) \subset {\rm GL}\, (2,\C)$ is a discrete subgroup and that
${\rm Ker}\, (\rho) \subset \diffCtwo$ is discrete as well (this happens, for example, when the homomorphism
$\rho$ is one-to-one). Then $G$ is discrete.

\noindent {\it Proof}. Suppose that $\{ g_j\}$, $g_j \neq {\rm id}$ for every~$j \in \N$,
is a sequence of elements in $G$ converging uniformly to the identity on some neighborhood $V$ of $(0,0) \in \C^2$.
Then the sequence of derivatives $\{ D_0 g_j \} \subset {\rm GL}\, (2,\C)$ must converge to the identity matrix~$I$
by virtue of the Cauchy formula. Since $\rho (G) \subset {\rm GL}\, (2,\C)$ is discrete, there follows that $D_0 g_j$
equals~$I$ for large $j \in \N$. Hence, modulo dropping finitely many terms in the mentioned sequence, we have
$g_j \in {\rm Ker}\, (\rho)$ for every~$j \in \N$. A contradiction then arises from the fact that
${\rm Ker}\, (\rho) \subset \diffCtwo$ is discrete. The claim is proved.
\end{ex}

\section{Abelian groups, normalizers, and general solvable subgroups of $\formdiffn$}

The rest of this paper will entirely be devoted to the proof of Theorem~\ref{commuting9} and to related
results. This section is divided into three subsections and, in the first one, some additional elementary
results concerning abelian subgroups of $\formdiffn$ are provided. The second subsection concerns more elaborate results
on normalizers of abelian subgroups of $\formdiffn$; see also
Lemma~\ref{newversionLemma22.1}. Finally the third section is essentially devoted to stating a detailed version
of Theorem~6 in \cite{ribon} in the specific case of subgroups of $\formdiffn$.
In what follows, we keep the notations of Sections~2.2 and~2.3.

\subsection{Elementary facts on abelian groups}

To begin with, consider an abelian subgroup $G \subset \formdiffn$ which, in principle, need not be finitely generated.
The group $G$ is necessarily torsion-free since all of its elements are tangent to the identity.
Hence a basis $\{ g_i \}_{i=1}^N \subset G$ for this group can be considered where $N \in
\N \cup \{ \infty \}$ (at this point we only assume the group is countably generated).

To the group $G$, it is associated an abelian Lie algebra $\mathfrak{g} \subset \ghatX$ which is generated (both as Lie
algebra and as vector space) by the infinitesimal generators $X_i$ of the elements $g_i$ in the above mentioned basis.
Let us first consider the case in which $\mathfrak{g}$ contains two formal vector fields $X$ and $Y$ which are not
everywhere parallel. In particular every vector field in $\ghatXone$ can be written as a linear combination of
$X$ and $Y$ with coefficients in the field $\fieldC$; see Section~2.2. Then we have:

\begin{lemma}
\label{newversionLemma11.1}
Under the above assumption, the abelian Lie algebra $\mathfrak{g}$ is generated over $\C$ by $X$ and $Y$.
In particular $G$ is contained in the exponential of $\mathfrak{g}$ though $G$ is not contained in the
exponential of a single vector field in $\mathfrak{g}$.
\end{lemma}

\begin{proof}
Consider $Z \in \mathfrak{g}$ and let $Z = aX +bY$ with $a,b \in \fieldC$. Since $\mathfrak{g}$ is abelian,
it follows that $[Z,X]=[Z,Y]=0$ which in turn leads to
$$
\frac{\partial a}{\partial X} = \frac{\partial a}{\partial Y} = \frac{\partial b}{\partial X} =
\frac{\partial b}{\partial Y} = 0 \, .
$$
Since $X$ and $Y$ are not everywhere parallel, there follows that $a, b$ are both constants i.e.,
$a,b \in \C$ proving the first part of the statement.

To conclude that $G$ cannot be contained in the exponential of a single vector field just note that,
if this were the case, the Lie algebra $\mathfrak{g}$ would coincide with the one-dimensional vector space
spanned by the vector field in question. This clearly contradicts the existence of two non everywhere parallel
vector fields $X$ and $Y$ in $\mathfrak{g}$.
\end{proof}

The argument above also yields the following corollary:

\begin{corol}
\label{newversionLemma11.2}
Assume that $\mathfrak{g} \subset \ghatX$ is a (non-trivial) abelian Lie algebra. Then one of the
following holds:
\begin{enumerate}

\item Suppose that $\mathfrak{g}$ contains two vector fields $X$ and $Y$ that are not everywhere parallel.
Then $\mathfrak{g}$ can be identified with the two-dimensional vector space spanned by $X$ and $Y$
over $\C$.

\item There is a basis $\{ X_i \}_{i=1}^N$, $N \in \N \cup \{ \infty \}$, for $\mathfrak{g}$ having such that,
setting $X=X_1$, we have $X_i = h_iX$ for every $2 \leq i \leq N$ where $h_i \in \fieldC$ is a first integral of $X$
(i.e., $\partial h_i /\partial X =0$).\qed
\end{enumerate}
\end{corol}

The Lie algebra $\mathfrak{g}$ spanned by vector fields $X$ and $Y$ as in item~(1) above will be referred to as
the {\it linear span of $X$ and $Y$}\, so that the phrase ``the linear span of $X$ and $Y$'' implies
that $X$ and $Y$ commute and that they are not everywhere parallel.
Concerning item~(2), we note that $N \in \N$ if and only if
$\mathfrak{g}$ is finitely generated as vector space. More importantly, although $X$ and $X_i$, $i\geq 2$, belong to
$\ghatX$, the equation $X_i = h_i X$ does not imply that
$h_i$ lies in $\formalC$, as opposed to $\fieldC$, since $X$ is not supposed
to have isolated singularities.

Recall that the {\it centralizer}\, of an element $F \in \formdiffn$ is the group formed by those elements in $\formdiffn$
commuting with $F$ (and hence commuting with every element in the cyclic group generated by $F$). To
characterize the centralizer of an element $F \in \formdiffn$, we denote by $X$ its infinitesimal generator.
Note that there may or may not exist another vector field $Y$ commuting with $X$ while
not everywhere parallel to $X$. When this vector field $Y$ exists, it is never unique since every linear combination
of $X$ and $Y$ will have similar properties. Furthermore, if $X$ happens
to admit some non-constant first integral $h$, then $hY$ will also commute with $X$. When both
$h$ and $Y$ exist, then
every element of $\formdiffn$ whose infinitesimal generator $Z$ has
the form $Z=aX + bY$, where $X, Y$ are as above and $a,\, b$ are first integrals of $X$, automatically belongs
to the centralizer of $F$, cf. Lemma~\ref{commuting1}. With this notation, the centralizer of $F$ admits
the following characterization.

\begin{lemma}
\label{commuting2}
Let $F \in \formdiffn$ be given and denote by $X$ its infinitesimal generator. Then the centralizer of $F$ in $\formdiffn$
coincides with one of the following groups.
\begin{description}
\item[{\sc Case 1}] Suppose that every vector field $Y \in \ghatX$ commuting with $X$ is everywhere parallel to $X$.
Then the centralizer of $F$ consists of the subgroup of $\formdiffn$ whose elements have infinitesimal generators
of the form $hX$, where $h \in \fieldC$ is a formal first integral of $X$. In particular, if $X$ admits only constants
as first integrals, then the centralizer of $F$ is reduced to the exponential of $X$.

\item[{\sc Case 2}] Suppose there is $Y \in \ghatX$ which is not everywhere parallel to $X$ and still commutes with $X$.
Then the centralizer of $F$ coincides with the subgroup of $\formdiffn$ consisting of those elements $F \in \formdiffn$ whose
infinitesimal generators have the form $aX +bY$, where $a,b \in \fieldC$ are (formal) first integral of $X$.
\end{description}
\end{lemma}

\begin{proof}
Suppose that $H$ is an element of $\formdiffn$ commuting with $F$. Denoting by $Z$ the infinitesimal generator of $H$, it follows
from Lemma~\ref{commuting1} that $[X,Z]=0$. Conversely the $1$-parameter group obtained as the exponential of $Z$
is automatically contained in the centralizer of $F$.

Next, suppose that the assumption in Case~1 is verified. Then the quotient $h$ between $Z$ and $X$ can be defined as an element
of $\fieldC$ satisfying $Z = hX$. Therefore the condition $[X,Z]=0$ becomes $dh.X=0$, i.e. $h$ is a first integral for $X$.

Consider now the existence of $Y$, not everywhere parallel to $X$, verifying $[X,Y]=0$.
It is clear that the elements of $\formdiffn$ described
in Case~2 belong to the centralizer of $F$. Thus only the converse needs to be proved. Since $H$ commutes with $F$, Lemma~\ref{commuting1}
yields again $[X,Z]=0$.
Since $Y$ is not a multiple of $X$, there are functions $a(x,y) , \, b(x,y) \in \fieldC$ such that $Z = a X + b Y$. Now the equation $[X, Z] = 0$
yields
$$
( \partial a /\partial X) .X +  ( \partial b /\partial X) .Y =0 \, .
$$
Thus the fact that $Y$ is not a multiple of $X$ ensures that $( \partial a /\partial X)  = ( \partial b /\partial X) =0$. In other words, both
$a, \, b$ are first integrals of $X$. The lemma follows.
\end{proof}

Concerning the situation described in Case~2 of Lemma~\ref{commuting2}, it is already known that $Y$ is not uniquely defined.
Nonetheless, the reader will note that every other choice of a vector field commuting with $X$ and not everywhere parallel to~$X$
leads to the same group of elements commuting with $F$.

Here is an easy consequence of Lemma~\ref{commuting2}.

\begin{lemma}
\label{lastversionLemma2}
Suppose that $h$ is a non-constant first integral of $X$ and let $F_1 = {\rm Exp} \, (X)$ and $F_2 = {\rm Exp} \, (hX)$
be elements in $\formdiffn$.
The intersection of the centralizers of $F_1$ and $F_2$, i.e. the set of elements in $\formdiffn$ commuting
with both $F_1, F_2$ is the subgroup of $\formdiffn$ constituted by those elements whose infinitesimal generators
have the form $aX$, where $a$ is a first integral of $X$. In particular, this group is abelian.
\end{lemma}

\begin{proof}
Let $\overline{F} \in \formdiffn$ be an element commuting with both $F_1$ and $F_2$ and denote by $Z$ the
infinitesimal generator of $F$. If $Z$ is everywhere parallel to~$X$, then the statement follows
from Lemma~\ref{commuting2}, Case~1. Assume now that $Z$ is not everywhere parallel to~$X$ and note that
we must have $[Z,X]=[Z,hX]=0$ (Lemma~\ref{commuting1}). From this, there follows that $\partial h /\partial Z=0$.
Since $\partial h /\partial X=0$ and $X, \, Z$ are not everywhere parallel, we conclude that $h$ must be
constant which gives the desired contradiction.
\end{proof}

\subsection{On normalizers of certain abelian groups}

Recall that the {\it normalizer}\, of a group
$G \subset \formdiffn$ is the maximal subgroup $N_G$ of $\formdiffn$ containing $G$ and such that $G$ is a normal subgroup of $N_G$.
Similarly, the {\it centralizer} of an abelian group $G$ is the maximal subgroup of $\formdiffn$ containing
$G$ in its center. Given $F \in \formdiffn$, we shall refer to the {\it normalizer}\, (resp. {\it centralizer}) of
$F$ meaning the normalizer (resp. centralizer) of the cyclic group generated by $F$. This section is intended to establishing
certain results concerning normalizers of abelian groups. In the sequel we shall freely use the following consequence
of the constructions detailed in Section~2.3: given a subgroup $G \subset \formdiffn$ whose Lie algebra is denoted
by $\mathfrak{g}$, the normalizer of $G$ in $\formdiffn$ naturally acts by pull-backs on $\mathfrak{g}$.
Note that this assertion becomes apparent if $\mathfrak{g}$ is thought of as being the Lie
algebra generated by the infinitesimal generators of all elements of $G$.

We begin with an easy observation:

\begin{lemma}
\label{newversionLemma11.3}
Given $F \in \formdiffn$, the normalizer and the centralizer of $F$ coincide and hence are described
by Lemma~\ref{commuting2}.
\end{lemma}

\begin{proof}
It suffices to show that the normalizer of $F$ is contained in the centralizer of $F$. For this,
denote by $Z$ the infinitesimal generator of $F$ and consider an element $g$ in the normalizer of $F$.
The Lie algebra associated to the cyclic group generated by $F$ has dimension~$1$ and consists of constant
multiples of $Z$. Since $g$ acts on this Lie algebra by pull-backs, there follows that $g^{\ast} Z = cZ$
for some $c \in \C$. However, the automorphism induced by $g$ is unipotent since $g \in \formdiffn$ and this implies that $c=1$,
i.e. $g^{\ast} Z =Z$. There follows that $g$ commutes with $F$ and the lemma is proved.

Alternatively, Hadamard lemma (Formula~\ref{hadamardlemmastatement1}) shows that $g^{\ast} Z -Z \in \mathfrak{g}$
would have order at the origin strictly greater than the order of $Z$ unless $g^{\ast} Z =Z$. Therefore the latter possibility
must hold.
\end{proof}

Consider now an abelian group $G \subset \formdiffn$ whose Lie algebra coincides with
the linear span of two vector fields $X, \, Y$ as in item~(1) of Corollary~\ref{newversionLemma11.2}.
Concerning the Lie algebra $\mathfrak{g}$ of $G$, two different situations may occur, namely:
all linear combinations of $X, \, Y$ may or may not have the
same order at the origin. Clearly, when not all these vector fields have the same order at the origin, there is
a unique (up to a multiplicative constant) vector field $Z$ in $\mathfrak{g}$ whose order at the origin
strictly greater than the orders of all remaining vector fields in $\mathfrak{g}$.

\begin{lemma}
\label{lastversionLemma3}
Let $G \subset \formdiffn$ be an abelian group whose Lie algebra $\mathfrak{g}$ is isomorphic to the linear span
of vector fields $X, \, Y \in \ghatX$. Then one of the following holds:
\begin{enumerate}
  \item Assume that all vector fields in the linear span of $X, Y$ have the same order at the origin.
Then the normalizer $N_G$ of $G$ is contained in the exponential of $\mathfrak{g}$.

  \item Assume that there is a vector field $Z$ in the linear span of $X, Y$ whose order is greater than the orders
  of the remaining vector fields. Then the normalizer $N_G$ of $G$ is either abelian or metabelian. Also $N_G$ is necessarily
  a nilpotent group.
\end{enumerate}
\end{lemma}

\begin{proof}
Consider the subgroup $\Gamma_{{\rm abelian}-1}$ of $\formdiffn$ consisting of all elements in $\formdiffn$
that act on $\mathfrak{g}$. In other words, a formal diffeomorphism $F \in \formdiffn$ lies in $\Gamma_{{\rm abelian}-1}$
if and only if $F^{\ast} \mathfrak{g} \subset \mathfrak{g}$. Recalling that the normalizer $N_G \subset \formdiffn$
of $G$ naturally acts on $\mathfrak{g}$, there follows that $N_G \subseteq \Gamma_{{\rm abelian}-1}$. To prove
the lemma it is therefore sufficient to study the group $\Gamma_{{\rm abelian}-1}$.

Let then $F \in \Gamma_{{\rm abelian}-1} \subset \formdiffn$
and consider the action of $F$ on $\mathfrak{g}$.
Since $F$ is unipotent, the eigenvalues of the corresponding automorphism are equal to~$1$ so that either this automorphism
coincides with the identity or it is non-diagonalizable. In the former case $F$ belongs to the exponential of
$\mathfrak{g}$ since it commutes with both $X$ and $Y$.

Suppose now that the above mentioned action of $F$ is not diagonalizable. Up to a change of basis, we can assume that
$F^{\ast} X =X$ and $F^{\ast} Y = Y + X$. In particular, Hadamard lemma applied to $F^{\ast} Y - Y$ shows that
the order of $X$ is strictly larger than the order of $Y$ so that we are in the situation described in item~(2).
In other words, if $X$ and $Y$ are as in item~(1), then the group $\Gamma_{{\rm abelian}-1}$ coincides with the
exponential ${\rm Exp}\, (\mathfrak{g})$ of $\mathfrak{g}$ and the lemma follows at once.

It remains to study the case were $X$ and $Y$ are as in item~(2). Without loss of generality, we can assume that
the order of $X$ is strictly larger than the order of~$Y$ so that the $X$ is distinguished in $\mathfrak{g}$ as the
unique vector field (up to a multiplicative constant) having maximal order at the origin. In particular, for
every $F \in \Gamma_{{\rm abelian}-1}$, we have $F^{\ast} X=X$.

Next fix a vector field $Z \in \mathfrak{g}$ which is not a constant multiple of $X$. The vector field
$F^{\ast} Z - Z$ lies in $\mathfrak{g}$ and has order strictly larger than the order of~$Z$. Thus
$F^{\ast} Z - Z$ must be a constant multiple of $X$, i.e. we have
\begin{equation}
F^{\ast} Z = Z + cX \label{homomorphism1}
\end{equation}
for some constant $c \in \C$ depending only on $F$ (note that $c$ may equal zero and this is certainly the case
when $F$ lies in the exponential of $\mathfrak{g}$). In particular,
we have obtained a map $\sigma$ from $\Gamma_{{\rm abelian}-1}$ to $\C$ that
assigns to $F \in \Gamma_{{\rm abelian}-1}$ the constant $c \in \C$ appearing in Equation~(\ref{homomorphism1}). However,
since every element $F \in \Gamma_{{\rm abelian}-1}$ verifies $F^{\ast} X =X$, there also follows that
$\sigma : \Gamma_{{\rm abelian}-1} \rightarrow \C$ is a group homomorphism.
Furthermore the kernel of $\sigma$ consists of those elements fixing both $Z$ and $X$ so that
this kernel can be identified with the exponential ${\rm Exp}\, (\mathfrak{g})$ of $\mathfrak{g}$.
Summarizing, the group $\Gamma_{{\rm abelian}-1}$ can alternately be defined by
the short exact sequence
\begin{equation}
0 \longrightarrow {\rm Exp}\, (\mathfrak{g}) \simeq \C^2 \longrightarrow \Gamma_{{\rm abelian}-1}
\stackrel{\sigma}{\longrightarrow} \C \longrightarrow 0 \, . \label{ashortexactsequence1}
\end{equation}
This sequence realizes $\Gamma_{{\rm abelian}-1}$ as
an abelian extension of an abelian group so that $\Gamma_{{\rm abelian}-1}$ must be
step~$2$ solvable. Since $N_G$ naturally sits
inside $\Gamma_{{\rm abelian}-1}$, we conclude that $N_G$ is either abelian or metabelian.

It only remains to check that the group $\Gamma_{{\rm abelian}-1}$ is, in fact, nilpotent. For this note that
$\Gamma_{{\rm abelian}-1}$ is a complex Lie group of dimension~$3$ as follows from sequence~(\ref{ashortexactsequence1}).
Denoting by $\mathfrak{g}_{{\rm abelian}-1}$ its Lie algebra, we see that $\mathfrak{g}_{{\rm abelian}-1}$
is neither abelian (otherwise there is nothing to be proved) nor isomorphic to the Lie algebra of ${\rm PSL}\, (2,\C)$
since $X$ commutes with $Y$. Furthermore the image of $\mathfrak{g}_{{\rm abelian}-1}$ by the adjoint representation must
be contained in $\mathfrak{g}$. Since $X$ is distinguished in $\mathfrak{g}_{{\rm abelian}-1}$ for its maximal order at the
origin, there follows that $X$ lies in the center of $\mathfrak{g}_{{\rm abelian}-1}$. Finally, by considering a third
element $\widetilde{Z}$ in $\mathfrak{g}_{{\rm abelian}-1}$ so that $X, Y, \widetilde{Z}$ form a basis for
$\mathfrak{g}_{{\rm abelian}-1}$, we also conclude that $[\widetilde{Z}, Y] = cX$ since $[\widetilde{Z}, Y]$ lies
in $\mathfrak{g}$ and has order strictly larger than the order of~$Y$. From this it follows that
$\mathfrak{g}_{{\rm abelian}-1}$ is isomorphic to the Lie algebra of strictly upper-triangular $3\times3$ matrices or,
equivalently, that $N_G$ is isomorphic to a subgroup of the group of unipotent upper-triangular $3\times 3$ matrices.
The lemma is proved.
\end{proof}

The next lemma completes the description of the normalizers of (non-trivial) finitely generated
abelian groups, cf. Corollary~\ref{newversionLemma11.2}.

\begin{lemma}
\label{commuting7nowlemma}
Let $G \subset \formdiffn$ be a finitely generated abelian group all of whose elements have infinitesimal
generators parallel to a certain formal vector field $X$. Assume that the rank of $G$ is at least~$2$.
Then the normalizer $N_G$ of $G$ in $\formdiffn$ is either abelian or metabelian. Furthermore, it is also a
nilpotent group.
\end{lemma}

\begin{proof}
Recall that the Lie algebra $\mathfrak{g}$ of $G$ is generated both as Lie algebra and as vector space by the
infinitesimal generators of a set of elements forming a basis for $G$. Denote by $n \geq 2$ the dimension of this Lie algebra.
Let $X$ denote an element in $\mathfrak{g}$ whose order at $(0,0) \in \C^2$ is maximal among
all vector fields in $\mathfrak{g}$. The existence of $X$ is guaranteed by the fact that $\mathfrak{g}$ has finite
dimension.

Next let $\Gamma_{{\rm abelian}-2} \subset \formdiffn$ be the group formed by all those  diffeomorphisms $F$
in $\formdiffn$ for which $F^{\ast} \mathfrak{g} \subset \mathfrak{g}$. In particular the normalizer
$N_G$ of $G$ is naturally contained in $\Gamma_{{\rm abelian}-2}$. Furthermore, for every $F \in
\Gamma_{{\rm abelian}-2}$, we have $F^{\ast} X =X$ since $X$ has maximal order in $\mathfrak{g}$ and
$F$ is unipotent. In particular, if there are more than one vector field (up to multiplicative constants) in
$\mathfrak{g}$ having maximal order at $(0,0) \in \C^2$, the group $\Gamma_{{\rm abelian}-2}$ must be abelian
since it will lie in the intersection of the centralizers of $X$ and another vector field
$hX$, where $h$ is a non-constant first integral of $X$; see Lemma~\ref{lastversionLemma2}. In fact,
$\Gamma_{{\rm abelian}-2}$ will coincide
with the exponential of the infinite dimensional Lie algebra formed by vector fields
of the form $aX$ where $a$ is a first integral of~$X$.

Suppose now that up to multiplicative constants
$X$ is the unique vector field in $\mathfrak{g}$ whose order at the origin is maximal. Consider a vector field
$Y \in \mathfrak{g}$ whose order at the origin is the ``second largest possible'' in the sense that the condition
of having a vector field $Z$ in $\mathfrak{g}$ whose order at $(0,0) \in \C^2$ is strictly greater than the order of $Y$
implies that $Z$ must be a constant multiple
of~$X$. Note that the vector field $Y$ clearly exists since $n \geq 2$, though it is not necessarily unique (always up to
multiplicative constants).
For $F \in \Gamma_{{\rm abelian}-2}$, we consider $F^{\ast} Y -Y \in \mathfrak{g}$. As previously seen,
Hadamard's lemma implies that $F^{\ast} Y = Y + cX$ for some constant $c \in \C$ and for every $F \in \Gamma_{{\rm abelian}-2}$.
Thus, arguing as in the proof of Lemma~\ref{lastversionLemma3}, we conclude that the assignment $F \in \Gamma_{{\rm abelian}-2}
\mapsto c \in \C$ such that $F^{\ast} Y = Y + cZ$ defines a homomorphism from $\Gamma_{{\rm abelian}-2}$ to $\C$ whose kernel
is an abelian group. There follows that $\Gamma_{{\rm abelian}-2}$ is either abelian or metabelian.
In fact, the group $\Gamma_{{\rm abelian}-2}$ can alternately be defined by the exact sequence
\begin{equation}
0 \longrightarrow G_X \longrightarrow \Gamma_{{\rm abelian}-2}
\longrightarrow \C \longrightarrow 0 \, , \label{ashortexactsequence2}
\end{equation}
where $G_X \subset \formdiffn$ is the abelian group all of whose elements have infinitesimal generator of the form
$aX$ where $a$ is a first integral for $X$. To complete the proof of the lemma, it suffices to check that
$\Gamma_{{\rm abelian}-2}$ is nilpotent. For this let $E_1 \subset \mathfrak{g}$ denote the vector space spanned by $X$.
Next considering the subspace $E_2 \subset \mathfrak{g}$ spanned by all elements in $\mathfrak{g}$ having ``second greatest
order'' at the origin. Clearly $E_1 \cap E_2 = \{0 \}$. If $E_1 \oplus E_2$ is strictly contained in $\mathfrak{g}$,
we continue inductively by defining $E_3$ as the subspace whose vector fields have ``third greatest order''. With this
procedure, we obtain a decomposition of $\mathfrak{g}$ as direct sum of subspace $E_1 \oplus \cdots \oplus E_d$ for some
$d \geq 2$. Moreover, for every $F \in \Gamma_{{\rm abelian}-2}$ and $Z \in E_i$, $i=1, \ldots ,d$,
we have $F^{\ast} Z -Z \in E_1 \oplus \cdots \oplus E_{i-1}$ thanks again to Hadamard lemma (where $F^{\ast} Z -Z = 0$ if $i=1$).
Since, in addition, every element in $\formdiffn$ is unipotent, we conclude from the preceding that that these automorphism
have upper block-triangular form with all the eigenvalues equal to~$1$. Denote by ${\rm Aut}\, (\mathfrak{g})$
the group formed by these automorphism and note that this group is nilpotent. Now, it is immediate to check that
the assignment of the corresponding induced automorphism to every element in $\Gamma_{{\rm abelian}-2}$ is a homomorphism
from $\Gamma_{{\rm abelian}-2}$ onto ${\rm Aut}\, (\mathfrak{g})$ whose kernel is $G_X$. Hence, we obtain
$\Gamma_{{\rm abelian}-2} = G_X \rtimes {\rm Aut}\, (\mathfrak{g})$ i.e., $\Gamma_{{\rm abelian}-2}$ is the semidirect product
of $G_X$ (abelian) and ${\rm Aut}\, (\mathfrak{g})$ (nilpotent). The nilpotent character of $\Gamma_{{\rm abelian}-2}$
follows at once. The proof of the lemma is completed.
\end{proof}

\subsection{Classification of solvable groups in $\formdiffn$}

In this section, we shall detail the classification of finitely generated solvable subgroups of $\formdiffn$
obtained in the work of Martelo-Ribon \cite{ribon}, cf. Theorem~6 in the mentioned paper.
However we begin with a more general lemma.

\begin{lemma}
\label{newversionLemma22.1}
Suppose that $G_0 \subset \formdiffn$ is an abelian group whose Lie algebra coincides with the linear
span of two vector fields $X$ and $Y$. Suppose also that $G_1$ is a non-abelian finitely generated group containing $G_0$ as
a normal subgroup. Then the normalizer of $G_1$ is metabelian.
\end{lemma}

\begin{proof}
Since $G_1$ is non-abelian so is the normalizer of $G_0$. Therefore, it follows from Lemma~\ref{lastversionLemma3}
that $G_1$ is isomorphic to a (non-abelian) subgroup of $\Gamma_{{\rm abelian}-1}$ which, in turn, is a step~$2$
solvable (and nilpotent) group. Also, still owing to Lemma~\ref{lastversionLemma3},
we can assume without loss of generality that the order of $X$ at $(0,0) \in \C^2$ is strictly larger than the
corresponding order of~$Y$ and that the exponential of $X$ contains the center of $G_1$.

Consider the Lie algebra associated with $G_1$ and note that this Lie algebra cannot be abelian.
Also this Lie algebra contains (strictly) the linear span of $X$ and $Y$ so that its dimension is
at least~$3$. On the other hand, the Lie algebra of $G_1$ is isomorphic to a sub-algebra of the Lie algebra of
$\Gamma_{{\rm abelian}-1}$. Since the latter algebra has dimension~$3$ (Lemma~\ref{lastversionLemma3}), there follows that
the two Lie algebras coincide. In particular $G_1$ is Zariski-dense in $\Gamma_{{\rm abelian}-1}$. In turn, the Zariski-denseness
of $G_1$ in $\Gamma_{{\rm abelian}-1}$ implies that these two groups share the same normalizer in
$\formdiffn$. Therefore, to prove Lemma~\ref{newversionLemma22.1}, it suffices to establish the lemma below concerning
the normalizer of $\Gamma_{{\rm abelian}-1}$.
\end{proof}

\begin{lemma}
\label{nextlemmabelow}
The normalizer of $\Gamma_{{\rm abelian}-1}$ coincides with $\Gamma_{{\rm abelian}-1}$ itself.
\end{lemma}

\begin{proof}
We need to show that the normalizer of $\Gamma_{{\rm abelian}-1}$ is contained in $\Gamma_{{\rm abelian}-1}$.
We know that every element $F \in \formdiffn$ lying in the normalizer of $\Gamma_{{\rm abelian}-1}$ acts by pull-backs on
the Lie algebra $\mathfrak{g}_{{\rm abelian}-1}$ which, in turn, is spanned as vector space by three vector fields
$X, Y$, and $Z$ where $X$ lies in the center and where $[Y,Z]=X$. In particular the order of $X$ at the origin is strictly
larger than the corresponding orders of $Y$ and $Z$. We claim that $X$ is distinguished in $\mathfrak{g}_{{\rm abelian}-1}$
as the vector field of maximal order at the origin. To check this assertion, note that every vector field in
$\mathfrak{g}_{{\rm abelian}-1}$ having maximal order at the origin must lie in the center of $\mathfrak{g}_{{\rm abelian}-1}$.
If this center were not spanned by constant multiples of $X$, then the Lie algebra $\mathfrak{g}_{{\rm abelian}-1}$ would
be contained in the centralizer of two vector fields which would force $\mathfrak{g}_{{\rm abelian}-1}$ to be abelian; cf. Section~5.1.
The resulting contradiction proves the claim.

To finish the proof of Lemma~\ref{newversionLemma22.1}, we proceed as follows. Consider the family $\mathfrak{F}$
of vector fields having the form $c_1 Y + c_2 Z$ contained in the Lie algebra of $\Gamma_{{\rm abelian}-1}$ with
$c_1, c_2 \in \C$. Let $W$ denote a vector field in $\mathfrak{F}$ having maximal order at the origin among vector fields
in this family. The existence of $W$ is clear although it need not be unique. Note however that $W$ does not coincide
with $X$ since $X$, $Y$ and $Z$ are linearly independent over~$\C$. Now, for $F$ in the normalizer of
$\Gamma_{{\rm abelian}-1}$, consider $F^{\ast} W -W \in \mathfrak{g}_{{\rm abelian}-1}$. By construction,
we also have $F^{\ast} W -W = a_1 X + a_2 Y +a_3Z$, for certain $a_i,a_2,a_3 \in \C$. However the order of both
$X$ and $F^{\ast} W -W$ are strictly greater than the order of $a_2 Y +a_3Z$. Thus we must have
$F^{\ast} W -W = a_1 X$. The lemma now follows by repeating the arguments of Lemma~\ref{lastversionLemma3}
(if $X$ and $W$ are not everywhere parallel) or of Lemma~\ref{commuting7nowlemma}
(if $X$ and $W$ are everywhere parallel). Lemma~\ref{newversionLemma22.1} is proved.
\end{proof}

We can now provide the formal classification of non-abelian solvable subgroups of $\formdiffn$. As mentioned,
the list below is a consequence of Theorem~6 in \cite{ribon} specified for the case of subgroups of $\formdiffn$.
Consider then a finitely generated solvable non-abelian group $G \subset \formdiffn$ and denote by $D^k G$ the non-trivial
derived subgroup of $G$ having highest order~$k$. In other words, $k \geq 1$ is such that
$D^k G$ is abelian and not reduced to the identity.
The reader is reminded that, albeit abelian, the group $D^k G$ need not be finitely generated even if $G$ is so.
In view of Lemma~\ref{newversionLemma22.1}, we can assume that every element in $D^k G$ has an infinitesimal generator
of the form $hX$ where $X$ is some fixed vector field and where $h$ is some first integral of~$X$.
In particular, these infinitesimal generators form the abelian Lie algebra $D^k \mathfrak{g}$ associated to $D^k G$.
Note also that $X$ can be supposed to belong to $D^k \mathfrak{g}$ since the quotient between two first integrals still
is a first integral. These assumptions will be made without further comments in what follows.

A last remark is needed before the classification of solvable subgroups of $\formdiffn$ can be stated.
Consider two vector fields $X$ and $\ooY$ such that the commutator $[X, \ooY]$ has the form $aX$ i.e.,
it is everywhere parallel to~$X$. Then the very definition of commutator for two vector fields yield the
following ``generalized Schwarz theorem''
\begin{equation}
\frac{\partial}{\partial \ooY} \left( \frac{\partial f}{\partial X} \right)
- \frac{\partial}{\partial X} \left( \frac{\partial f}{\partial \ooY} \right) = \frac{\partial f}{\partial [X,\ooY]}
= a \, \frac{\partial f}{\partial X} \label{generalizedSchwarz}
\end{equation}
for every $f \in \fieldC$. In particular $\ooY$ {\it derives first integrals of $X$ into first integrals
of $X$}, i.e. if $\varphi$ is a first integral of $X$ then so is $\partial \varphi/\partial \ooY$.

Let $\mathcal{I}_X$ denote the field formed by all first integrals of $X$ (we may assume this field contains non-constant
elements). The first case in the classification is the following:
\begin{itemize}

\item[Case 1 -] Suppose that $G$ is metabelian and that all of its elements have infinitesimal generator parallel
to a same vector field (necessarily $X$). Then the Lie algebra $\mathfrak{g}$ of $G$ is constituted by
vector fields of the form $uX$. Moreover $X$ can be chosen so as to ensure the existence
of $f \in \fieldC$ such that $\partial f /\partial X =\overline{h}$ is a non-zero element
in $\mathcal{I}_X$. Furthermore the assignment $uX \in \mathfrak{g} \mapsto u$ identifies
$\mathfrak{g}$ with a differential algebra $\mathcal{A}$ which, in turn, is
constituted by functions having the form $\varphi_1 f +\varphi_2$ where $f$ is as above and
$\varphi_1, \varphi_2$ belong to $\mathcal{I}_X$.

\item[Case 2 -] Suppose that $G$ is metabelian but contains an element whose infinitesimal generator is
not everywhere parallel to~$X$. Then there is a vector field $\ooY$ not everywhere parallel to~$X$ and
possessing the following property:

\noindent \hspace{0.3cm} ($\bullet$) \hspace{0.3cm} $[\ooY,X]= \tilde{h} X$ where $\tilde{h}$ is a
first integral for~$X$ (in particular if $\tilde{h} \equiv 0$ then $X,\ooY$ commute).

\noindent Furthermore there
is a function $f \in \fieldC$ such that $\partial f /\partial X$ is a non-identically zero first integral of $X$
and there is a certain first integral $h$ of $X$ such that every vector field in $\mathfrak{g}$ has
the form $(\varphi_1 f + \varphi_2) X + \alpha h \ooY$ where $\alpha \in \C$ and where $\varphi_1, \varphi_2$ are first integrals of~$X$.
Moreover, for every pair $(\varphi_1 f + \varphi_2) X + \alpha_1 h \ooY$ and $(\varphi_3 f + \varphi_4) X + \alpha_2 h \ooY$ of
elements in $\mathfrak{g}$, the function
$$
\alpha_2 \frac{\partial (\varphi_1 f + \varphi_2)}{\partial \ooY} - \alpha_1 \frac{\partial (\varphi_3 f + \varphi_4)}{\partial \ooY}
$$
lies in $\mathcal{I}_X$.

\item[Case 3 -] Suppose that $G$ is not metabelian so that $k=2$. Then the Lie algebra
of $G$ contains a vector field $\ooY$ not everywhere parallel to~$X$ and satisfying
the same condition as in Case~2 (namely $[\ooY,X]= \tilde{h} X$ where $\tilde{h}$ is a
first integral for~$X$).
Moreover the solvable Lie algebra of $G$ cab be identified with (a sub-algebra of)
the algebra $\mathfrak{g}_{\rm step-3}$ consisting of all formal vector fields having the form
$$
(\varphi_1 f + \varphi_2)X + \alpha h\ooY \,
$$
where $\alpha \in \C$, $\varphi_1, \varphi_2$ are first integrals of $X$, and where $h$ is a fixed first integral of $X$.

\end{itemize}

\begin{obs}
\label{justaremarkmoreorlessuseless-1}
{\rm Consider Case~2 and Case~3 above along with the corresponding vector field $\overline{Y}$. For every first
integral $h$ of $X$, note that the vector field
$h \overline{Y}$ satisfies the same conditions as $\overline{Y}$
(as indicated in Case~2). Hence, up to changing the vector field $\overline{Y}$, we can say that the
Lie algebra $\mathfrak{g}_{\rm step-3}$ consists of the vector fields having the form
$$
(\varphi_1 f + \varphi_2)X + \alpha \ooY \,
$$
where $\varphi, \, \varphi_2$, and $\alpha$ are as in Case~3. A similar simplification is possible in Case~2.}
\end{obs}

\section{Towards Theorem~\ref{commuting9} - Induced Lie algebra maps}

In the remainder two sections of this paper, the proof of Theorem~\ref{commuting9} will finally be completed.
In the present section, we shall obtain a number of general auxiliary results allowing us to derive properties
about infinitesimal generators from properties involving formal diffeomorphisms in $\formdiffn$. Some of these
results hold interest in their own and, in any event, they
will come in hand for the proof of Theorem~\ref{commuting9} provided in the next section.

We begin the discussion with a rather general lemma.

\begin{lemma}
\label{NormalizeractingonAlgebra}
Suppose we are given a Lie subalgebra $\mathfrak{g}_1$ of $\ghatX$ along with a formal diffeomorphism
$F \in \formdiffn$ whose infinitesimal generator is denoted by $Z$. Assume that $F^{\ast} \mathfrak{g}_1
\subseteq \mathfrak{g}_1$. Then there is a well-defined homomorphism $[Z, \, . \,] : \mathfrak{g}_1
\rightarrow \mathfrak{g}_1$ assigning to $X \in \mathfrak{g}_1$ the commutator $[Z,X] \in \mathfrak{g}_1$.
\end{lemma}

\begin{proof}
The proof amounts to checking that the commutator $[Z,X]$ lies in $\mathfrak{g}_1$
provided that so does~$X$. To do this,
the complex one-parameter group given by the exponential of $Z$ will be denoted by
$F_t$, $t \in \C$, so that $F_1 =F$. The proof of the lemma depends on the following
claim:

\noindent {\it Claim}. We have $F_t^{\ast} \mathfrak{g}_1 \subseteq \mathfrak{g}_1$ for every $t \in \C$.

\noindent {\it Proof of the Claim}. The argument is similar to the one employed in the proof
of Lemma~\ref{connectnessfromribon}. We consider the linear group $D_k$ for some $k$ fixed
and denote by $F_{t,k}$ the element of $D_k$ induced by $F_t$. Similarly
$\mathfrak{g}_{k, (1)}$
will denote the sub-algebra of the Lie algebra of $D_k$ induced by $\mathfrak{g}_1$.
It suffices to show that $F_{t,k}$ satisfies $F_{t,k}^{\ast} \mathfrak{g}_{k, (1)} \subseteq \mathfrak{g}_{k, (1)}$
for every $t \in \C$ and every $k \in \N$. Up to passing
to some conveniently chosen Grassmann space where $\mathfrak{g}_{k, (1)}$
becomes identified to a point, the condition
$F_{t,k}^{\ast} \mathfrak{g}_{k, (1)} \subseteq \mathfrak{g}_{k, (1)}$ becomes an algebraic equation on
the variable~$t$.
Here, as in the proof of Lemma~\ref{connectnessfromribon}, the fundamental observation
leading to the algebraic nature of this equation is the fact that $F_{t,k}$ is unipotent:
its infinitesimal generator in the
Lie algebra of $D_k$ is a nilpotent vector field. In turn, the exponential of a nilpotent vector field
has polynomial entries on $t$ since a sufficiently large powers of the corresponding matrix will vanish
identically. From this it follows that the subset of $\C$ consisting of those
$t \in \C$ for which $F_{t,k}^{\ast} \mathfrak{g}_{k, (2)} \subseteq \mathfrak{g}_{k, (1)}$ is a Zariski-closed set. However,
this set contains the positive integers $\Z_+$ and hence must coincide with all of $\C$. The claim is proved.\qed

The rest of the proof of Lemma~\ref{NormalizeractingonAlgebra} relies on Hadamard lemma. Note that
for every $t \in \C$ and every vector field $X \in \mathfrak{g}_1$, the vector field $F_{t}^{\ast} X$
lies in $\mathfrak{g}_1$ as a consequence of the Claim. Therefore Hadamard lemma yields
$$
\frac{1}{t} ( F_{t}^{\ast} X - X) = [Z,X] + \frac{t}{2} [Z,[Z,X]] + \cdots \, .
$$
For every value of $t \in \C^{\ast}$,
the left hand side of the preceding equation lies in $\mathfrak{g}_1$ since both
$F_{t}^{\ast} X$ and $X$ belong to the Lie algebra $\mathfrak{g}_1$.
However, since $\mathfrak{g}_1$ is closed, the limit of the left hand side when
$t \rightarrow 0$ also belongs to $\mathfrak{g}_1$. This limit, however, is clearly equal to
$[Z,X]$. The proof of the lemma is completed.
\end{proof}

The following consequence of Lemma~\ref{NormalizeractingonAlgebra} is worth stating:

\begin{corol}
\label{LemmainvolvingHadamard}
Suppose we are given $F \in \formdiffn$ and $X \in \ghatX$ such that $F^{\ast} X$ is everywhere parallel
to $X$. Then the infinitesimal generator $Z$ of $F$ is such that the commutator $[Z,X]$ is everywhere
parallel to~$X$.
\end{corol}

\begin{proof}
Consider the smallest Lie algebra $\mathfrak{g}$ stable under pull-backs by $F$ and containing the vector field $X$.
Since  $F^{\ast} X$ is everywhere parallel to $X$, this Lie algebra is fully constituted by vector fields
everywhere parallel to $X$. Now apply the lemma to $\mathfrak{g}_1 = \mathfrak{g}$ to conclude that
$[Z,X]$ must belong to $\mathfrak{g}$. The lemma follows.
\end{proof}

Another very useful by-product of Lemma~\ref{NormalizeractingonAlgebra} is as follows:

\begin{corol}
\label{addedinFebruary-1}
Let $\mathfrak{g}_1$ and $F \in \formdiffn$ be as in Lemma~\ref{NormalizeractingonAlgebra}.
Assume that $Z$ is the infinitesimal generator of $F$ and consider a vector field $X$ in $\mathfrak{g}_1$.
Then all the iterated commutators $[Z,\ldots [Z,[Z,X]] \ldots ]$ lie in $\mathfrak{g}_1$.
\end{corol}

\begin{proof}
We already know that $[Z,X]$ belongs to $\mathfrak{g}_1$. Let us check that $[Z,[Z,X]]$ belong to
$\mathfrak{g}_1$ as well. Keeping the notation used in the proof
of Lemma~\ref{NormalizeractingonAlgebra}, we have that $F_t^{\ast} \mathfrak{g}_1 \subseteq \mathfrak{g}_1$ for every $t \in \C$.
Now note that
$$
\frac{2}{t}
\left( \frac{1}{t} ( F_{t}^{\ast} X - X) - [Z,X] \right) = [Z,[Z,X]] + O \, (t) \, .
$$
Again for every $t \in \C$ the left side of the above equation lies in $\mathfrak{g}_1$ since both
$F_{t}^{\ast} X - X$ and $[Z,X]$ do so. By taking the limit as $t \rightarrow 0$ we then conclude that
$[Z,[Z,X]] \in \mathfrak{g}_1$ as desired. The rest of the proof is a simple induction argument.
\end{proof}

Our next lemma is also rather general and, albeit slightly technical, it will be very useful in
our discussion.

\begin{lemma}
\label{addedinJanuary2015.Number2}
Assume we are given a set $S \subset \formdiffn$ consisting of $s \geq 2$ formal diffeomorphisms
$F_1, \ldots ,F_s$. Denote by $Z_i$ the infinitesimal generator of $F_i$, $i=1,\ldots ,s$. Assume that
every element in the set $S(1) = \{ [F_i^{\pm 1}, F_j^{\pm 1} ] \, ; \; F_i, \, F_j \in S \}$ has infinitesimal
generator coinciding with a constant $c_{i^{\pm 1}, j^{\pm 1}}$ multiple of a certain vector field~$Y \in \ghatX$.
Then for every pair $i, j \in \{ 1, \ldots ,s\}$, the commutator $[Z_i,Z_j]$ coincides with $Y$ times a certain
constant in $\C$ (depending on $i$ and $j$). Moreover, in this case, we must have $[Z_i,Y]=[Z_j,Y]=0$ unless $[Z_i,Z_j]=0$.
\end{lemma}

\begin{proof}
According to Campbell-Hausdorff formula in~(\ref{finalcommutators}), the infinitesimal generator
$c_{i, j} Y$ of $F_i \circ F_j \circ F_i^{-1} \circ F_j^{-1}$ is given by
\begin{equation}
c_{i, j} Y = [Z_i,Z_j] + \frac{1}{2} \left( [Z_i,[Z_i,Z_j]] + [Z_j,[Z_i,Z_j]] \right) +
\cdots \, .\label{campbellH-1}
\end{equation}
Naturally we can assume that $c_{i, j} \neq 0$, otherwise the statement follows from
Lemma~\ref{commuting1}. On the other hand, as observed in the proof of Lemma~\ref{commuting1}, the first
non-zero homogeneous component of $[Z_i,Z_j]$ coincides with the first non-zero homogeneous component of the
entire right hand side of~(\ref{campbellH-1}). In particular, the value of $c_{i, j}$ is determined
by comparing the first non-zero homogeneous component of $[Z_i,Z_j]$ with the first non-zero homogeneous component of
$Y$.

Consider now the commutator $F_i \circ F_j^{-1} \circ F_i^{-1} \circ F_j$ whose infinitesimal generator is
$c_{i,j^{-1}} Y$ where
\begin{equation}
c_{i, j^{-1}} Y = -[Z_i,Z_j] + \frac{1}{2} \left( [Z_i,[Z_i,-Z_j]] + [-Z_j,[Z_i,-Z_j]] \right) +
\cdots \, .\label{campbellH-2}
\end{equation}
Again $c_{i, j^{-1}}$ is determined by comparing the first non-zero homogeneous components of $Y$
and of $[Z_i,Z_j]$ so that we must have $c_{i, j^{-1}} =- c_{i, j}$. Adding up equations~(\ref{campbellH-1})
and~(\ref{campbellH-2}), we obtain
$$
0 = [Z_j,[Z_i,Z_j]] + \cdots \,
$$
where the ellipsis stand for terms whose orders are
greater than the order of $[Z_j,[Z_i,Z_j]]$. From this, we conclude
that $[Z_j,[Z_i,Z_j]]$ must vanish identically. Analogously $[Z_i,[Z_i,Z_j]]$ vanishes identically
as well. In turn, the right hand side of~(\ref{campbellH-1}) (resp.~(\ref{campbellH-2})) becomes reduced to
$[Z_i,Z_j]$. The lemma follows at once.
\end{proof}

We can now begin a direct approach to the proof of Theorem~\ref{commuting9}
by recalling the general strategy to prove this type of statement.
Consider a pseudo-solvable group
$G$ along with a finite generating set $S =S(0)$ leading to a sequence of sets $S(j)$ that degenerates into
$\{ {\rm id} \}$ for large enough $j \in \N$. Denote by
$G(j)$ (resp. $G (j,j-1)$) the subgroup generated by
$S(j)$ (resp. $S(j) \cup S(j-1)$). Let $k$ be the largest integer for which $S(k)$ is not reduced to the identity.
It then follows that $G(k)$ is abelian. Similarly the group $G (k,k-1)$
is solvable. Next denote by $m$ the {\it smallest}\, integer for which
$G (m,m-1)$ is solvable. Unless otherwise mentioned, we shall always assume aiming at
a contradiction that $m \geq 2$. Recall also that
every element $F$ in $S (m-2)$ satisfies the condition
\begin{equation}
F^{\pm 1} \circ G (m-1) \circ F^{\mp 1} \subset G (m,m-1) \, . \label{Conjugatingsolvablegroups-10}
\end{equation}
Actually a slightly more precise formulation of this property is provided by condition~(\ref{Conjugatingsolvablegroups-2}).
Our aim will be to prove that the group generated by $G (m,m-1) \cup S(m-2) = G(m-1,m-2)$ is still solvable which, in turn, will
contradict the fact that $m \geq 2$.

At this juncture, it is convenient to single out a couple of simple consequences stemming from
condition~(\ref{Conjugatingsolvablegroups-10}). These are as follows.
\begin{itemize}
  \item Assume that the group $G(m-1)$ is Zariski-dense in $G(m,m-1)$. Then the two groups share the same Lie algebra
  and, in fact, they are both Zariski-dense in the exponential of this common Lie algebra. In this
  case condition~(\ref{Conjugatingsolvablegroups-10}) implies that $F \in S(m-2)$ must belong to the
  normalizer of $G(m,m-1)$. This remark will simplify the discussion in Section~7 at a couple of points.

  \item Let $\mathfrak{g} (m-1)$ (resp. $\mathfrak{g} (m)$) denote the Lie algebra associated with the group $G(m-1)$
  (resp. $G(m)$) while $\mathfrak{g} (m,m-1)$ will denote the Lie algebra associated with $G(m,m-1)$.
  Clearly $\mathfrak{g} (m-1) \subset \mathfrak{g} (m,m-1)$. Condition~(\ref{Conjugatingsolvablegroups-10}) then implies
  that $F^{\ast} (\mathfrak{g} (m-1)) \subseteq \mathfrak{g} (m,m-1)$.

\end{itemize}

Keeping the above notation, let us consider in closer detail the fact that
$F^{\ast} (\mathfrak{g} (m-1)) \subseteq \mathfrak{g} (m,m-1)$. Note that this situation is close to the content
of Lemma~\ref{NormalizeractingonAlgebra} except that we are not certain
to also have $F^{\ast} (\mathfrak{g} (m,m-1)) \subseteq \mathfrak{g} (m,m-1)$. To overcome this difficulty and
be able to exploit Lemma~\ref{NormalizeractingonAlgebra}, a further elaboration on these conditions is needed.
To begin the discussion, recall that
neither $G(m)$ nor $G(m-1)$ is reduced to the identity so that the corresponding Lie algebras
$\mathfrak{g} (m)$ and $\mathfrak{g} (m-1)$ are non-trivial. First, we have:

\begin{lemma}
\label{nodimension-1-June2015}
Without loss of generality, we can always assume that the dimension of the Lie algebra $\mathfrak{g} (m-1)$
is at least~$2$.
\end{lemma}

\begin{proof}
The proof amounts to checking that Theorem~\ref{commuting9} holds whenever the dimension of
$\mathfrak{g} (m-1)$ is exactly~$1$. For this we assume once and for all
that the dimension of $\mathfrak{g} (m-1)$ equals~$1$ so that every element in $S(m-1)$ has the
same infinitesimal generator $Y$ up to a multiplicative constant. Consider the formal diffeomorphisms $F_1, \ldots ,F_s$
in the set $S(m-2)$. The infinitesimal generator of $F_i$ is denoted by $Z_i$, $i=1,\ldots ,s$.

Assume first that to every $i=1, \ldots, s$ there corresponds $j(i) \in \{ 1, \ldots ,s\}$ such that the commutator
$[Z_i ,Z_{j(i)}]$ does not vanish identically. Under this assumption, Lemma~\ref{addedinJanuary2015.Number2} immediately implies
that the Lie algebra generated by $Y, Z_1, \ldots ,Z_s$ is solvable (actually nilpotent). The proof
of Theorem~\ref{commuting9} follows at once.

Now consider the more general case where there is $r \leq s-2$ such that
$Z_1, \ldots , Z_r$ commute with every $Z_i$, $i=1,\ldots ,s$. Moreover,
to every $i \in \{ r+1 ,\ldots, s\}$ there corresponds $j(i) \in \{ r+1 ,\ldots, s\}$ so that $[Z_i ,Z_{j(i)}]$
does not vanish identically.
The difficulty to apply Lemma~\ref{addedinJanuary2015.Number2} in this situation
lies in the fact that this lemma provides no information on the commutators $[Z_i, Y]$ for $i=1, \ldots, r$. The
desired information, however, can be derived from Jacobi identity as follows. Given $Z_i$ with $i=1, \ldots, r$,
choose two non-commuting vector fields $Z_{j_1}$ and $Z_{j_2}$ (in particular $j_1, \, j_2 \in \{ r+1 ,\ldots, s\}$).
Jacobi identity then yields
$$
0 = [Z_i ,[Z_{j_1}, Z_{j_2}]] + [Z_{j_1}, [Z_{j_2}, Z_i]] + [Z_{j_2}, [Z_i,Z_{j_1}]] \, .
$$
Since $[Z_{j_1}, Z_{j_2}]$ is a constant multiple of $Y$ and $[Z_{j_2}, Z_i] = [Z_i,Z_{j_1}] =0$, there follows
that $[Z_i,Y]=0$. Therefore the Lie algebra generated by $Y, Z_1, \ldots ,Z_s$ must still be solvable and this
yields Theorem~\ref{commuting9} in the situation in question.

Finally suppose that $[Z_i,Z_j]=0$ for every pair $i,j \in \{1, \ldots ,s\}$. Since $G(m-1)$ is not reduced to the identity,
there must exist an element $\overline{F} \in S(m-3)$ which does not commute with, say, $F_1$. Denoting by $\overline{Z}$
the infinitesimal generator of $\overline{F}$, Lemma~\ref{addedinJanuary2015.Number2} can still be applied
to ensure that $[Z_1, \overline{Z}]$ coincides with a constant multiple of $Y$ whereas $[Z_1,Y]=[\overline{Z},Y]=0$.
In particular $[Z_i,Y]=0$ for every $i=1, \ldots, s$ such that $[Z_i ,\overline{Z}]$ does not vanish identically.
On the other hand, if $[Z_{i_0},\overline{Z}]=0$ for some $i_0 \in \{1, \ldots, s\}$, then Jacobi identity gives us again
$$
0= [Z_{i_0},[Z_1 ,\overline{Z}]] + [Z_1, [\overline{Z} , Z_{i_0}]] + [\overline{Z},[Z_{i_0}, Z_1]] \, .
$$
Since $[Z_{i_0},\overline{Z}]= [Z_{i_0}, Z_1] =0$ (by assumption) and $[Z_1, \overline{Z}]$ coincides with a constant
multiple of $Y$, we conclude that $[Z_{i_0} ,Y]=0$ so that the Lie algebra generated by $Y, Z_1, \ldots ,Z_s$ is again
solvable. The proof of the lemma is completed.
\end{proof}

Now, we state:

\begin{lemma}
\label{universalsolvablealgebra-June2015}
There is a maximal solvable Lie algebra $\mathfrak{g}^{\infty} (m,m-1)$ containing $\mathfrak{g} (m,m-1)$
along with another subalgebra $\mathfrak{g}^{\infty,\ast} (m,m-1)$ which satisfies the following conditions:
\begin{itemize}
  \item $\mathfrak{g}^{\infty,\ast} (m,m-1)$ contains $\mathfrak{g} (m-1)$

  \item $\mathfrak{g}^{\infty,\ast} (m,m-1)$ is invariant under the action of $F$ by pull-backs. Moreover
  $\mathfrak{g}^{\infty,\ast} (m,m-1)$ is also uniform in the sense that it can be chosen so as
to be simultaneously invariant by every formal diffeomorphism $F \in \formdiffn$ fulfilling
condition~(\ref{Conjugatingsolvablegroups-10}).
\end{itemize}
\end{lemma}

\begin{proof}
Consider the non-trivial solvable (isomorphic) Lie algebras $\mathfrak{g} (m-1)$ and $F^{\ast} (\mathfrak{g} (m-1))$
which are both contained in the solvable Lie algebra $\mathfrak{g} (m,m-1)$. The proof of the lemma relies on the classification
of solvable Lie algebras as described in Section~5.3. To begin with consider a non-zero vector field
$X \in \mathfrak{g} (m-1)$.

\noindent {\it Case 1}. Assume that all vector fields in $\mathfrak{g} (m-1)$ are everywhere parallel to $X$.

In this case $F^{\ast} (\mathfrak{g} (m-1))$ is a Lie algebra formed by mutually everywhere parallel vector fields.
Owing to Lemma~\ref{nodimension-1-June2015} there also follows that the dimension of both Lie algebras
$\mathfrak{g} (m-1)$ and $F^{\ast} (\mathfrak{g} (m-1))$ is at least~$2$. Finally both
$\mathfrak{g} (m-1)$ and $F^{\ast} (\mathfrak{g} (m-1))$ are subalgebras of the solvable Lie algebra $\mathfrak{g} (m,m-1)$.
Direct inspection in the classification of solvable Lie
algebras provided in Section~5.3 then shows that both $\mathfrak{g} (m-1)$ and $F^{\ast} (\mathfrak{g} (m-1))$
are contained in a Lie algebra of the form $(\varphi_1 f + \varphi_2)X$ (with the notation of Section~5.3). In
particular $F^{\ast} X$ is everywhere parallel to~$X$. It also follows that the
(maximal) Lie algebra of the form $(\varphi_1 f + \varphi_2)X$ is invariant by $F$ and contains
$\mathfrak{g} (m-1)$. The Lie algebra in question can then be taken as $\mathfrak{g}^{\infty,\ast} (m,m-1)$.
Note however that our construction does not ensure that $\mathfrak{g}^{\infty,\ast} (m,m-1)$ also contains
$\mathfrak{g} (m,m-1)$. However, it is again
clear from the classification in Section~5.3 that the smallest Lie algebra containing both $\mathfrak{g}^{\infty,\ast} (m,m-1)$
and $\mathfrak{g}^{\infty} (m,m-1)$ is still a solvable Lie algebra. Thus we can choose $\mathfrak{g}^{\infty} (m,m-1)$ to
coincide with the Lie algebra generated by both $\mathfrak{g}^{\infty,\ast} (m,m-1)$
and $\mathfrak{g}^{\infty} (m,m-1)$. This proves the lemma in Case~1.

\noindent {\it Case 2}. Assume that $\mathfrak{g} (m-1)$ contains a vector field that is not everywhere parallel to $X$.

According to the discussion in Section~5.3, the solvable Lie algebra $\mathfrak{g} (m,m-1)$ either is isomorphic
to $\mathfrak{g}_{{\rm abelian}-1}$ (the Lie algebra associated with the group $\Gamma_{{\rm abelian}-1}$) or
is as in Case~2 or Case~3 of the same section.

First we assume that $\mathfrak{g} (m,m-1)$ is isomorphic to $\mathfrak{g}_{{\rm abelian}-1}$ and hence of
dimension~$3$. Recall from the proof of Lemma~\ref{lastversionLemma3} that $\mathfrak{g} (m,m-1) \simeq \mathfrak{g}_{{\rm abelian}-1}$
is generated by three vector fields $X, Y, \widetilde{Z}$ such that $[X,Y]=[X,\widetilde{Z}]=0$ and $[\widetilde{Z}, Y] = cX$.
Moreover $X$ spans the center of $\mathfrak{g} (m,m-1)$ and $X$ is also
distinguished in $\mathfrak{g} (m,m-1)$ as the vector field of maximal order at the origin (see proof of Lemma~\ref{nextlemmabelow}
for details).

On the other hand, we can assume that the dimension of $\mathfrak{g} (m-1)$ is strictly less than~$3$
otherwise all the algebras $\mathfrak{g} (m-1)$, $F^{\ast} (\mathfrak{g} (m-1))$, and $\mathfrak{g} (m,m-1)$
must coincide so that $\mathfrak{g} (m,m-1)$ is $F$ invariant and the lemma follows. By resorting
to Lemma~\ref{nodimension-1-June2015}, we can hence assume that the dimension of $\mathfrak{g} (m-1)$ is exactly~$2$.
Therefore either $\mathfrak{g} (m-1)$ is a linear span or it is isomorphic to an affine algebra. However,
the affine algebra case can be ruled out since all vector fields have zero linear part so that the
commutator have order strictly larger than the order of the initial vector fields; see Lemma~\ref{commuting1}.
In other words, $\mathfrak{g} (m-1)$ must be a linear span of commuting vector fields. Clearly $X$ lies
in $\mathfrak{g} (m-1)$, otherwise $\mathfrak{g} (m,m-1)$ would be abelian. Similarly $X$ also belongs
to $F^{\ast} (\mathfrak{g} (m-1))$. In fact, $F$ must preserve $X$ since $X$ has maximal order in $\mathfrak{g} (m,m-1)$.

Without loss of generality, we can assume that $Y$ is not everywhere parallel to $X$ and lies in $\mathfrak{g} (m-1)$.
To complete the proof of the lemma for the case where $\mathfrak{g} (m,m-1)$ is isomorphic to
$\mathfrak{g}_{{\rm abelian}-1}$. It suffices
to check that $\widetilde{Z}$ can be chosen so as to be everywhere parallel to~$X$. In fact, since $F$ preserves $X$
there follows that $F$ must preserve $\widetilde{Z}$ as well provided that $X$ and $\widetilde{Z}$ are everywhere parallel.
To check the claim, first note that $\widetilde{Z}$ has the form $aX +bY$ since $[X,\widetilde{Z}]=0$
(here $a$ and $b$ are first integrals of $X$). In turn, equation $[\widetilde{Z}, Y] = cX$ ensures that $b$ is also
a first integral of $Y$ so that $b$ must be a constant $\tilde{c} \in \C$. Now the claim follows by replacing $\widetilde{Z}$
by $\widetilde{Z} - \tilde{c} Y$.

To finish the proof of the lemma, it only remains to consider the possibility of having $\mathfrak{g} (m,m-1)$
as in Case~2 or Case~3 of Section~5.3. If the dimension of $\mathfrak{g} (m-1)$ equals~$3$ or greater, then this Lie algebra
must contain two linearly independent vector fields everywhere parallel to~$X$. Hence $F$ should preserve
the Lie algebra formed by vector fields of the form $(\varphi_1 f + \varphi_2)X$ as above. Since $\mathfrak{g} (m-1)$
also contains a vector field of the form $(\varphi_1 f + \varphi_2)X + cY$, we also conclude that $F^{\ast} (cY)$ must
coincide with a constant multiple of $Y$ up to adding another vector field of the form $(\varphi_1 f + \varphi_2)X$.
The lemma results as once in this case.

Suppose now that the dimension of $\mathfrak{g} (m-1)$ equals~$2$ and consider non-everywhere parallel vector fields
$X$ and $\overline{Y}$ in $\mathfrak{g} (m-1)$. As already seen, the Lie algebra $\mathfrak{g} (m-1)$ must be abelian.
In particular, $\overline{Y}$ yields a representation of the Lie algebra $\mathfrak{g} (m,m-1)$
by vector fields of the form
$$
(\varphi_1 f + \varphi_2)X + ch\overline{Y}
$$
where $c \in \C$ and $h$ is some fixed first integral of~$X$. However $h$ must be constant since $\overline{Y}$ itself
lies in this algebra. Thus $F$ must taken the linear span of $X$ and $\overline{Y}$ to a linear span contained in
the algebra formed by the vector fields $(\varphi_1 f + \varphi_2)X + c\overline{Y}$. The invariance of $\mathfrak{g} (m,m-1)$
by $F$ results at once. The proof of the lemma is completed.
\end{proof}

The combination of Lemma~\ref{NormalizeractingonAlgebra}, Corollary~\ref{addedinFebruary-1}, and
Lemma~\ref{universalsolvablealgebra-June2015} immediately yields the following lemma:

\begin{lemma}
\label{addedinJanuary2015}
With the preceding notations, consider an element $F \in S(m-2)$ and denote by $Z$ its infinitesimal generator.
Then for every vector field $X \in \mathfrak{g} (m-1)$, the commutator $[Z,X]$ lies in
$\mathfrak{g}^{\infty,\ast} (m,m-1) \subseteq \mathfrak{g}^{\infty} (m,m-1)$. In fact,
all the iterated commutators $[Z,\ldots [Z,[Z,X]] \ldots ]$ lie in
$\mathfrak{g}^{\infty,\ast} (m,m-1) \subseteq \mathfrak{g}^{\infty} (m,m-1)$.\qed
\end{lemma}

We can now close this section with a technical lemma which, albeit slightly unrelated
to the preceding material, will be rather useful in the next section.

\begin{lemma}
\label{addedinJanuary2015.Number3}
Assume that $Z_1, Z_2$, and $X$ are vector fields in $\ghatX$ satisfying the following conditions:
\begin{itemize}
  \item $Z_2 =aX +bZ_1$ and $[Z_1,X]$ is everywhere parallel to~$X$.

  \item $b$ is a first integral of $X$.

  \item The time-one maps $F_1$ and $F_2$ induced respectively by $Z_1$ and $Z_2$ are such that
  the infinitesimal generator of $F_1 \circ F_2 \circ F_1^{-1} \circ F_2^{-1}$ has the form
  $hX$ for some first integral $h$ of $X$.
\end{itemize}
Then the commutator $[Z_1,Z_2]$ is everywhere parallel to~$X$ or, equivalently, $b$ is a first integral
of $Z_1$.
\end{lemma}

\begin{proof}
We can assume that $b$ is not identically zero, otherwise the statement is clear.
Note that $Z_1$ and $Z_2$ have similar properties. More precisely both $[Z_1, X]$ and $[Z_2,X]$ are everywhere
parallel to~$X$ and, in particular, they both derive first integrals of $X$ into first integrals of $X$.
Also none of these vector fields is everywhere parallel to~$X$. Denoting by ${\rm ord}\, (b)$ the order of the formal
function $b$ at $(0,0) \in \C^2$, first note the following:

\noindent {\it Claim}. Without loss of generality we can assume that ${\rm ord}\, (b) \geq 0$.

\noindent {\it Proof of the Claim}. As observed above, the roles of $Z_1$ and $Z_2$ are interchangeable. Thus
we can work either with $Z_2 = aX +bZ_1$ or with $Z_1 = \tilde{a} X + \tilde{b} Z_2$. A direct inspection in
the formulas for the coefficients $a, \, b$ and $\tilde{a}, \tilde{b}$ then shows that $\tilde{b} =1/b$.
In fact, modulo considering the obvious extensions of these vector fields to $\C^3$ the
vector product (denoted by $\wedge$) of the various vector fields in question becomes well defined. All these vector products are
pairwise parallel since their only non-zero component necessarily corresponds to the ``third'' (added) component.
Now just note that $b$ equals the ratio of $X \wedge Z_2$ and $X \wedge Z_1$ whereas $\tilde{b}$ is the ratio of
$X \wedge Z_1$ and $X \wedge Z_2$. The claim results at once.\qed

Assuming then ${\rm ord}\, (b) \geq 0$, we shall use the Campbell-Hausdorff
formula in~(\ref{finalcommutators}). More precisely, note that
$$
[Z_1, Z_2] = a_1 X + \frac{\partial b}{\partial Z_1} \, Z_1 \, .
$$
Assume aiming at a contradiction that $\partial b/\partial Z_1$ does not vanish identically and
denote by ${\rm ord}\, (\partial b/\partial Z_1)$ the order of $\partial b/\partial Z_1$
at the origin. Note that ${\rm ord}\, (\partial b/\partial Z_1)$ is strictly greater than
${\rm ord}\, (b)$ since the linear part of $Z_1$ at the origin vanishes. Hence we have
${\rm ord}\, (\partial b/\partial Z_1)\geq 1$.

The proof is reduced to check that the components in the direction of~$Z_1$ of all the remaining terms in
Campbell-Hausdorff formula~(\ref{finalcommutators}) have order strictly larger than the order
of $(\partial b/\partial Z_1) Z_1$. In the sequel the reader is reminded that $b$ and all its derivatives
with respect to $Z_1$ are first integrals for $X$. We begin with the term
$$
\frac{1}{2} \left( [Z_{1},[Z_{1}, Z_{2}]] + [Z_{2},[Z_{1},Z_{2}]] \right) \, .
$$
Recalling that $[Z_1,Z_2] = a_1 X + (\partial b/\partial Z_1) Z_1$, we first obtain
$$
[Z_1, [Z_1,Z_2]] = [Z_1, a_1 X + (\partial b/\partial Z_1) Z_1] = a_2 X + \frac{\partial^2 b}{\partial Z_1^2} \, Z_1 \, .
$$
Since the linear part of $Z_1$ at the origin equals zero, there follows that the order
of $(\partial^2 b/\partial Z_1^2) Z_1$ is strictly greater than the order of
$(\partial b/\partial Z_1) Z_1$ as desired. Concerning the term $[Z_{2},[Z_{1},Z_{2}]]$, we have
$$
[Z_{2},[Z_{1},Z_{2}]] = [aX + bZ_1 , a_1 X + (\partial b/\partial Z_1) Z_1] = a_2X +
\left( b \left( \frac{\partial^2 b}{\partial Z_1^2} \right) - \left( \frac{\partial b}{\partial Z_1} \right)^2
\right) \, Z_1 \, .
$$
Since ${\rm ord}\, (b) \geq 0$, there follows again that the order of $b (\partial^2 b/\partial Z_1^2) Z_1$ is
is strictly greater than the order of $(\partial b/\partial Z_1) Z_1$. Similarly, since
${\rm ord}\, (\partial b/\partial Z_1)\geq 1$, the order of $(\partial b/\partial Z_1)^2 Z_1$ is strictly
greater than the order of $(\partial b/\partial Z_1) Z_1$. The proof of the lemma now results from a straightforward
induction argument.
\end{proof}

\section{Proof of Theorem~\ref{commuting9}}

To better organize the discussion, Theorem~\ref{commuting9} will be proved by gradually
increasing the complexity of the solvable group $G (m,m-1)$. The simplest possible structure for $G (m,m-1)$ corresponds
to an abelian group and this case is handled by Proposition~\ref{firtcasePropositioncommuting9}
below. In the sequel we always keep the notation used in Section~6.

\begin{prop}
\label{firtcasePropositioncommuting9}
Assume that the group $G (m,m-1) \subset \formdiffn$
is abelian. Then the initial group $G$ is solvable.
\end{prop}

\begin{proof}
Note that, by definition, none of the sets $S(m)$ and $S(m-1)$ is reduced to the identity.
The dimension of the Lie algebra $\mathfrak{g} (m,m-1)$ is finite since $G (m,m-1)$ is abelian
and finitely generated by construction. On the other hand, Lemma~\ref{nodimension-1-June2015}
ensures that the dimension of $\mathfrak{g} (m-1)$, and hence the dimension of $\mathfrak{g} (m,m-1)$,
is at least~$2$.

According to Corollary~\ref{newversionLemma11.2},
either $\mathfrak{g} (m,m-1)$ coincides with the linear span of two (non-everywhere parallel, commuting) vector fields
$X$ and $Y$ or it is generated by (finitely many) vector fields having the form $hX$ where $h$ is a first integral
of~$X$. Since the dimension of $\mathfrak{g} (m,m-1)$ is at least~$2$, in the latter case
there also follows that $\mathfrak{g} (m,m-1)$ contains $X$ and some other vector field
$Y=hX$, where $h$ is a non-constant first integral of $X$.

Assume first that $\mathfrak{g} (m,m-1)$ coincides with the linear span of vector fields $X$ and $Y$.
Since the dimension of $\mathfrak{g} (m-1)$ is at least~$2$, there follows that these two Lie algebras
should coincide. In other words, relation~(\ref{Conjugatingsolvablegroups-10}) implies that every diffeomorphism
$F \in S(m-2)$ should leave $\mathfrak{g} (m,m-1)$ invariant. By virtue of the material in Section~5.2, we conclude
that $G(m-1,m-2)$ is a subgroup of $\Gamma_{{\rm abelian}-1}$ and hence it is solvable.

Summarizing the preceding, to prove our proposition we can assume that $\mathfrak{g} (m,m-1)$ consists of
vector fields which are given as the product of $X$ by a first integral of itself.
Also, we can assume that both
$X$ and some vector field $Y=hX$ lie in $\mathfrak{g} (m-1)$, where $h$ is a non-constant first integral of~$X$.

Assume first that $\mathfrak{g} (m-1)$ coincides with $\mathfrak{g} (m,m-1)$. In this case
relation~(\ref{Conjugatingsolvablegroups-10}) implies again that every diffeomorphism
$F \in S(m-2)$ should leave $\mathfrak{g} (m,m-1)$ invariant. There follows that $G(m-1,m-2)$ is a subgroup of
$\Gamma_{{\rm abelian}-2}$ and hence solvable (cf. Section~5.2). In other words, we can assume that
$\mathfrak{g} (m-1)$ is strictly contained in $\mathfrak{g} (m,m-1)$.

More generally let $S(m-2) = \{ F_1, \ldots ,F_s \}$. For every $i=1, \ldots ,s$, denote
by $Z_i$ the infinitesimal generator of $F_i$. Next, fixed $i$, we know that $F_i^{\ast} X$ and
$F_i^{\ast} Y = F_i^{\ast} (hX)$ are both equal to first integrals of $X$ times $X$ itself. From this, we conclude
that the Lie algebra $\{ \varphi X\}$ consisting of all vector fields of the form $\varphi X$, where $\varphi$ is a first integral
of~$X$, if left invariant by all the formal diffeomorphisms  $F_i$, $i=1, \ldots ,s$. In other words,
in the statement of Lemma~\ref{addedinJanuary2015}, we can choose
$\mathfrak{g}^{\infty,\ast} (m,m-1) = \mathfrak{g}^{\infty} (m,m-1) = \{ \varphi X\}$.
In particular, Lemma~\ref{addedinJanuary2015} yields
$$
[Z_i,X] = h_i X
$$
for some first integral $h_i$ of $X$, $i=1, \ldots, s$.

Assume now that all
the vector fields $Z_i$ are everywhere parallel to $X$. Let $Z_i =a_i X$. Since
$[Z_i,X] = h_i X$, there follows that $\partial a_i /\partial X$ is a first integral of
$X$. Therefore the Lie algebra $\mathfrak{g} (m-1,m-2)$ is as in Case~1 of Section~5.3. In particular
$\mathfrak{g} (m-1,m-2)$ is solvable and this establishes the proposition in the case in question.

To complete the proof of the proposition there only remains to consider the case where not all the vector fields
$Z_1 , \ldots ,Z_s$ are everywhere parallel to~$X$. We can then assume that $Z_1$ is not everywhere parallel
to $X$. Note, however, that all the vector fields $Z_i$ still derive first integrals of
$X$ into first integrals of $X$ since $[Z_i,X] = h_i X$ (see Formula~(\ref{generalizedSchwarz})).
Now for $i \in \{ 2, \ldots ,s\}$, we set
$Z_i = a_iX + b_iZ_1$. Since $[Z_i,X] = h_{i} X$, we conclude that both
$\partial a_i /\partial X$ and $b_i$ are first integral of $X$. Owing to Lemma~\ref{addedinJanuary2015.Number3},
we therefore conclude that $[Z_1, Z_i]$ is everywhere parallel to $X$, for every $i=1, \ldots ,s$.
However, the condition of having $[Z_1, Z_i]$ everywhere parallel to $X$ implies that
$b_i$ must be a first integral for $Z_1$. Therefore $b_i$
is actually constant since it is also a first integral for $X$ (and $X$ and $Z_1$ are
not everywhere parallel). The solvable nature of the Lie algebra in question is now clear and this completes
the proof of the proposition.
\end{proof}

From now on we always assume that the {\it finitely generated solvable group}\, $G (m,m-1)$ is not abelian.
Denote by $D^s G (m,m-1)$ the non-trivial derived subgroup of $G (m,m-1)$ having highest
order~$s$. $D^s G (m,m-1)$ is also the only non-trivial abelian derived subgroup
of $G (m,m-1)$. Furthermore, we have $s \in \{ 1, 2\}$; cf. \cite{ribon}
or Section~5.3. Note however that the non-trivial abelian group $D^s (m,m-1)$
may fail to be finitely generated.
The abelian Lie algebra associated with $D^s G (m,m-1)$ will be denoted by
$D^s \mathfrak{g} (m,m-1)$. Then, we have:

\begin{lemma}
\label{newversionLemma33.1}
Suppose that $D^s \mathfrak{g} (m,m-1)$ coincides with the linear span of two vector fields $X$ and $Y$.
Then the initial group $G$ is solvable.
\end{lemma}

\begin{proof}
To begin with let $S(m-2) = \{ F_1, \ldots ,F_s \}$. The infinitesimal generator of $F_i$ will be
denoted by $Z_i$, $i=1 ,\ldots ,s$.
Recall that $D^s G (m,m-1)$ is a normal subgroup of $D^{s-1} G (m,m-1)$ which, in turn, is not an abelian group.
Thus Lemma~\ref{lastversionLemma3} ensures that $D^{s-1} G (m,m-1) \subset G(m,m-1)$ is isomorphic
to a non-abelian subgroup of $\Gamma_{{\rm abelian}-1}$. In turn, Lemma~\ref{newversionLemma22.1} shows that
the normalizer of $D^{s-1} G (m,m-1)$ is metabelian which implies that $D^{s-1} G (m,m-1) = G(m,m-1)$ i.e.,
we necessarily have $s=1$. In other words, $G(m,m-1)$ is isomorphic to a non-abelian subgroup of $\Gamma_{{\rm abelian}-1}$
and hence $\mathfrak{g} (m,m-1)$ coincides with the Lie algebra $\mathfrak{g}_{{\rm abelian}-1}$ of $\Gamma_{{\rm abelian}-1}$.
Thus $\mathfrak{g} (m,m-1)$ is generated by vector fields $X$, $Y$, and $\widetilde{Z}$ such that
$[X,Y]=[X,\widetilde{Z}]=0$ and $[Y ,\widetilde{Z}] = cX$ for some $c \in \C$.

Consider the Lie algebra $\mathfrak{g} (m-1)$. We can
assume that $\mathfrak{g} (m-1)$ is strictly contained in $\mathfrak{g} (m,m-1)$ otherwise
$S(m-2)$ is contained in the normalizer of $\Gamma_{{\rm abelian}-1}$ coinciding with $\Gamma_{{\rm abelian}-1}$
itself (see Lemma~\ref{nextlemmabelow}).
Similarly, owing to Lemma~\ref{nodimension-1-June2015}, we can assume that the dimension of
$\mathfrak{g} (m-1)$ is strictly larger than~$1$. Hence, the Lie algebra $\mathfrak{g} (m-1)$ must have dimension
equal to~$2$ and, since it is a sub-algebra of $\mathfrak{g} (m,m-1) \simeq \mathfrak{g}_{{\rm abelian}-1}$,
it is therefore abelian.

Since $\mathfrak{g} (m-1)$ is abelian, there follows that it
must contain $X$ since $X$ lies in the center of $\mathfrak{g}_{{\rm abelian}-1}$. Recalling from the proof of
Lemma~\ref{nextlemmabelow} that $X$ is distinguished in $\mathfrak{g}_{{\rm abelian}-1}$ as the vector field of maximal
order (up to constant multiples), there follows that $X$ also belongs to $F_i^{\ast} \mathfrak{g} (m-1) \subset
\mathfrak{g} (m,m-1) \simeq \mathfrak{g}_{{\rm abelian}-1}$, for every $i=1, \ldots, s$. Indeed, for every $F_i \in
S(m-2)$, we must have $F_i^{\ast} X =X$ since $F_i$ is unipotent.

Consider now another vector field $\overline{Y}$ in $\mathfrak{g} (m-1)$ which is linearly independent with $X$. In principle
$\overline{Y}$ may be everywhere parallel to~$X$. However, if $\overline{Y}$ is everywhere parallel to~$X$,
then the abelian sub-algebra generated by $X$ and by $\overline{Y}$ is the (unique) maximal abelian sub-algebra
of $\mathfrak{g}_{{\rm abelian}-1}$ consisting of vector fields everywhere parallel to~$X$. In fact, the vector field
$Y$ is not everywhere parallel to~$X$ (by assumption) and this implies the preceding assertion.
A similar conclusion holds for the Lie algebra $F_i^{\ast} \mathfrak{g} (m-1) \subset
\mathfrak{g} (m,m-1) \simeq \mathfrak{g}_{{\rm abelian}-1}$, since $F_i^{\ast} X =X$. From the maximal character of
the Lie algebras in question, we therefore conclude that
$F_i^{\ast} \mathfrak{g} (m-1) = \mathfrak{g} (m-1)$; i.e. $\mathfrak{g} (m-1)$ is invariant by $F_i$, for every
$i=1, \ldots, s$. In other words,
$S(m-2)$ is contained in the normalizer of a two-dimensional abelian
Lie algebra and hence the group generated by $S(m-1) \cup S(m-2)$ is solvable, see Lemma~\ref{commuting7nowlemma}.

We can now assume that $X$ and $\overline{Y}$ are not everywhere parallel. We still have $F_i^{\ast} X =X$ which
implies that $[Z_i ,X]=0$. On the other hand, $\overline{Y} = \overline{a} X + \overline{b} Y$ since
$[X, \overline{Y}]=0$ (where $\overline{a}$ and $\overline{b}$ are first integrals of~$X$). Since
$[Y, \overline{Y}]$ must be a multiple of~$X$, there follows that $\overline{b}$ is also a first integral of~$Y$
and, therefore, an actual constant.

To complete the proof we now proceed as follows. For $i$ fixed, $F_i^{\ast} (\overline{Y}) =
(\overline{a} \circ F_i) X + \overline{b} F_i^{\ast} Y$. Since this vector field still belongs to
$\mathfrak{g} (m,m-1) \simeq \mathfrak{g}_{{\rm abelian}-1}$, we conclude that $F_i^{\ast} Y$ is still a constant.
Because $F_i$ is unipotent (i.e. tangent to the identity), this constant must be~$1$ so that $F_i$ actually preserves
both $X$ and $Y$. As already seen, this implies that $F_i$ is actually contained in the abelian group generated
by the exponentials of $X$ and $Y$. It is now clear that $G(m-1,m-2)$ is still solvable and this completes the proof
of the lemma.
\end{proof}

In view of Lemma~\ref{newversionLemma33.1} we assume in what follows that $D^s \mathfrak{g} (m,m-1)$ is fully constituted by
vector fields of the form $hX$, where $h$ is a first integral for $X$ (by way of notation, we also suppose that
$X$ itself belongs to $D^s \mathfrak{g} (m,m-1)$). Note that the dimension of $D^s \mathfrak{g} (m,m-1)$ is finite
if and only if $D^s G (m,m-1)$ is finitely generated. In this case, the group
$D^{s-1} G (m,m-1)$ has non-trivial center: since inner automorphisms of
$D^{s-1} G (m,m-1)$ leave $D^s G (m,m-1)$ invariant, they must also leave invariant
those vector fields in $D^s \mathfrak{g} (m,m-1)$
having maximal order at the origin of $\C^2$ (as follows from Hadamard lemma, cf. Section~5).
In turn, the center of $D^{s-1} G (m,m-1)$ must be contained in the exponential of a
single vector field $X$ otherwise a contradiction
would arise from Lemma~\ref{lastversionLemma2}. Thus those elements in the intersection
of the exponential of $X$ with the group $D^s G (m,m-1)$ lie in the center of
$D^{s-1} G (m,m-1)$ proving our assertion. Next note that the center of $G (m,m-1)$
is non-trivial if and only if the center of $D^{s-1} G(m,m-1)$ is non-trivial.
In fact, if the center of $D^{s-1} G$ is non-trivial, then the chain of normal subgroups
$$
D^s G (m,m-1) \triangleleft D^{s-1} G (m,m-1) \triangleleft \cdots \triangleleft G (m,m-1)
$$
implies that $D^{i-1} G$ normalizes $D^i G$ so that $D^{i-1} G$ should also normalize the center
of $D^i G$. Hence the center of $D^{s-1} G$ lies also in the center of $G$. The converse is clear.

In the general case, however, the center of $G (m,m-1)$ may be
trivial. Furthermore (non-abelian) solvable subgroups of $\formdiffn$ having non-trivial center are
easy to characterize. In fact, let $G$ be a (non-abelian) solvable subgroup of $\formdiffn$
and denote by $D^s G$ the (non-trivial) abelian derived
subgroup of $G$ (it is the non-trivial derived subgroup of maximal order~$s$).

\begin{lemma}
\label{newversionLemma33.2}
Let $G$ and $D^s G$ be as above. Assume that the Lie algebra $D^s \mathfrak{g}$ associated with $D^s G$ is constituted
by vector fields everywhere parallel to a certain vector field $X$. Assume also that $G$ has non-trivial center.
Then $s=1$. Moreover the Lie algebra $\mathfrak{g}$ associated with $G$
is constituted by vector fields of the form $aX +\alpha Y$ where $a$ is a first integral of $X$ and where $\alpha \in \C$.
Moreover $X,Y$ are non-everywhere parallel commuting vector fields.
In particular the center of $G$ is (non-trivial and) contained in ${\rm Exp}\, (tX)$.
\end{lemma}

\begin{proof}
Ultimately the result is just a special case of the classification of groups presented in Section~5.3.
For the convenience of the reader, we shall provide a self-contained argument. First, according to
the previous discussion, we know that the center of $G$ is trivial if and only if the center of $D^{s-1} G$
is so. By assumption, in the present case none of these centers turns out to be trivial.
Still owing to the above discussion, we assume that
$X \in D^s \mathfrak{g}$ is such that its exponential contains the center of $D^{s-1} G$ and of $G$. Therefore
every vector field in $\mathfrak{g}$ has the form $aX +bY$ where $a, b$ are first
integrals of $X$ and where $Y$ is a vector field commuting with $X$ and not everywhere parallel to~$X$. The reader
will also note that a vector field $Y$ as indicated must exist since $G$ would be abelian otherwise.

Consider now the Lie algebra $D^{s-1} \mathfrak{g}$ associated to $D^{s-1} G$. The commutator of two vector fields
$Z_1, Z_2 \in D^{s-1} \mathfrak{g}$ must be contained in $D^s \mathfrak{g}$ and hence it must have the form $hX$
where $h$ is some first integral of $X$. Setting $Z_1 =a_1X +b_1Y$ and $Z_2 =a_2 X +b_2Y$, the preceding
implies that $b_1/b_2$ must be a constant unless one between $b_1, b_2$ vanishes identically.
In other words, there must exist a function $f \in \fieldC$ such that the following holds:

\noindent {\it Claim}.
Every vector field $Z \in D^{s-1} \mathfrak{g}$ has the from $Z =aX +\alpha fY$ where $a, \, f$ are first integrals of $X$
and $\alpha$ is a constant in $\C$ depending on $Z$.\qed

We also note that the general form of the quotient $b_1/b_2$ satisfies
the co-cycle relation $(b_1/b_2) (b_2/b_3) = b_1/b_3$ which is necessary to have a well-defined Lie algebra.
Furthermore, the vector fields in $D^s \mathfrak{g}$ sits inside the above mentioned form (just take $\alpha =0$).

Suppose now that $s\geq 2$ so that the Lie algebra $D^{s-2} \mathfrak{g}$ can be considered. The preceding
argument can thus be repeated: let $Z_1, Z_2$ be vector fields in $D^{s-2} \mathfrak{g}$ leading to a commutator
$[Z_1,Z_2]$ in $D^{s-1} \mathfrak{g} \setminus D^k \mathfrak{g}$. Letting $Z_1 =a_1X +b_1Y$ and $Z_2 =a_2 X +b_2Y$,
we obtain
$$
\frac{\partial (b_1/b_2)}{\partial Y} = \alpha f
$$
for some $\alpha \in \C$ and for $f$ as in the above claim. Naturally it can be supposed that $f$ is not a constant.
If $H$ is a specific function satisfying
$\partial H /\partial Y =f$, then the quotient $b_1/b_2$ has the general form $\alpha H +\varphi$ where $\varphi$ is
a first integral of $Y$. Nonetheless, to have a well-defined Lie algebra, we still need to check the co-cycle
relation $(b_1/b_2) (b_2/b_3) = b_1/b_3$. In particular $b_2/b_1$ must admit the same pattern i.e., we must
have $b_2/b_1 = \overline{\alpha} H +\overline{\varphi}$, for a suitable constant
$\overline{\alpha} \in \C$ and first integral $\overline{\varphi}$ of $Y$. Furthermore,
the fact that $(b_1/b_2)(b_2/b_1) =1$
immediately leads to $\alpha \overline{\alpha} H^2 +
H (\alpha \overline{\varphi} + \overline{\alpha} \varphi) + \varphi \overline{\varphi}=1$. Therefore,
by taking the derivative with respect to~$Y$, we obtain
$$
(2 \alpha\overline{\alpha} H + \alpha \overline{\varphi} + \overline{\alpha} \varphi) \, . \, \frac{\partial H}{\partial Y}
=(2 \alpha \overline{\alpha } H + \alpha\overline{\varphi} + \overline{\alpha} \varphi) \, . \, f = 0 \, .
$$
Since $f$ is not identically zero, it follows that $H$ must be a first integral for $Y$ since $\varphi, \,
\overline{\varphi}$ are so. In any event, a contradiction arises at once.
From this contradiction, we conclude that $s$ equals~$1$. The lemma then follows by replacing $Y$ by $fY$, cf.
Remark~\ref{justaremarkmoreorlessuseless-1}.
\end{proof}

In what follows we always set $S(m-2) = \{ F_1, \ldots ,F_s \}$ while the infinitesimal
generator of $F_i$ will be denoted by $Z_i$.
Before discussing the case in which the group $G (m,m-1) \subset \formdiffn$ has non-trivial center,
it is however convenient to settle the following special case:

\begin{lemma}
\label{newversionLemma33.5}
Assume that $G(m,m-1)$ is as in Case~1 of Section~5.3; i.e. the infinitesimal generator of every element
in $G(m,m-1)$ is parallel to a certain vector field~$X$ (and $G(m,m-1)$ is not abelian). Then the
initial group $G$ is solvable.
\end{lemma}

\begin{proof}
The Lie algebra $\mathfrak{g} (m,m-1)$ associated with $G(m,m-1)$ is thus formed by vector fields
having the form $(\varphi_1 f + \varphi_2) X$ where $\varphi_1, \varphi_2$ are first integrals of $X$
and where $f$ satisfies $\partial f/\partial X =\overline{h}$ for some non-identically zero first integral
$\overline{h}$ of~$X$. Furthermore, by virtue of Lemma~\ref{nodimension-1-June2015}, we can assume that
the dimension of $\mathfrak{g} (m-1)$ is at least~$2$.

For every $Y \in \mathfrak{g} (m-1)$ and $i \in \{1, \ldots ,s \}$, the vector field $F_i^{\ast} Y$ lies
in $\mathfrak{g} (m,m-1)$ and hence is everywhere parallel to~$X$. There follows that the solvable Lie algebra
$\mathfrak{g}^{\infty,\ast} (m,m-1)$ (as well as $\mathfrak{g}^{\infty} (m,m-1)$) constructed in
Lemma~\ref{universalsolvablealgebra-June2015} is fully constituted by vector fields everywhere parallel to~$X$.
Hence they are still contained in the algebra $\{ (\varphi_1 f + \varphi_2) X \}$ consisting
of all vector fields having the form $(\varphi_1 f + \varphi_2) X$ indicated
above. Now, for $Y$ in $\mathfrak{g} (m-1)$, Lemma~\ref{addedinJanuary2015} ensures that the
commutator $[Z_i,Y]$ lies in $\{ (\varphi_1 f + \varphi_2) X \}$.
Since all the vector fields are everywhere parallel to $X$, we conclude that $[Z_i, X]$ is everywhere parallel
to~$X$ as well. In particular $Z_i$ derives first integrals of $X$ into first integrals of $X$.

Suppose now that all the vector fields $Z_i$ are everywhere parallel to $X$. Set $Z_i = a_i X$ and let
$[Z_i, X] = (-2 \varphi_{1,i}  \overline{h} f - \varphi_{2,i} \overline{h}) X$. We then have $\partial a_i /\partial X =
2\varphi_{1,i} \overline{h} f + \varphi_{2,i} \overline{h}$ so that $a_i = \varphi_{1,i} f^2 + \varphi_{2,i} f + \varphi_{3,i}$
where $\varphi_{3,i}$ is another first integral for $X$. Another application of Lemma~\ref{addedinJanuary2015}
ensures that $[Z_i,[Z_i,X]]$ belongs to $\{ (\varphi_1 f + \varphi_2) X \}$ as well.
A direct computation of $[Z_i,[Z_i,X]]$, however, yields
$$
[Z_i,[Z_i,X]] = 2 \varphi_{1,i}^2 \overline{h}^2 f^2 + \widetilde{\varphi}_{2,i} f + \widetilde{\varphi}_{3,i}
$$
for suitable first integrals $\widetilde{\varphi}_{2,i}, \, \widetilde{\varphi}_{3,i}$ of~$X$.
Since $[Z_i,[Z_i,X]]$ lies in $\mathfrak{g} (m,m-1)$, we conclude that $\varphi_{1,i}$ vanishes identically.
Therefore the Lie algebra generated by $\mathfrak{g} (m-1)$ and the vector fields $Z_1, \ldots ,Z_s$ is still solvable.

It remains to consider the case in which not all the vector fields $Z_i$ are everywhere parallel to~$X$.
We can then assume that $Z_1$ is not everywhere parallel to~$X$. Yet, the reader is reminded that $Z_1$
derives first integrals of $X$ into first integrals of $X$.

For $i =2, \ldots ,s$, we set $Z_i = a_i X + b_i Z_1$. Since $[Z_i,X]$ is everywhere parallel to~$X$, we still
conclude that $b_i$ is a first integral of~$X$. On the other hand, the Lie algebra $\mathfrak{g} (m-1)$ contains
the infinitesimal generators of the commutators $F_1 \circ F_i \circ F_1^{-1} \circ F_i^{-1}$ so that
Lemma~\ref{addedinJanuary2015.Number3} ensures that $[Z_1,Z_i]$ is everywhere parallel to~$X$. In other words,
all the coefficients $b_i$ are constants in $\C$. Hence to complete the proof of the lemma it suffices
to check that $a_i$ has the form $\varphi_{1,i} f + \varphi_{2,i}$ for suitable first integrals
$\varphi_{1,i}$ and $\varphi_{2,i}$ of~$X$. This straightforward verification is left to the reader since it
essentially amounts to keeping track of the components parallel to~$X$ of the indicated vector fields by
repeating the argument used in the case where all the vector fields $Z_1, \ldots ,Z_s$
are everywhere parallel to~$X$. The proof of the lemma is completed.
\end{proof}

Now we state:

\begin{prop}
\label{firtcasePropositioncommuting9-furtherstep}
Keeping the preceding notation, assume that
the non-abelian solvable group $G (m,m-1) \subset \formdiffn$ has non-trivial center. Then the initial group $G$ is solvable.
\end{prop}

\begin{proof}
We keep the preceding notation so that $S(m-2) = \{ F_1, \ldots ,F_s \}$ and the infinitesimal generator of
$F_i$ is denoted by $Z_i$.
By assumption, the solvable Lie algebra $\mathfrak{g} (m,m-1)$ is an in Lemma~\ref{newversionLemma33.2}. Recall also
Lemma~\ref{nodimension-1-June2015} allows us to assume that the dimension of $\mathfrak{g} (m-1) \subset
\mathfrak{g} (m,m-1)$ is at least~$2$.

The proof of the proposition will be split into two cases according to whether or not $\mathfrak{g} (m-1)$ is abelian.

\noindent {\it Case A}. Assume that $\mathfrak{g} (m-1)$ is abelian.

We begin by considering the abelian sub-algebras of $\mathfrak{g} (m,m-1)$ having dimension at least~$2$. Owing to the
description of $\mathfrak{g} (m,m-1)$ provided by Lemma~\ref{newversionLemma33.2}, these
algebras fall into two classes, namely:
\begin{enumerate}
  \item Lie algebras of dimension~$2$ containing non everywhere parallel vector fields. This type of Lie
  algebra has one of the following forms:
  \begin{itemize}
    \item It may be generated by $X$ and by another vector field $Y$ having the form
  $aX + \alpha Y$, with $\alpha \in \C^{\ast}$.

    \item It may be generated by vector fields of the form $aX + \alpha Y$ and $ca X + \beta Y$ where
    $c$, $\alpha$ and $\beta$ are all constants. Moreover $c \neq \beta/\alpha$.
  \end{itemize}

  \item Lie algebras constituted by vector fields that are everywhere parallel to $X$ (and hence of
  the form $hX$ for some first integral $h$ of $X$).
\end{enumerate}

Consider first the case where $\mathfrak{g} (m-1)$ is as in item~(1) above.
Consider also the Lie algebras $\mathfrak{g}^{\infty,\ast} (m,m-1)$ and $\mathfrak{g}^{\infty} (m,m-1)$
provided by Lemma~\ref{universalsolvablealgebra-June2015}. The solvable Lie algebra
$\mathfrak{g}^{\infty,\ast} (m,m-1)$ contains $\mathfrak{g} (m-1)$ and hence it
is not fully constituted by vector fields everywhere parallel to~$X$. Thurefore it must be as in Cases~2 or~3 of
Section~5.3; i.e. vector fields in $\mathfrak{g}^{\infty,\ast} (m,m-1)$ have the form
$(\varphi_1 f + \varphi_2) X + \alpha Y$ with $\alpha \in \C$ (the possibility of always having $\varphi_1=0$
is not excluded either).

Fix $\overline{Y} \in \mathfrak{g} (m-1)$
which is not everywhere parallel to~$X$. Hence we have $\overline{Y} = aX + \alpha Y$ for some $\alpha \in \C^{\ast}$.
According to Lemma~\ref{addedinJanuary2015},
all the iterated commutators $[Z_i, \ldots [Z_i, \overline{Y}]\ldots]$ lie in $\mathfrak{g}^{\infty,\ast} (m,m-1)$.
However, as we iterate these commutators, the orders of
the resulting vector fields keep increasing strictly since the linear parts of all the involved vector fields are
zero. Since the components in the direction $Y$ have all fixed order (they only differ by a multiplicative constant),
there follows that some sufficiently high commutator will be everywhere parallel to~$X$. A further iteration of this
commutator will still be everywhere parallel to~$X$ for the same reason. From this there follows that
$[Z_i ,X]$ must be everywhere parallel to~$X$. Moreover, we also have:

\noindent {\it Claim}. $[Z_i, Y]$ is everywhere parallel to~$X$.

\noindent {\it Proof of the Claim}. Consider the first iterated
commutator $[Z_i, [Z_i, \ldots [Z_i, \overline{Y}]\ldots]]$ which is everywhere parallel to~$X$. The preceding
iterated commutator $[Z_i, \ldots [Z_i, \overline{Y}]\ldots]$ then still has the form $aX + \beta Y$ for some
$\beta \in \C^{\ast}$. Therefore the commutator $[Z_i, aX +\beta Y]$ is everywhere parallel to~$X$. However the commutator
$[Z_i, aX]$ is everywhere parallel to~$X$ as well since so is $[Z_i, X]$. Therefore the commutator $[Z_i, \beta Y]$ must
be everywhere parallel to~$X$ as well and this completes the proof of the claim.\qed

Next set $Z_i = a_i X + b_iY$. Since $[Z_i, X]$ is everywhere parallel to~$X$, there follows that $b_i$ is
a first integral of $X$. Similarly $b_i$ is also a first integral of $Y$ since $[Z_i ,Y]$ is everywhere parallel
to~$X$. In other words, $b_i$ is constant. Finally $a_i$ must have the form $\varphi_1 f + \varphi_2$ as now
follows from considering commutators $[Z_i, \overline{Y}_1]$ and $[Z_i,\overline{Y}_2]$ in
$\mathfrak{g}^{\infty,\ast} (m,m-1)$ for two linearly independent vector fields $\overline{Y}_1$ and $\overline{Y}_2$
in $\mathfrak{g} (m-1)$. Therefore the group $G(m-1,m-2)$ is again solvable and this prove the
proposition in the present case.

To finish the discussion of Case~A, suppose now that $\mathfrak{g} (m-1)$ is fully constituted by vector fields having
the form $hX$ where $h$ is a first integral of $X$. In this case $F_i^{\ast} (\mathfrak{g} (m-1)) \subset
\mathfrak{g} (m,m-1)$ is an abelian
sub-algebra of $\mathfrak{g} (m,m-1)$ fully constituted by pairwise everywhere parallel vector fields. Since the dimension
of $\mathfrak{g} (m-1)$ is at least~$2$, the description above of the abelian sub-algebras of
$\mathfrak{g} (m,m-1)$ ensures that $F_i^{\ast} (\mathfrak{g} (m-1))$ is again formed by vector fields everywhere parallel
to~$X$. In other words, the commutator $[Z_i, X]$ is everywhere parallel to~$X$. In particular
$\mathfrak{g}^{\infty,\ast} (m,m-1)$ has the form $(\varphi_1 f + \varphi_2) X$ as in Case~1 of Section~5.3.

Again let $Z_i = a_i X + b_iY$ so as to conclude that $b_i$ is a first integral of~$X$ from the fact that
$[Z_i, X]$ is everywhere parallel to~$X$. The crucial point here compared
to the previous case lies in the fact that only commutators of $Z_i$ with vector fields everywhere parallel to~$X$
are controlled which, in turn, prevents us from repeating the above argument
to conclude that $b_i$ is a constant. To overcome this difficulty we proceed as follows.

Assume first that $s=1$ so that $Z_1 = a_1 X + b_1Y$ with $b_1$ being a first integral of $X$. To conclude
that $G(m-1,m-2)$ is solvable, it is therefore sufficient to check that $a_1$ has the above indicated form
$\varphi_1 f + \varphi_2$. This however follows from the same computations carried out in the proof
of Lemma~\ref{newversionLemma33.5}. More precisely, consider two linearly independent vector fields
$(\varphi_3 f + \varphi_4) X$ and $(\varphi_5 f + \varphi_6) X$ in $\mathfrak{g} (m-1)$. Owing to
Lemma~\ref{addedinJanuary2015}, the commutators $[Z_1, (\varphi_3 f + \varphi_4) X]$ and
$[Z_1, (\varphi_5 f + \varphi_6) X]$ still possesses the general form $(\varphi_1 f + \varphi_2) X$. From this
there follows that $a_1$ has the general form $\varphi_1 f + \varphi_2$ and completes the proof of the proposition
in the present case.

Assume now that $s \geq 2$. Without loss of generality, we can assume that $b_1$ is not identically zero.
In other words, $Z_1$ is not everywhere parallel to~$X$. Now, following the argument given at the end of
the proof of Lemma~\ref{newversionLemma33.5}, we set
$Z_i = \overline{a}_i X + \overline{b}_i Z_1$, for $i =2, \ldots ,s$.
Since $[Z_i,X]$ is everywhere parallel to~$X$, we still
conclude that $\overline{b}_i$ is a first integral of~$X$. On the other hand, the Lie algebra $\mathfrak{g} (m-1)$ contains
the infinitesimal generators of the commutators $F_1 \circ F_i \circ F_1^{-1} \circ F_i^{-1}$ so that
Lemma~\ref{addedinJanuary2015.Number3} ensures that $[Z_1,Z_i]$ is everywhere parallel to~$X$. In other words,
all the coefficients $\overline{b}_i$ are constants in $\C$. The proof of Proposition~\ref{firtcasePropositioncommuting9-furtherstep}
in Case~A is completed.

\noindent {\it Case B}. Assume that $\mathfrak{g} (m-1)$ is not abelian (and thus it is metabelian).

Since $\mathfrak{g} (m-1)$ is not abelian, it necessarily contains a vector field $\overline{Y}$ of the form
$aX + \alpha Y$ with $\alpha \neq 0$. Moreover
$D^1 \mathfrak{g} (m-1)$ is non-trivial and automatically constituted by vector fields
of the form $hX$, $h$ first integral of $X$. Clearly $F_i^{\ast} (D^1 \mathfrak{g} (m-1))$ is contained in
$D^1 (F_i^{\ast} \mathfrak{g} (m-1)) \subset D^1 \mathfrak{g} (m,m-1)$. There follows again that $F_i^{\ast} (X)$
is everywhere parallel to~$X$ (or equivalently that $[Z_i, X]$ is everywhere parallel to~$X$). Note however that this
conclusion can also be obtained by repeating the argument in the beginning of the proof of Case~A. Similarly, the
argument employed in the proof of the preceding Claim also applies to the present situation and implies
that $[Z_i ,Y]$ is everywhere parallel to~$X$. Hence we can again set
$Z_i =a_i X + b_iY$ where $b_i$ is a constant (for all $i=1, \ldots, s$).
Finally $a_i$ must have the form $\varphi_1 f + \varphi_2$ as now
follows from considering commutators $[Z_i, aX + \alpha Y]$ and $[Z_i,hX]$ in
$\mathfrak{g}^{\infty,\ast} (m,m-1)$ for $\overline{Y} = aX + \alpha Y$ as above and some other vector field
$hX \in \mathfrak{g} (m-1)$, where $h$ is some first integral of~$X$. The proof of
Proposition~\ref{firtcasePropositioncommuting9-furtherstep} is now completed.
\end{proof}

To prove Theorem~\ref{commuting9} it only remains to discuss the
general case of a solvable group with trivial center. In fact, given that Case~1 of Section~5.3
was already settled by Lemma~\ref{newversionLemma33.5},
we can assume that $G(m,m-1)$ is as in Case~2 or in Case~3 of Section~5.3. In particular
there exists a vector field $\ooY$ which is not everywhere parallel to~$X$ and satisfies
the condition indicated in the above mentioned Case~2. The reader is also reminded that the highest order
non-trivial derived Lie algebra $D^s \mathfrak{g} (m,m-1)$ of $\mathfrak{g} (m,m-1)$ consists of
vector fields of the form $hX$ where $h$ is a first integral for $X$. Without loss of generality we also
assume that $X \in D^s \mathfrak{g} (m,m-1)$.
We are finally able to prove Theorem~\ref{commuting9}.

\begin{proof}[Proof of Theorem~\ref{commuting9}]
As mentioned, it suffices to consider the situations where $G(m,m-1)$ is as
in Case~2 and in Case~3. However, to abridge notation, we shall only deal with
Case~3 since Case~2 can be regarded as a particular one.

Therefore the Lie algebra $\mathfrak{g} (m,m-1) \subset \mathfrak{g}_{\rm step-3}$
consists of vector fields having the form
$$
(\varphi_1 f + \varphi_2)X + \alpha \ooY \,
$$
where $\alpha \in \C$ and $\varphi_1, \varphi_2, f$ are first integrals of $X$.
Moreover $[\ooY, X] = hX$ where $h$ is a first integral of~$X$, possibly vanishing identically.
The above relation also ensures that $\ooY$ derives first integrals of $X$ into first integrals
of~$X$.

Next note that $\mathfrak{g}_{\rm step-3}$ is the largest solvable Lie algebra of $\ghatX$
so that both solvable Lie algebras
$\mathfrak{g}^{\infty,\ast} (m,m-1)$ and $\mathfrak{g}^{\infty} (m,m-1)$ are naturally contained
in $\mathfrak{g}_{\rm step-3}$. In particular, there follows from Lemma~\ref{addedinJanuary2015}
that all iterated commutators $[Z_i,\ldots [Z_i,[Z_i,\widetilde{Z}]] \ldots ]$ belong to
$\mathfrak{g}_{\rm step-3}$ provided that $\widetilde{Z} \in \mathfrak{g} (m-1)$.

Consider the Lie algebra $\mathfrak{g} (m-1)$ and recall that its dimension can be assumed greater
than or equal to~$2$. Again the
discussion will be split into two cases according to whether or not
$\mathfrak{g} (m-1)$ is abelian.

\noindent {\it The abelian case}: assume that $\mathfrak{g} (m-1)$ is abelian.

By assumption $\mathfrak{g} (m-1)$ is an abelian sub-algebra of $\mathfrak{g}_{\rm step-3}$ whose dimension
is at least~$2$. Owing to the description of abelian Lie algebras in Section~5.1, an abelian Lie algebra
having dimension at least~$2$  either is a linear span of two vector fields or consists of vector fields that are
everywhere parallel. When this Lie algebra is contained in $\mathfrak{g}_{\rm step-3}$, we obtain:

\noindent {\it Claim~1}. Without loss of generality we can assume that $\mathfrak{g} (m-1)$ either is spanned
by $X$ and $\ooY$ or it consists of vector fields having the form $hX$ where $h$ is a first integral of $X$.

\noindent {\it Proof of Claim~1}. Suppose first that $\mathfrak{g} (m-1)$ is a linear span of two vector fields
$(\varphi_1 f + \varphi_2) X + \alpha \ooY$
and $(\varphi_3 f + \varphi_4) X + \beta \ooY$. By taking a suitable linear combination of them, $\mathfrak{g} (m-1)$
is also spanned by $(\varphi_1 f + \varphi_2) X + \alpha \ooY$ and by $(\varphi_5 f + \varphi_6) X$. Now set
$(\varphi_1 f + \varphi_2) X + \alpha \ooY$ as your ``new vector field $\ooY$'' and $(\varphi_5 f + \varphi_6) X$
as the ``new vector field $X$''.

Suppose now that all vector fields in $\mathfrak{g} (m-1)$ are pairwise everywhere parallel. Since the
dimension of $\mathfrak{g} (m-1)$ is at least~$2$, there follows that these vector fields have to be everywhere parallel
to~$X$. Now it is clear that the quotient between two of these vector fields must be a first integral
of~$X$ so that the claim follows from choosing one of these vector fields as the ``new vector field $X$''.\qed

We begin with the case in which $\mathfrak{g} (m-1)$ can be identified with the linear span of the vector fields
$X$ and $\ooY$. The extra difficulty arising in the present case when compared to the proof of
Proposition~\ref{firtcasePropositioncommuting9-furtherstep} (Case~A) lies in the fact that the vector field
$X$ is no longer ``canonical''. Indeed, in the context of Proposition~\ref{firtcasePropositioncommuting9-furtherstep},
the corresponding vector field $X$ was naturally associated with the center of the Lie algebra $\mathfrak{g} (m,m-1)$
and hence was uniquely determined (up to multiplicative constants).
This no longer holds here since the center of $\mathfrak{g} (m-1)$ is supposed to be trivial.  Yet the argument
employed in the proof of Proposition~\ref{firtcasePropositioncommuting9-furtherstep} still holds here. Alternatively,
a slightly different argument is possible. This argument is summarized by Claim~2 below.

\noindent {\it Claim~2}. Without loss of generality we can assume that $[Z_i, X]$ is everywhere parallel to~$X$.

\noindent {\it Proof of Claim~2}. It is known that $[Z_i ,X]= a_{X,1,i} X + \alpha_{1,i} \ooY$ with $\alpha_{1,i} \in
\C$ (Lemma~\ref{addedinJanuary2015}). We assume that $\alpha_{1,i} \neq 0$ otherwise there is nothing to be proved.
Lemma~\ref{addedinJanuary2015} also ensures that $[Z_i,\ooY] = a_{\ooY,1,i} X + \beta_{1,i} \ooY$ with $\beta_{1,i} \in \C$.

Now consider the commutator $[Z_i,[Z_i,X]]$. This vector field belongs to $\mathfrak{g}_{\rm step-3}$
(owing again to Lemma~\ref{addedinJanuary2015}) and thus has the form $a_{X,2,i} X + \alpha_{2,i} \ooY$ with $\alpha_{2,i} \in \C$.
On the other hand, a direct computation yields
\begin{eqnarray*}
[Z_i,[Z_i,X]] & = & \left( \frac{\partial a_{X,1,i}}{\partial Z_i} \right) X + a_{X,1,i} [Z_i, X] +
\alpha_{1,i} [Z_i, \ooY] = \\
& =& \widetilde{a}_X X + (a_{X,1,i} \alpha_{1,i} + \beta_{1,i} \alpha_{1,i}) \ooY \, .
\end{eqnarray*}
Thus $\alpha_{2,i} = a_{X,1,i} \alpha_{1,i} + \beta_{1,i} \alpha_{1,i}$ so that $a_{X,1,i}$ is a constant since
$\alpha_{1,i} \neq 0$.

Similarly $[Z_i,[Z_i,\ooY]] = a_{\ooY,2,i} X + \beta_{2,i} \ooY$ has the form
$$
[Z_i,[Z_i,\ooY]] = \widetilde{a}_{\ooY} X + (a_{\ooY,1,i} \alpha_{1,i} + \beta_{1,i}^2) \ooY \, .
$$
Thus $a_{\ooY,1,i}$ is constant as well. Summarizing the preceding, every $Z_i$ induces an endomorphism of
the linear span of $X$ and $\ooY$. Therefore the Lie algebra generated by $X, \ooY, Z_1, \ldots ,Z_s$ is solvable
by virtue of Lemma~\ref{lastversionLemma3}.\qed

We assume in the sequel that $[Z_i,X]$ is everywhere parallel to~$X$ so that $Z_i = a_i X + b_i\ooY$
where $b_i$ is a first integral of~$X$. However, by
resorting to the argument given in the proof of Proposition~\ref{firtcasePropositioncommuting9-furtherstep},
we see that $[Z_i,\ooY]$ is everywhere parallel to~$X$ as well. In fact, all the iterated commutators
$[Z_i,\ldots [Z_i,[Z_i,\ooY]] \ldots ]$ belong to $\mathfrak{g}_{\rm step-3}$ while their orders at the origin
becomes arbitrarily large: therefore at some point they must become everywhere parallel to~$X$. This fact combined
with Claim~2 ensures that the commutator $[Z_i,\ooY]$ must be everywhere parallel to~$X$. In turn, this implies
that $b_i$ is a first integral of $Y$ as well so that we actually have
$Z_i =a_i X + \alpha_i \ooY$ with $\alpha_i \in \C$. To prove that the Lie algebra generated by
$X, \ooY, Z_1 ,\ldots ,Z_s$ still is solvable, there only remains to check that $a_i$ has the form
$a_i = \varphi_{1,i} f + \varphi_{2,i}$ ($\varphi_{1,i}, \, \varphi_{2,i}$ being first integrals of~$X$). This
however follows from the same argument employed in the proof of Lemma~\ref{newversionLemma33.5}. In other words,
the theorem is proved provided that $\mathfrak{g} (m-1)$ coincides with the linear span of two vector fields.

To complete the discussion of the case in which $\mathfrak{g} (m-1)$ is abelian, it remains to
consider the situation in which $\mathfrak{g} (m-1)$
consists of (two or more) vector fields having the form $hX$ where $h$ is a first integral of $X$.
The argument is essentially the same used in the analogous situation occurring in the proof
of Proposition~\ref{firtcasePropositioncommuting9-furtherstep}. We summarize the discussion in the sequel.
For every $i=1, \ldots ,s$, $F_i^{\ast} \mathfrak{g} (m-1)$ is again an abelian Lie all of whose vector fields
are pairwise everywhere parallel. Since $F_i^{\ast} \mathfrak{g} (m-1) \subseteq \mathfrak{g}_{\rm step-3}$,
we conclude that $F_i^{\ast} \mathfrak{g} (m-1)$ has again the form $h (\varphi_1 f + \varphi_2)X$ with
$\varphi_1, \varphi_2$ first integrals of $X$ (in particular $[Z_i,X]$ is everywhere parallel to~$X$).
As already seen, this also implies that $[Z_i ,\ooY]$ is everywhere parallel to~$X$. Assembling this information,
there follows that $Z_i$ has the form $Z_i = a_i X + \alpha_i \ooY$ with $\alpha_i \in \C$ and where
$a_i$ is a first integral of~$X$. Therefore the
Lie algebra generated by $Z_1, \ldots ,Z_s$ along with vector fields of the form
$hX$ ($h$ first integral of $X$) is solvable and this ends the proof
of Theorem~\ref{commuting9} in the case where $\mathfrak{g} (m-1)$ is an abelian algebra.

\noindent {\it The non-abelian case}. Suppose that $\mathfrak{g} (m-1)$ is not abelian.

The fundamental observation explaining why the non-abelian case discussed below is somewhat simpler
than the abelian one lies in the fact that the derived Lie algebra $D^1 \mathfrak{g}_{\rm step-3}$
consists of vector fields everywhere parallel to~$X$. To exploit this remark, we proceed as follows.

Since $\mathfrak{g} (m-1)$ is not abelian, its derived Lie algebra $D^1 \mathfrak{g} (m-1)$ is not trivial. Furthermore
$D^1 \mathfrak{g} (m-1)$ clearly consists of vector fields everywhere parallel to~$X$.
Fixed $i \in \{ 1, \ldots ,s \}$, consider then
the map from $\mathfrak{g} (m-1)$ to $\mathfrak{g} (m,m-1)$ consisting of taking the commutator with $Z_i$.
This map clearly sends $D^1 \mathfrak{g} (m-1)$ in $D^1 \mathfrak{g} (m,m-1)$ so that there is $aX \in D^1 \mathfrak{g} (m-1)$
such that $[Z_i,aX]$ is again everywhere parallel to~$X$. Thus we obtain once and for all that
$[Z_i,X]$ is everywhere parallel to~$X$. In turn, this implies that $[Z_i,\ooY]$ is everywhere parallel to~$X$ as well
(cf. Claim~2 in the proof of Proposition~\ref{firtcasePropositioncommuting9-furtherstep}). Thus
$Z_i a_i X + \alpha_i \ooY$ with $\alpha_i \in \C$. Once again to conclude that $a_i$ has the form
$a_i = \varphi_{1,i} f + \varphi_{2,i}$ ($\varphi_{1,i}, \, \varphi_{2,i}$ first integrals of~$X$) it suffices
to repeat the argument employed in the proof of Lemma~\ref{newversionLemma33.5}. The proof
of Theorem~\ref{commuting9} is completed.
\end{proof}

\vspace{0.1cm}

\begin{flushleft}
{\sc Julio Rebelo} \\
Institut de Math\'ematiques de Toulouse\\
118 Route de Narbonne\\
F-31062 Toulouse, FRANCE.\\
rebelo@math.univ-toulouse.fr

\end{flushleft}

\vspace{0.1cm}

\begin{flushleft}
{\sc Helena Reis} \\
Centro de Matem\'atica da Universidade do Porto, \\
Faculdade de Economia da Universidade do Porto, \\
Portugal\\
hreis@fep.up.pt \\

\end{flushleft}

\end{document}